\documentclass[leqno,a4paper,11pt]{amsart}
\usepackage[latin1]{inputenc}
\usepackage{amsmath}
\usepackage{amsfonts}
\usepackage{amssymb}
\usepackage{amsthm}
\usepackage[usenames, dvipsnames]{color}
\usepackage{tikz}
\usepackage[all]{xy}
\usepackage{geometry}

\usepackage[colorlinks=true, linkcolor=blue, 
hyperfootnotes=true,citecolor=blue,urlcolor=black]{hyperref}
\usepackage{marginnote}
\usepackage{pifont}

\definecolor{liens}{rgb}{1,0,0}

\newtheorem{thmintro}{Theorem}

\newtheorem{thm}{Theorem}[section]
\newtheorem{cor}[thm]{Corollary}
\newtheorem{lemma}[thm]{Lemma}
\newtheorem{prop}[thm]{Proposition}

\theoremstyle{definition}
\newtheorem{defi}[thm]{Definition} 

\theoremstyle{remark}
\newtheorem{rmk}[thm]{Remark}
\newtheorem{ex}[thm]{Example}

\numberwithin{equation}{section}
\def\iup{{\tilde{\iota}}}

\def\N{\mathbb{N}}

\def\Z{\mathbb{Z}}
\def\C{\mathbb{C}}
\def\R{\mathbb{R}}
\def\Q{\mathbb{Q}}

\def\P1{\mathbb{P}^{1}}

\def\beq{\begin{equation}}
\def\eeq{\end{equation}}

\def\P2{\mathbb{P}^{2}}
\def\cW{\mathcal{W}}
\def\a{\alpha}
\def\b{\beta}

\def\ZX{{\mathbb Z}}

\def\GL{{\rm GL}}

\def\P1{\mathbb{P}^{1}}

\def\cU{\mathcal{U}}

\def\s{\sigma}

\def\q{\mathbf{q}}
\def\cM{\mathcal{M}}
\def\P1{\bold{P}^1}

\def\cW{\mathcal{W}}

\begin{document}

\title[$t$-derivative transcendence of the generating function of walks]{Length derivative of the generating function of  walks confined in the quarter plane}
\author{Thomas Dreyfus}
\address{Institut de Recherche Math\'ematique Avanc\'ee, U.M.R. 7501 Universit\'e de Strasbourg et C.N.R.S. 7, rue Ren\'e Descartes 67084 Strasbourg, FRANCE}
\email{dreyfus@math.unistra.fr}
\author{Charlotte Hardouin}
\address{Universit\'e Paul Sabatier - Institut de Math\'ematiques de Toulouse, 118 route de Narbonne, 31062 Toulouse.}
\email{hardouin@math.univ-toulouse.fr}

\thanks{This project has been partially founded by ANR De rerum natura project (ANR-19-CE40-0018). The second author would like to thank the ANR-11-LABX-0040-CIMI within
the program ANR-11-IDEX-0002-0 for its partial support.}

\date{\today}

\bibliographystyle{amsalpha} 
\begin{abstract} 
In the present paper, we use difference Galois theory  to study the nature of the generating function  counting walks with small steps in the quarter plane. These series are trivariate formal power series
$Q(x,y,t)$ that count the number of walks  confined in the first quadrant of the plane with a fixed  set of admissible steps, called the model of the walk. While the variables $x$ and $y$ are associated to the ending point of the path, the variable $t$ encodes its length. In this paper, we prove 
that in the unweighted case, $Q(x,y,t)$ satisfies an algebraic differential relation with respect to $t$ if and only if it satisfies an algebraic differential relation with respect  $x$ (resp.~$y$). Combined   with \cite{BMM,BostanKauersTheCompleteGenerating,BBMR16,DHRS,DreyfusHardouinRoquesSingerGenuszero},   we  are able to characterize the $t$-differential transcendence of the $79$ models of walks  listed by Bousquet-M\'elou  and Mishna.
\end{abstract}

\maketitle
\setcounter{tocdepth}{1}
\tableofcontents 
\sloppy 
%%%%%%%%%%%%%%%%%%%%%%%%%%%%%%%%%%%%%%%%%%%%%%%%%%%%%%%%%%%%%%%%%%%%%%%%%%%%%%%%%%%%%%
%%%%%%%%%%%%%%%%%%%%%%%%%%%%%%%%%%%%%%%%%%%%%%%%%%%%%%%%%%%%%%%%%%%%%%%%%%%%%%%%%%%%%%
\section*{Introduction}

Classifying   lattice walks in restricted domains is an important problem in enumerative combinatorics. Recently much progress has been 
made in the study of walks with small steps in the quarter plane. A small steps model in the quarter plane $\Z_{\geq 0}\times \Z_{\geq 0}$ is composed by a set of admissible cardinal directions $\mathcal{D}\subset \{
\begin{tikzpicture}[scale=0.3]
;
\draw[thick,->](0,0)--(-1,0);
\end{tikzpicture}
,
\begin{tikzpicture}[scale=0.3]
;
\draw[thick,->](0,0)--(-1,1);
\end{tikzpicture}
 ,
 \begin{tikzpicture}[scale=0.3]
;
\draw[thick,->](0,0)--(0,1);
\end{tikzpicture}
 ,\begin{tikzpicture}[scale=0.3]
;
\draw[thick,->](0,0)--(1,1);
\end{tikzpicture},
\begin{tikzpicture}[scale=0.3]
;
\draw[thick,->](0,0)--(1,0);
\end{tikzpicture}
,
\begin{tikzpicture}[scale=0.3]
;
\draw[thick,->](0,0)--(1,-1);
\end{tikzpicture}
,
\begin{tikzpicture}[scale=0.3]
;
\draw[thick,->](0,0)--(0,-1);
\end{tikzpicture}
,
\begin{tikzpicture}[scale=0.3]
;
\draw[thick,->](0,0)--(-1,-1);
\end{tikzpicture} \}$. Given $\mathcal{D}$, we consider the walks that start at $(0,0)$, with directions in $\mathcal{D}$, and that stay in the quarter plane, for instance:
\begin{center}
\begin{tikzpicture}[scale=.8, baseline=(current bounding box.center)]
\foreach \x in {0,1,2,3,4,5,6,7,8,9,10}
  \foreach \y in {0,1,2,3,4}
    \fill(\x,\y) circle[radius=0pt];
\draw (0,0)--(10,0);
\draw (0,0)--(0,4);
\draw[->](0,0)--(1,1);
\draw[->](1,1)--(1,0);
\draw[->](1,0)--(0,1);
\draw[->](0,1)--(1,2);
\draw[->](1,2)--(2,1);
\draw[->](2,1)--(2,0);
\draw[->](2,0)--(3,1);
\draw[->](3,1)--(3,0);
\draw[->](3,0)--(4,1);
\draw[->](4,1)--(3,2);
\draw[->](3,2)--(2,3);
\draw[->](2,3)--(2,2);
\draw[->](2,2)--(3,3);
\draw[->](3,3)--(4,2);
\draw[->](4,2)--(4,1);
\draw[->](4,1)--(5,0);
\draw[->](5,0)--(6,1);
\draw[->](6,1)--(6,0);
\draw[->](6,0)--(7,1);
\draw[->](7,1)--(8,0);
\draw[->](8,0)--(9,1);
\draw[->](9,1)--(9,0);
\draw[->](9,0)--(10,1);
\draw[->](10,1)--(9,2);
\draw[->](9,2)--(8,3);
\draw[->](8,3)--(8,2);
\end{tikzpicture}
\quad\quad
$\mathcal{D}=\left\{\begin{tikzpicture}[scale=.4, baseline=(current bounding box.center)]
\foreach \x in {-1,0,1} \foreach \y in {-1,0,1} \fill(\x,\y) circle[radius=2pt];
\draw[thick,->](0,0)--(-1,1);
\draw[thick,->](0,0)--(1,1);
\draw[thick,->](0,0)--(1,-1);
\draw[thick,->](0,0)--(0,-1);
\end{tikzpicture}\right\}$
\end{center}

 For a given model, one defines $q_{i,j,k}$ to be the number of walks confined to the first quadrant of the plane that 
begin at $(0,0)$ and end at $(i,j)$ in $k$ admissible steps.  The algebraic nature of the  associated complete generating function $Q(x,y,t)=\sum_{i,j,k=0}^{\infty}q_{i,j,k}x^{i}y^{j}t^{k}$ captures many important combinatorial properties of the model: symmetries,  asymptotic information, and  recursive relations of the coefficients.

Among the $2^8-1=255$ models  in the first quadrant of the plane, Bousquet-M\'elou  and Mishna proved in \cite{BMM} that, after accounting for symmetries and eliminating the trivial and one dimensional cases, only $79$ cases remained.  It is worth mentioning that the generating function is algebraic in all the trivial and one dimensional cases.

For any choice of a variable  $\star$ among $x,y,t$, one classifies the algebraic nature of the generating series $Q(x,y,t)$ with respect to $\star$
as follows:
\begin{itemize}
\item  \emph{Algebraic cases:} the series $Q(x,y,t)$  satisfies a nontrivial polynomial relation with coefficients in $\Q(x,y,t)$. 
\item \emph{Transcendental $\star$-holonomic cases:} the series $Q(x,y,t)$ is transcendental and holonomic with respect to $\star$, i.e. there exists $n\in \Z_{\geq 0}$, such that  there exist $a_{0},\dots,a_{n}\in \Q (x,y,t)$, not 	all zero, such that 
$$
0 =\displaystyle \sum_{\ell=0}^{n}a_{\ell}\frac{d}{d \star}^{\ell}Q(x,y,t).
$$
\item \emph{Nonholonomic $\frac{d}{d\star}$-differentially algebraic cases:} the series $Q(x,y,t)$ is nonholonomic and $\frac{d}{d \star}$-differentially algebraic, i.e. 
there exists $n\in \Z_{\geq 0}$, such that there exists nonzero multivariate polynomial $P_{\star}\in  \Q(x,y,t)[X_{0},\dots,X_{n}]$, such that 
$$
0=P_{\star}(Q(x,y,t),\dots,\frac{d}{d \star}^{n}Q(x,y,t)).$$
   We stress out the fact that in the above definition, it is equivalent to  require that  $P_{\star}\in  \Q[X_{0},\dots,X_{n}]$, see Remark \ref{rmk:CdifftransQdifftrans}.
\item \emph{$\frac{d}{d \star}$-differentially transcendental cases:} the series is  not $\frac{d}{d \star}$-differentially algebraic.
\end{itemize}

The authors of \cite{BMM,BostanKauersTheCompleteGenerating,BBMR16,DHRS,DreyfusHardouinRoquesSingerGenuszero}   proved that the  algebraic nature of the generating series was identical for the variables $x$ and $y$. The classification of the  models of walks regarding the algebraic nature of their series with respect to the variables $x$ and $y$  is the culmination of ten years of research and the works of many researchers  (see Figure \ref{figcas} below).

\pagebreak

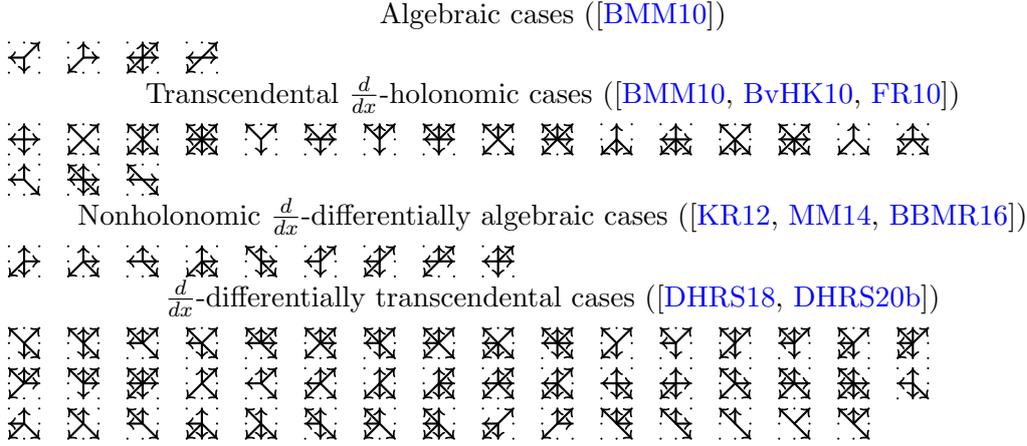
\begin{figure}[!h]
\begin{trivlist}
\item  \begin{center}
Algebraic cases (\cite{BMM})
\end{center}
$\begin{array}{llll}
\begin{tikzpicture}[scale=.2, baseline=(current bounding box.center)]
\foreach \x in {-1,0,1} \foreach \y in {-1,0,1} \fill(\x,\y) circle[radius=2pt];
%\draw[thick,->](0,0)--(-1,1);
%\draw[thick,->](0,0)--(0,1);
\draw[thick,->](0,0)--(1,1);
\draw[thick,->](0,0)--(-1,0);
%\draw[thick,->](0,0)--(1,0);
%\draw[thick,->](0,0)--(-1,-1);
\draw[thick,->](0,0)--(0,-1);
%\draw[thick,->](0,0)--(1,-1);
\end{tikzpicture}
& 
\begin{tikzpicture}[scale=.2, baseline=(current bounding box.center)]
\foreach \x in {-1,0,1} \foreach \y in {-1,0,1} \fill(\x,\y) circle[radius=2pt];
%\draw[thick,->](0,0)--(-1,1);
\draw[thick,->](0,0)--(0,1);
%\draw[thick,->](0,0)--(1,1);
%\draw[thick,->](0,0)--(-1,0);
\draw[thick,->](0,0)--(1,0);
\draw[thick,->](0,0)--(-1,-1);
%\draw[thick,->](0,0)--(0,-1);
%\draw[thick,->](0,0)--(1,-1);
\end{tikzpicture} 
&
\begin{tikzpicture}[scale=.2, baseline=(current bounding box.center)]
\foreach \x in {-1,0,1} \foreach \y in {-1,0,1} \fill(\x,\y) circle[radius=2pt];
%\draw[thick,->](0,0)--(-1,1);
\draw[thick,->](0,0)--(0,1);
\draw[thick,->](0,0)--(1,1);
\draw[thick,->](0,0)--(-1,0);
\draw[thick,->](0,0)--(1,0);
\draw[thick,->](0,0)--(-1,-1);
\draw[thick,->](0,0)--(0,-1);
%\draw[thick,->](0,0)--(1,-1);
\end{tikzpicture}
&
\begin{tikzpicture}[scale=.2, baseline=(current bounding box.center)]
\foreach \x in {-1,0,1} \foreach \y in {-1,0,1} \fill(\x,\y) circle[radius=2pt];
\draw[thick,->](0,0)--(1,1);
\draw[thick,->](0,0)--(-1,-1);
\draw[thick,->](0,0)--(1,0);
\draw[thick,->](0,0)--(-1,0);
\end{tikzpicture}
\end{array}$

\item \begin{center} Transcendental $\frac{d}{dx}$-holonomic cases (\cite{BMM,BostanKauersTheCompleteGenerating,FR10})\end{center}

$\begin{array}{llllllllllllllll}
\begin{tikzpicture}[scale=.2, baseline=(current bounding box.center)]
\foreach \x in {-1,0,1} \foreach \y in {-1,0,1} \fill(\x,\y) circle[radius=2pt];
%\draw[thick,->](0,0)--(-1,1);
\draw[thick,->](0,0)--(0,1);
%\draw[thick,->](0,0)--(1,1);
\draw[thick,->](0,0)--(-1,0);
\draw[thick,->](0,0)--(1,0);
%\draw[thick,->](0,0)--(-1,-1);
\draw[thick,->](0,0)--(0,-1);
%\draw[thick,->](0,0)--(1,-1);
\end{tikzpicture}
&
\begin{tikzpicture}[scale=.2, baseline=(current bounding box.center)]
\foreach \x in {-1,0,1} \foreach \y in {-1,0,1} \fill(\x,\y) circle[radius=2pt];
\draw[thick,->](0,0)--(-1,1);
%\draw[thick,->](0,0)--(0,1);
\draw[thick,->](0,0)--(1,1);
%\draw[thick,->](0,0)--(-1,0);
%\draw[thick,->](0,0)--(1,0);
\draw[thick,->](0,0)--(-1,-1);
%\draw[thick,->](0,0)--(0,-1);
\draw[thick,->](0,0)--(1,-1);
\end{tikzpicture}
&
\begin{tikzpicture}[scale=.2, baseline=(current bounding box.center)]
\foreach \x in {-1,0,1} \foreach \y in {-1,0,1} \fill(\x,\y) circle[radius=2pt];
\draw[thick,->](0,0)--(-1,1);
\draw[thick,->](0,0)--(0,1);
\draw[thick,->](0,0)--(1,1);
%\draw[thick,->](0,0)--(-1,0);
%\draw[thick,->](0,0)--(1,0);
\draw[thick,->](0,0)--(-1,-1);
\draw[thick,->](0,0)--(0,-1);
\draw[thick,->](0,0)--(1,-1);
\end{tikzpicture}
&
\begin{tikzpicture}[scale=.2, baseline=(current bounding box.center)]
\foreach \x in {-1,0,1} \foreach \y in {-1,0,1} \fill(\x,\y) circle[radius=2pt];
\draw[thick,->](0,0)--(-1,1);
\draw[thick,->](0,0)--(0,1);
\draw[thick,->](0,0)--(1,1);
\draw[thick,->](0,0)--(-1,0);
\draw[thick,->](0,0)--(1,0);
\draw[thick,->](0,0)--(-1,-1);
\draw[thick,->](0,0)--(0,-1);
\draw[thick,->](0,0)--(1,-1);
\end{tikzpicture}
&
\begin{tikzpicture}[scale=.2, baseline=(current bounding box.center)]
\foreach \x in {-1,0,1} \foreach \y in {-1,0,1} \fill(\x,\y) circle[radius=2pt];
\draw[thick,->](0,0)--(-1,1);
%\draw[thick,->](0,0)--(0,1);
\draw[thick,->](0,0)--(1,1);
%\draw[thick,->](0,0)--(-1,0);
%\draw[thick,->](0,0)--(1,0);
%\draw[thick,->](0,0)--(-1,-1);
\draw[thick,->](0,0)--(0,-1);
%\draw[thick,->](0,0)--(1,-1);
\end{tikzpicture}
&
\begin{tikzpicture}[scale=.2, baseline=(current bounding box.center)]
\foreach \x in {-1,0,1} \foreach \y in {-1,0,1} \fill(\x,\y) circle[radius=2pt];
\draw[thick,->](0,0)--(-1,1);
%\draw[thick,->](0,0)--(0,1);
\draw[thick,->](0,0)--(1,1);
\draw[thick,->](0,0)--(-1,0);
\draw[thick,->](0,0)--(1,0);
%\draw[thick,->](0,0)--(-1,-1);
\draw[thick,->](0,0)--(0,-1);
%\draw[thick,->](0,0)--(1,-1);
\end{tikzpicture}
&
\begin{tikzpicture}[scale=.2, baseline=(current bounding box.center)]
\foreach \x in {-1,0,1} \foreach \y in {-1,0,1} \fill(\x,\y) circle[radius=2pt];
\draw[thick,->](0,0)--(-1,1);
\draw[thick,->](0,0)--(0,1);
\draw[thick,->](0,0)--(1,1);
%\draw[thick,->](0,0)--(-1,0);
%\draw[thick,->](0,0)--(1,0);
%\draw[thick,->](0,0)--(-1,-1);
\draw[thick,->](0,0)--(0,-1);
%\draw[thick,->](0,0)--(1,-1);
\end{tikzpicture}
&
\begin{tikzpicture}[scale=.2, baseline=(current bounding box.center)]
\foreach \x in {-1,0,1} \foreach \y in {-1,0,1} \fill(\x,\y) circle[radius=2pt];
\draw[thick,->](0,0)--(-1,1);
\draw[thick,->](0,0)--(0,1);
\draw[thick,->](0,0)--(1,1);
\draw[thick,->](0,0)--(-1,0);
\draw[thick,->](0,0)--(1,0);
%\draw[thick,->](0,0)--(-1,-1);
\draw[thick,->](0,0)--(0,-1);
%\draw[thick,->](0,0)--(1,-1);
\end{tikzpicture}
&
\begin{tikzpicture}[scale=.2, baseline=(current bounding box.center)]
\foreach \x in {-1,0,1} \foreach \y in {-1,0,1} \fill(\x,\y) circle[radius=2pt];
\draw[thick,->](0,0)--(-1,1);
\draw[thick,->](0,0)--(0,1);
\draw[thick,->](0,0)--(1,1);
%\draw[thick,->](0,0)--(-1,0);
%\draw[thick,->](0,0)--(1,0);
\draw[thick,->](0,0)--(-1,-1);
%\draw[thick,->](0,0)--(0,-1);
\draw[thick,->](0,0)--(1,-1);
\end{tikzpicture}
& 
\begin{tikzpicture}[scale=.2, baseline=(current bounding box.center)]
\foreach \x in {-1,0,1} \foreach \y in {-1,0,1} \fill(\x,\y) circle[radius=2pt];
\draw[thick,->](0,0)--(-1,1);
\draw[thick,->](0,0)--(0,1);
\draw[thick,->](0,0)--(1,1);
\draw[thick,->](0,0)--(-1,0);
\draw[thick,->](0,0)--(1,0);
\draw[thick,->](0,0)--(-1,-1);
%\draw[thick,->](0,0)--(0,-1);
\draw[thick,->](0,0)--(1,-1);
\end{tikzpicture}
&
\begin{tikzpicture}[scale=.2, baseline=(current bounding box.center)]
\foreach \x in {-1,0,1} \foreach \y in {-1,0,1} \fill(\x,\y) circle[radius=2pt];
%\draw[thick,->](0,0)--(-1,1);
\draw[thick,->](0,0)--(0,1);
%\draw[thick,->](0,0)--(1,1);
%\draw[thick,->](0,0)--(-1,0);
%\draw[thick,->](0,0)--(1,0);
\draw[thick,->](0,0)--(-1,-1);
\draw[thick,->](0,0)--(0,-1);
\draw[thick,->](0,0)--(1,-1);
\end{tikzpicture}
&
\begin{tikzpicture}[scale=.2, baseline=(current bounding box.center)]
\foreach \x in {-1,0,1} \foreach \y in {-1,0,1} \fill(\x,\y) circle[radius=2pt];
%\draw[thick,->](0,0)--(-1,1);
\draw[thick,->](0,0)--(0,1);
%\draw[thick,->](0,0)--(1,1);
\draw[thick,->](0,0)--(-1,0);
\draw[thick,->](0,0)--(1,0);
\draw[thick,->](0,0)--(-1,-1);red
\draw[thick,->](0,0)--(0,-1);
\draw[thick,->](0,0)--(1,-1);
\end{tikzpicture}
&
\begin{tikzpicture}[scale=.2, baseline=(current bounding box.center)]
\foreach \x in {-1,0,1} \foreach \y in {-1,0,1} \fill(\x,\y) circle[radius=2pt];
\draw[thick,->](0,0)--(-1,1);
%\draw[thick,->](0,0)--(0,1);
\draw[thick,->](0,0)--(1,1);
%\draw[thick,->](0,0)--(-1,0);
%\draw[thick,->](0,0)--(1,0);
\draw[thick,->](0,0)--(-1,-1);
\draw[thick,->](0,0)--(0,-1);
\draw[thick,->](0,0)--(1,-1);
\end{tikzpicture}
&
\begin{tikzpicture}[scale=.2, baseline=(current bounding box.center)]
\foreach \x in {-1,0,1} \foreach \y in {-1,0,1} \fill(\x,\y) circle[radius=2pt];
\draw[thick,->](0,0)--(-1,1);
%\draw[thick,->](0,0)--(0,1);
\draw[thick,->](0,0)--(1,1);
\draw[thick,->](0,0)--(-1,0);
\draw[thick,->](0,0)--(1,0);
\draw[thick,->](0,0)--(-1,-1);
\draw[thick,->](0,0)--(0,-1);
\draw[thick,->](0,0)--(1,-1);
\end{tikzpicture}
&
\begin{tikzpicture}[scale=.2, baseline=(current bounding box.center)]
\foreach \x in {-1,0,1} \foreach \y in {-1,0,1} \fill(\x,\y) circle[radius=2pt];
\draw[thick,->](0,0)--(0,1);
\draw[thick,->](0,0)--(1,-1);
\draw[thick,->](0,0)--(-1,-1);
\end{tikzpicture} 
&
\begin{tikzpicture}[scale=.2, baseline=(current bounding box.center)]
\foreach \x in {-1,0,1} \foreach \y in {-1,0,1} \fill(\x,\y) circle[radius=2pt];
\draw[thick,->](0,0)--(0,1);
\draw[thick,->](0,0)--(1,-1);
\draw[thick,->](0,0)--(-1,-1);
\draw[thick,->](0,0)--(1,0);
\draw[thick,->](0,0)--(-1,0);
\end{tikzpicture}\\
\begin{tikzpicture}[scale=.2, baseline=(current bounding box.center)]
\foreach \x in {-1,0,1} \foreach \y in {-1,0,1} \fill(\x,\y) circle[radius=2pt];
%\draw[thick,->](0,0)--(-1,1);
\draw[thick,->](0,0)--(0,1);
%\draw[thick,->](0,0)--(1,1);
\draw[thick,->](0,0)--(-1,0);
%\draw[thick,->](0,0)--(1,0);
%\draw[thick,->](0,0)--(-1,-1);
%\draw[thick,->](0,0)--(0,-1);
\draw[thick,->](0,0)--(1,-1);
\end{tikzpicture}
&
\begin{tikzpicture}[scale=.2, baseline=(current bounding box.center)]
\foreach \x in {-1,0,1} \foreach \y in {-1,0,1} \fill(\x,\y) circle[radius=2pt];
\draw[thick,->](0,0)--(-1,1);
\draw[thick,->](0,0)--(0,1);
%\draw[thick,->](0,0)--(1,1);
\draw[thick,->](0,0)--(-1,0);
\draw[thick,->](0,0)--(1,0);
%\draw[thick,->](0,0)--(-1,-1);
\draw[thick,->](0,0)--(0,-1);
\draw[thick,->](0,0)--(1,-1);
\end{tikzpicture}&
\begin{tikzpicture}[scale=.2, baseline=(current bounding box.center)]
\foreach \x in {-1,0,1} \foreach \y in {-1,0,1} \fill(\x,\y) circle[radius=2pt];
\draw[thick,->](0,0)--(-1,1);
\draw[thick,->](0,0)--(1,-1);
\draw[thick,->](0,0)--(1,0);
\draw[thick,->](0,0)--(-1,0);
\end{tikzpicture}
&&&&&&&&&&&&
\end{array}$
\item
\begin{center} Nonholonomic $\frac{d}{dx}$-differentially algebraic cases (\cite{KurkRasch, MelcMish,BBMR16}) \end{center}
$\begin{array}{lllllllll}
\begin{tikzpicture}[scale=.2, baseline=(current bounding box.center)]
\foreach \x in {-1,0,1} \foreach \y in {-1,0,1} \fill(\x,\y) circle[radius=2pt];
%\draw[thick,->](0,0)--(-1,1);
\draw[thick,->](0,0)--(0,1);
%\draw[thick,->](0,0)--(1,1);
%\draw[thick,->](0,0)--(-1,0);
\draw[thick,->](0,0)--(1,0);
\draw[thick,->](0,0)--(-1,-1);
\draw[thick,->](0,0)--(0,-1);
%\draw[thick,->](0,0)--(1,-1);
\end{tikzpicture}
&
\begin{tikzpicture}[scale=.2, baseline=(current bounding box.center)]
\foreach \x in {-1,0,1} \foreach \y in {-1,0,1} \fill(\x,\y) circle[radius=2pt];
%\draw[thick,->](0,0)--(-1,1);
\draw[thick,->](0,0)--(0,1);
%\draw[thick,->](0,0)--(1,1);
%\draw[thick,->](0,0)--(-1,0);
\draw[thick,->](0,0)--(1,0);
\draw[thick,->](0,0)--(-1,-1);
%\draw[thick,->](0,0)--(0,-1);
\draw[thick,->](0,0)--(1,-1);
\end{tikzpicture}
&
\begin{tikzpicture}[scale=.2, baseline=(current bounding box.center)]
\foreach \x in {-1,0,1} \foreach \y in {-1,0,1} \fill(\x,\y) circle[radius=2pt];
%\draw[thick,->](0,0)--(-1,1);
\draw[thick,->](0,0)--(0,1);
%\draw[thick,->](0,0)--(1,1);
\draw[thick,->](0,0)--(-1,0);
\draw[thick,->](0,0)--(1,0);
%\draw[thick,->](0,0)--(-1,-1);
%\draw[thick,->](0,0)--(0,-1);
\draw[thick,->](0,0)--(1,-1);
\end{tikzpicture}
&
\begin{tikzpicture}[scale=.2, baseline=(current bounding box.center)]
\foreach \x in {-1,0,1} \foreach \y in {-1,0,1} \fill(\x,\y) circle[radius=2pt];
%\draw[thick,->](0,0)--(-1,1);
\draw[thick,->](0,0)--(0,1);
%\draw[thick,->](0,0)--(1,1);
%\draw[thick,->](0,0)--(-1,0);
\draw[thick,->](0,0)--(1,0);
\draw[thick,->](0,0)--(-1,-1);
\draw[thick,->](0,0)--(0,-1);
\draw[thick,->](0,0)--(1,-1);
\end{tikzpicture}
&
\begin{tikzpicture}[scale=.2, baseline=(current bounding box.center)]
\foreach \x in {-1,0,1} \foreach \y in {-1,0,1} \fill(\x,\y) circle[radius=2pt];
\draw[thick,->](0,0)--(-1,1);
\draw[thick,->](0,0)--(0,1);
%\draw[thick,->](0,0)--(1,1);
%\draw[thick,->](0,0)--(-1,0);
\draw[thick,->](0,0)--(1,0);
%\draw[thick,->](0,0)--(-1,-1);
\draw[thick,->](0,0)--(0,-1);
\draw[thick,->](0,0)--(1,-1);
\end{tikzpicture}
&
\begin{tikzpicture}[scale=.2, baseline=(current bounding box.center)]
\foreach \x in {-1,0,1} \foreach \y in {-1,0,1} \fill(\x,\y) circle[radius=2pt];
%\draw[thick,->](0,0)--(-1,1);
\draw[thick,->](0,0)--(0,1);
\draw[thick,->](0,0)--(1,1);
\draw[thick,->](0,0)--(-1,0);
%\draw[thick,->](0,0)--(1,0);
%\draw[thick,->](0,0)--(-1,-1);
\draw[thick,->](0,0)--(0,-1);
%\draw[thick,->](0,0)--(1,-1);
\end{tikzpicture}
&
\begin{tikzpicture}[scale=.2, baseline=(current bounding box.center)]
\foreach \x in {-1,0,1} \foreach \y in {-1,0,1} \fill(\x,\y) circle[radius=2pt];
%\draw[thick,->](0,0)--(-1,1);
\draw[thick,->](0,0)--(0,1);
\draw[thick,->](0,0)--(1,1);
\draw[thick,->](0,0)--(-1,0);
%\draw[thick,->](0,0)--(1,0);
\draw[thick,->](0,0)--(-1,-1);
\draw[thick,->](0,0)--(0,-1);
%\draw[thick,->](0,0)--(1,-1);
\end{tikzpicture}
&
\begin{tikzpicture}[scale=.2, baseline=(current bounding box.center)]
\foreach \x in {-1,0,1} \foreach \y in {-1,0,1} \fill(\x,\y) circle[radius=2pt];
%\draw[thick,->](0,0)--(-1,1);
\draw[thick,->](0,0)--(0,1);
\draw[thick,->](0,0)--(1,1);
\draw[thick,->](0,0)--(-1,0);
\draw[thick,->](0,0)--(1,0);
\draw[thick,->](0,0)--(-1,-1);
%\draw[thick,->](0,0)--(0,-1);
%\draw[thick,->](0,0)--(1,-1);
\end{tikzpicture} 
&
\begin{tikzpicture}[scale=.2, baseline=(current bounding box.center)]
\foreach \x in {-1,0,1} \foreach \y in {-1,0,1} \fill(\x,\y) circle[radius=2pt];
%\draw[thick,->](0,0)--(-1,1);
\draw[thick,->](0,0)--(0,1);
\draw[thick,->](0,0)--(1,1);
\draw[thick,->](0,0)--(-1,0);
\draw[thick,->](0,0)--(1,0);
%\draw[thick,->](0,0)--(-1,-1);
\draw[thick,->](0,0)--(0,-1);
%\draw[thick,->](0,0)--(1,-1);
\end{tikzpicture}\end{array}$

\item
\begin{center} $\frac{d}{dx}$-differentially transcendental cases (\cite{DHRS,DreyfusHardouinRoquesSingerGenuszero})\end{center}
$\begin{array}{llllllllllllllll}
\begin{tikzpicture}[scale=.2, baseline=(current bounding box.center)]
\foreach \x in {-1,0,1} \foreach \y in {-1,0,1} \fill(\x,\y) circle[radius=2pt];
\draw[thick,->](0,0)--(-1,1);
\draw[thick,->](0,0)--(1,1);
\draw[thick,->](0,0)--(0,-1);
\draw[thick,->](0,0)--(1,-1);
\end{tikzpicture}
& 
\begin{tikzpicture}[scale=.2, baseline=(current bounding box.center)]
\foreach \x in {-1,0,1} \foreach \y in {-1,0,1} \fill(\x,\y) circle[radius=2pt];
\draw[thick,->](0,0)--(-1,1);
\draw[thick,->](0,0)--(0,1);
\draw[thick,->](0,0)--(1,1);
\draw[thick,->](0,0)--(0,-1);
\draw[thick,->](0,0)--(1,-1);
\end{tikzpicture}
&
\begin{tikzpicture}[scale=.2, baseline=(current bounding box.center)]
\foreach \x in {-1,0,1} \foreach \y in {-1,0,1} \fill(\x,\y) circle[radius=2pt];
\draw[thick,->](0,0)--(-1,1);
\draw[thick,->](0,0)--(0,1);
\draw[thick,->](0,0)--(1,1);
\draw[thick,->](0,0)--(-1,0);
%\draw[thick,->](0,0)--(1,0);
%\draw[thick,->](0,0)--(-1,-1);
%\draw[thick,->](0,0)--(0,-1);
\draw[thick,->](0,0)--(1,-1);
\end{tikzpicture}
&
\begin{tikzpicture}[scale=.2, baseline=(current bounding box.center)]
\foreach \x in {-1,0,1} \foreach \y in {-1,0,1} \fill(\x,\y) circle[radius=2pt];
\draw[thick,->](0,0)--(-1,1);
%\draw[thick,->](0,0)--(0,1);
\draw[thick,->](0,0)--(1,1);
\draw[thick,->](0,0)--(-1,0);
%\draw[thick,->](0,0)--(1,0);
%\draw[thick,->](0,0)--(-1,-1);
\draw[thick,->](0,0)--(0,-1);
\draw[thick,->](0,0)--(1,-1);
\end{tikzpicture}
&
\begin{tikzpicture}[scale=.2, baseline=(current bounding box.center)]
\foreach \x in {-1,0,1} \foreach \y in {-1,0,1} \fill(\x,\y) circle[radius=2pt];
\draw[thick,->](0,0)--(-1,1);
\draw[thick,->](0,0)--(0,1);
\draw[thick,->](0,0)--(1,1);
\draw[thick,->](0,0)--(-1,0);
\draw[thick,->](0,0)--(1,0);
%\draw[thick,->](0,0)--(-1,-1);
%\draw[thick,->](0,0)--(0,-1);
\draw[thick,->](0,0)--(1,-1);
\end{tikzpicture}
&
\begin{tikzpicture}[scale=.2, baseline=(current bounding box.center)]
\foreach \x in {-1,0,1} \foreach \y in {-1,0,1} \fill(\x,\y) circle[radius=2pt];
\draw[thick,->](0,0)--(-1,1);
\draw[thick,->](0,0)--(0,1);
\draw[thick,->](0,0)--(1,1);
%\drawend{trivlist}[thick,->](0,0)--(-1,0);
\draw[thick,->](0,0)--(1,0);
\draw[thick,->](0,0)--(-1,-1);
%\draw[thick,->](0,0)--(0,-1);
\draw[thick,->](0,0)--(1,-1);
\end{tikzpicture}
&
\begin{tikzpicture}[scale=.2, baseline=(current bounding box.center)]
\foreach \x in {-1,0,1} \foreach \y in {-1,0,1} \fill(\x,\y) circle[radius=2pt];
\draw[thick,->](0,0)--(-1,1);
\draw[thick,->](0,0)--(0,1);
\draw[thick,->](0,0)--(1,1);
\draw[thick,->](0,0)--(-1,0);
%\draw[thick,->](0,0)--(1,0);
%\draw[thick,->](0,0)--(-1,-1);
\draw[thick,->](0,0)--(0,-1);
\draw[thick,->](0,0)--(1,-1);
\end{tikzpicture}
&
\begin{tikzpicture}[scale=.2, baseline=(current bounding box.center)]
\foreach \x in {-1,0,1} \foreach \y in {-1,0,1} \fill(\x,\y) circle[radius=2pt];
\draw[thick,->](0,0)--(-1,1);
\draw[thick,->](0,0)--(0,1);
\draw[thick,->](0,0)--(1,1);
\draw[thick,->](0,0)--(-1,0);
%\draw[thick,->](0,0)--(1,0);
\draw[thick,->](0,0)--(-1,-1);
%\draw[thick,->](0,0)--(0,-1);
\draw[thick,->](0,0)--(1,-1);
\end{tikzpicture}
&
\begin{tikzpicture}[scale=.2, baseline=(current bounding box.center)]
\foreach \x in {-1,0,1} \foreach \y in {-1,0,1} \fill(\x,\y) circle[radius=2pt];
\draw[thick,->](0,0)--(-1,1);
%\draw[thick,->](0,0)--(0,1);
\draw[thick,->](0,0)--(1,1);
\draw[thick,->](0,0)--(-1,0);
%\draw[thick,->](0,0)--(1,0);
\draw[thick,->](0,0)--(-1,-1);
\draw[thick,->](0,0)--(0,-1);
\draw[thick,->](0,0)--(1,-1);
\end{tikzpicture}
&
\begin{tikzpicture}[scale=.2, baseline=(current bounding box.center)]
\foreach \x in {-1,0,1} \foreach \y in {-1,0,1} \fill(\x,\y) circle[radius=2pt];
\draw[thick,->](0,0)--(-1,1);
\draw[thick,->](0,0)--(0,1);
\draw[thick,->](0,0)--(1,1);
\draw[thick,->](0,0)--(-1,0);
\draw[thick,->](0,0)--(1,0);
%\draw[thick,->](0,0)--(-1,-1);
\draw[thick,->](0,0)--(0,-1);
\draw[thick,->](0,0)--(1,-1);
\end{tikzpicture}
&
\begin{tikzpicture}[scale=.2, baseline=(current bounding box.center)]
\foreach \x in {-1,0,1} \foreach \y in {-1,0,1} \fill(\x,\y) circle[radius=2pt];
\draw[thick,->](0,0)--(-1,1);
%\draw[thick,->](0,0)--(0,1);
\draw[thick,->](0,0)--(1,1);
%\draw[thick,->](0,0)--(-1,0);
%\draw[thick,->](0,0)--(1,0);
\draw[thick,->](0,0)--(-1,-1);
\draw[thick,->](0,0)--(0,-1);
%\draw[thick,->](0,0)--(1,-1);
\end{tikzpicture}
&
\begin{tikzpicture}[scale=.2, baseline=(current bounding box.center)]
\foreach \x in {-1,0,1} \foreach \y in {-1,0,1} \fill(\x,\y) circle[radius=2pt];
\draw[thick,->](0,0)--(-1,1);
%\draw[thick,->](0,0)--(0,1);
\draw[thick,->](0,0)--(1,1);
\draw[thick,->](0,0)--(-1,0);
%\draw[thick,->](0,0)--(1,0);
%\draw[thick,->](0,0)--(-1,-1);
\draw[thick,->](0,0)--(0,-1);
%\draw[thick,->](0,0)--(1,-1);
\end{tikzpicture}
&
\begin{tikzpicture}[scale=.2, baseline=(current bounding box.center)]
\foreach \x in {-1,0,1} \foreach \y in {-1,0,1} \fill(\x,\y) circle[radius=2pt];
\draw[thick,->](0,0)--(-1,1);
\draw[thick,->](0,0)--(0,1);
\draw[thick,->](0,0)--(1,1);
%\draw[thick,->](0,0)--(-1,0);
%\draw[thick,->](0,0)--(1,0);
\draw[thick,->](0,0)--(-1,-1);
\draw[thick,->](0,0)--(0,-1);
%\draw[thick,->](0,0)--(1,-1);
\end{tikzpicture}
&
\begin{tikzpicture}[scale=.2, baseline=(current bounding box.center)]
\foreach \x in {-1,0,1} \foreach \y in {-1,0,1} \fill(\x,\y) circle[radius=2pt];
\draw[thick,->](0,0)--(-1,1);
\draw[thick,->](0,0)--(0,1);
\draw[thick,->](0,0)--(1,1);
\draw[thick,->](0,0)--(-1,0);
%\draw[thick,->](0,0)--(1,0);
%\draw[thick,->](0,0)--(-1,-1);
\draw[thick,->](0,0)--(0,-1);
%\draw[thick,->](0,0)--(1,-1);
\end{tikzpicture}
&
\begin{tikzpicture}[scale=.2, baseline=(current bounding box.center)]
\foreach \x in {-1,0,1} \foreach \y in {-1,0,1} \fill(\x,\y) circle[radius=2pt];
\draw[thick,->](0,0)--(-1,1);
%\draw[thick,->](0,0)--(0,1);
\draw[thick,->](0,0)--(1,1);
\draw[thick,->](0,0)--(-1,0);
%\draw[thick,->](0,0)--(1,0);
\draw[thick,->](0,0)--(-1,-1);
\draw[thick,->](0,0)--(0,-1);
%\draw[thick,->](0,0)--(1,-1);
\end{tikzpicture}
&
\begin{tikzpicture}[scale=.2, baseline=(current bounding box.center)]
\foreach \x in {-1,0,1} \foreach \y in {-1,0,1} \fill(\x,\y) circle[radius=2pt];
\draw[thick,->](0,0)--(-1,1);
\draw[thick,->](0,0)--(0,1);
\draw[thick,->](0,0)--(1,1);
\draw[thick,->](0,0)--(-1,0);
%\draw[thick,->](0,0)--(1,0);
\draw[thick,->](0,0)--(-1,-1);
\draw[thick,->](0,0)--(0,-1);
%\draw[thick,->](0,0)--(1,-1);
\end{tikzpicture}
\\
\begin{tikzpicture}[scale=.2, baseline=(current bounding box.center)]
\foreach \x in {-1,0,1} \foreach \y in {-1,0,1} \fill(\x,\y) circle[radius=2pt];
\draw[thick,->](0,0)--(-1,1);
\draw[thick,->](0,0)--(0,1);
\draw[thick,->](0,0)--(1,1);
%\draw[thick,->](0,0)--(-1,0);
\draw[thick,->](0,0)--(1,0);
\draw[thick,->](0,0)--(-1,-1);
%\draw[thick,->](0,0)--(0,-1);
%\draw[thick,->](0,0)--(1,-1);
\end{tikzpicture}
&
\begin{tikzpicture}[scale=.2, baseline=(current bounding box.center)]
\foreach \x in {-1,0,1} \foreach \y in {-1,0,1} \fill(\x,\y) circle[radius=2pt];
\draw[thick,->](0,0)--(-1,1);
\draw[thick,->](0,0)--(0,1);
\draw[thick,->](0,0)--(1,1);
%\draw[thick,->](0,0)--(-1,0);
\draw[thick,->](0,0)--(1,0);
%\draw[thick,->](0,0)--(-1,-1);
\draw[thick,->](0,0)--(0,-1);
%\draw[thick,->](0,0)--(1,-1);
\end{tikzpicture}
&
\begin{tikzpicture}[scale=.2, baseline=(current bounding box.center)]
\foreach \x in {-1,0,1} \foreach \y in {-1,0,1} \fill(\x,\y) circle[radius=2pt];
\draw[thick,->](0,0)--(-1,1);
\draw[thick,->](0,0)--(0,1);
\draw[thick,->](0,0)--(1,1);
\draw[thick,->](0,0)--(-1,0);
\draw[thick,->](0,0)--(1,0);
\draw[thick,->](0,0)--(-1,-1);
\draw[thick,->](0,0)--(0,-1);
%\draw[thick,->](0,0)--(1,-1);
\end{tikzpicture}
&
\begin{tikzpicture}[scale=.2, baseline=(current bounding box.center)]
\foreach \x in {-1,0,1} \foreach \y in {-1,0,1} \fill(\x,\y) circle[radius=2pt];
%\draw[thick,->](0,0)--(-1,1);
\draw[thick,->](0,0)--(0,1);
\draw[thick,->](0,0)--(1,1);
%\draw[thick,->](0,0)--(-1,0);
%\draw[thick,->](0,0)--(1,0);
\draw[thick,->](0,0)--(-1,-1);
%\draw[thick,->](0,0)--(0,-1);
\draw[thick,->](0,0)--(1,-1);
\end{tikzpicture}
&
\begin{tikzpicture}[scale=.2, baseline=(current bounding box.center)]
\foreach \x in {-1,0,1} \foreach \y in {-1,0,1} \fill(\x,\y) circle[radius=2pt];
%\draw[thick,->](0,0)--(-1,1);
\draw[thick,->](0,0)--(0,1);
\draw[thick,->](0,0)--(1,1);
\draw[thick,->](0,0)--(-1,0);
%\draw[thick,->](0,0)--(1,0);
%\draw[thick,->](0,0)--(-1,-1);
%\draw[thick,->](0,0)--(0,-1);
\draw[thick,->](0,0)--(1,-1);
\end{tikzpicture}
&
\begin{tikzpicture}[scale=.2, baseline=(current bounding box.center)]
\foreach \x in {-1,0,1} \foreach \y in {-1,0,1} \fill(\x,\y) circle[radius=2pt];
%\draw[thick,->](0,0)--(-1,1);
\draw[thick,->](0,0)--(0,1);
\draw[thick,->](0,0)--(1,1);
\draw[thick,->](0,0)--(-1,0);
%\draw[thick,->](0,0)--(1,0);
\draw[thick,->](0,0)--(-1,-1);
%\draw[thick,->](0,0)--(0,-1);
\draw[thick,->](0,0)--(1,-1);
\end{tikzpicture}
&
\begin{tikzpicture}[scale=.2, baseline=(current bounding box.center)]
\foreach \x in {-1,0,1} \foreach \y in {-1,0,1} \fill(\x,\y) circle[radius=2pt];
%\draw[thick,->](0,0)--(-1,1);
\draw[thick,->](0,0)--(0,1);
\draw[thick,->](0,0)--(1,1);
%\draw[thick,->](0,0)--(-1,0);
%\draw[thick,->](0,0)--(1,0);
\draw[thick,->](0,0)--(-1,-1);
\draw[thick,->](0,0)--(0,-1);
\draw[thick,->](0,0)--(1,-1);
\end{tikzpicture}
&
\begin{tikzpicture}[scale=.2, baseline=(current bounding box.center)]
\foreach \x in {-1,0,1} \foreach \y in {-1,0,1} \fill(\x,\y) circle[radius=2pt];
%\draw[thick,->](0,0)--(-1,1);
\draw[thick,->](0,0)--(0,1);
\draw[thick,->](0,0)--(1,1);
%\draw[thick,->](0,0)--(-1,0);
\draw[thick,->](0,0)--(1,0);
\draw[thick,->](0,0)--(-1,-1);
\draw[thick,->](0,0)--(0,-1);
\draw[thick,->](0,0)--(1,-1);
\end{tikzpicture}
&
\begin{tikzpicture}[scale=.2, baseline=(current bounding box.center)]
\foreach \x in {-1,0,1} \foreach \y in {-1,0,1} \fill(\x,\y) circle[radius=2pt];
%\draw[thick,->](0,0)--(-1,1);
\draw[thick,->](0,0)--(0,1);
\draw[thick,->](0,0)--(1,1);
\draw[thick,->](0,0)--(-1,0);
\draw[thick,->](0,0)--(1,0);
\draw[thick,->](0,0)--(-1,-1);
%\draw[thick,->](0,0)--(0,-1);
\draw[thick,->](0,0)--(1,-1);
\end{tikzpicture}
&
 \begin{tikzpicture}[scale=.2, baseline=(current bounding box.center)]
\foreach \x in {-1,0,1} \foreach \y in {-1,0,1} \fill(\x,\y) circle[radius=2pt];
%\draw[thick,->](0,0)--(-1,1);
\draw[thick,->](0,0)--(0,1);
\draw[thick,->](0,0)--(1,1);
\draw[thick,->](0,0)--(-1,0);
%\draw[thick,->](0,0)--(1,0);
\draw[thick,->](0,0)--(-1,-1);
\draw[thick,->](0,0)--(0,-1);
\draw[thick,->](0,0)--(1,-1);
\end{tikzpicture}
&
\begin{tikzpicture}[scale=.2, baseline=(current bounding box.center)]
\foreach \x in {-1,0,1} \foreach \y in {-1,0,1} \fill(\x,\y) circle[radius=2pt];
%\draw[thick,->](0,0)--(-1,1);
\draw[thick,->](0,0)--(0,1);
%\draw[thick,->](0,0)--(1,1);
\draw[thick,->](0,0)--(-1,0);
\draw[thick,->](0,0)--(1,0);
%\draw[thick,->](0,0)--(-1,-1);
\draw[thick,->](0,0)--(0,-1);
\draw[thick,->](0,0)--(1,-1);
\end{tikzpicture}
&
\begin{tikzpicture}[scale=.2, baseline=(current bounding box.center)]
\foreach \x in {-1,0,1} \foreach \y in {-1,0,1} \fill(\x,\y) circle[radius=2pt];
%\draw[thick,->](0,0)--(-1,1);
\draw[thick,->](0,0)--(0,1);
%\draw[thick,->](0,0)--(1,1);
\draw[thick,->](0,0)--(-1,0);
\draw[thick,->](0,0)--(1,0);
\draw[thick,->](0,0)--(-1,-1);
\draw[thick,->](0,0)--(0,-1);
%\draw[thick,->](0,0)--(1,-1);
\end{tikzpicture}
&
\begin{tikzpicture}[scale=.2, baseline=(current bounding box.center)]
\foreach \x in {-1,0,1} \foreach \y in {-1,0,1} \fill(\x,\y) circle[radius=2pt];
\draw[thick,->](0,0)--(-1,1);
\draw[thick,->](0,0)--(0,1);
%\draw[thick,->](0,0)--(1,1);
%\draw[thick,->](0,0)--(-1,0);
\draw[thick,->](0,0)--(1,0);
\draw[thick,->](0,0)--(-1,-1);
%\draw[thick,->](0,0)--(0,-1);
\draw[thick,->](0,0)--(1,-1);
\end{tikzpicture}
&
\begin{tikzpicture}[scale=.2, baseline=(current bounding box.center)]
\foreach \x in {-1,0,1} \foreach \y in {-1,0,1} \fill(\x,\y) circle[radius=2pt];
\draw[thick,->](0,0)--(-1,1);
\draw[thick,->](0,0)--(0,1);
%\draw[thick,->](0,0)--(1,1);
\draw[thick,->](0,0)--(-1,0);
\draw[thick,->](0,0)--(1,0);
\draw[thick,->](0,0)--(-1,-1);
%\draw[thick,->](0,0)--(0,-1);
\draw[thick,->](0,0)--(1,-1);
\end{tikzpicture}
&
\begin{tikzpicture}[scale=.2, baseline=(current bounding box.center)]
\foreach \x in {-1,0,1} \foreach \y in {-1,0,1} \fill(\x,\y) circle[radius=2pt];
\draw[thick,->](0,0)--(-1,1);
\draw[thick,->](0,0)--(0,1);
%\draw[thick,->](0,0)--(1,1);
\draw[thick,->](0,0)--(-1,0);
\draw[thick,->](0,0)--(1,0);
\draw[thick,->](0,0)--(-1,-1);
\draw[thick,->](0,0)--(0,-1);
\draw[thick,->](0,0)--(1,-1);
\end{tikzpicture}
&
\begin{tikzpicture}[scale=.2, baseline=(current bounding box.center)]
\foreach \x in {-1,0,1} \foreach \y in {-1,0,1} \fill(\x,\y) circle[radius=2pt];
%\draw[thick,->](0,0)--(-1,1);
\draw[thick,->](0,0)--(0,1);
%\draw[thick,->](0,0)--(1,1);
\draw[thick,->](0,0)--(-1,0);
%\draw[thick,->](0,0)--(1,0);
%\draw[thick,->](0,0)--(-1,-1);
\draw[thick,->](0,0)--(0,-1);
\draw[thick,->](0,0)--(1,-1);
\end{tikzpicture}
\\
\begin{tikzpicture}[scale=.2, baseline=(current bounding box.center)]
\foreach \x in {-1,0,1} \foreach \y in {-1,0,1} \fill(\x,\y) circle[radius=2pt];
%\draw[thick,->](0,0)--(-1,1);
\draw[thick,->](0,0)--(0,1);
%\draw[thick,->](0,0)--(1,1);
\draw[thick,->](0,0)--(-1,0);
%\draw[thick,->](0,0)--(1,0);
\draw[thick,->](0,0)--(-1,-1);
%\draw[thick,->](0,0)--(0,-1);
\draw[thick,->](0,0)--(1,-1);
\end{tikzpicture}
&
\begin{tikzpicture}[scale=.2, baseline=(current bounding box.center)]
\foreach \x in {-1,0,1} \foreach \y in {-1,0,1} \fill(\x,\y) circle[radius=2pt];
\draw[thick,->](0,0)--(-1,1);
\draw[thick,->](0,0)--(0,1);
%\draw[thick,->](0,0)--(1,1);
%\draw[thick,->](0,0)--(-1,0);
%\draw[thick,->](0,0)--(1,0);
\draw[thick,->](0,0)--(-1,-1);
%\draw[thick,->](0,0)--(0,-1);
\draw[thick,->](0,0)--(1,-1);
\end{tikzpicture}
&
\begin{tikzpicture}[scale=.2, baseline=(current bounding box.center)]
\foreach \x in {-1,0,1} \foreach \y in {-1,0,1} \fill(\x,\y) circle[radius=2pt];
\draw[thick,->](0,0)--(-1,1);
\draw[thick,->](0,0)--(0,1);
%\draw[thick,->](0,0)--(1,1);
\draw[thick,->](0,0)--(-1,0);
%\draw[thick,->](0,0)--(1,0);
%\draw[thick,->](0,0)--(-1,-1);
%\draw[thick,->](0,0)--(0,-1);
\draw[thick,->](0,0)--(1,-1);
\end{tikzpicture}
&
\begin{tikzpicture}[scale=.2, baseline=(current bounding box.center)]
\foreach \x in {-1,0,1} \foreach \y in {-1,0,1} \fill(\x,\y) circle[radius=2pt];
%\draw[thick,->](0,0)--(-1,1);
\draw[thick,->](0,0)--(0,1);
%\draw[thick,->](0,0)--(1,1);
\draw[thick,->](0,0)--(-1,0);
%\draw[thick,->](0,0)--(1,0);
\draw[thick,->](0,0)--(-1,-1);
\draw[thick,->](0,0)--(0,-1);
\draw[thick,->](0,0)--(1,-1);
\end{tikzpicture}
&
\begin{tikzpicture}[scale=.2, baseline=(current bounding box.center)]
\foreach \x in {-1,0,1} \foreach \y in {-1,0,1} \fill(\x,\y) circle[radius=2pt];
\draw[thick,->](0,0)--(-1,1);
\draw[thick,->](0,0)--(0,1);
%\draw[thick,->](0,0)--(1,1);
%\draw[thick,->](0,0)--(-1,0);
%\draw[thick,->](0,0)--(1,0);
\draw[thick,->](0,0)--(-1,-1);
\draw[thick,->](0,0)--(0,-1);
\draw[thick,->](0,0)--(1,-1);
\end{tikzpicture}
&
 \begin{tikzpicture}[scale=.2, baseline=(current bounding box.center)]
\foreach \x in {-1,0,1} \foreach \y in {-1,0,1} \fill(\x,\y) circle[radius=2pt];
\draw[thick,->](0,0)--(-1,1);
\draw[thick,->](0,0)--(0,1);
%\draw[thick,->](0,0)--(1,1);
\draw[thick,->](0,0)--(-1,0);
%\draw[thick,->](0,0)--(1,0);
%\draw[thick,->](0,0)--(-1,-1);
\draw[thick,->](0,0)--(0,-1);
\draw[thick,->](0,0)--(1,-1);
\end{tikzpicture}
&
 \begin{tikzpicture}[scale=.2, baseline=(current bounding box.center)]
\foreach \x in {-1,0,1} \foreach \y in {-1,0,1} \fill(\x,\y) circle[radius=2pt];
\draw[thick,->](0,0)--(-1,1);
\draw[thick,->](0,0)--(0,1);
%\draw[thick,->](0,0)--(1,1);
\draw[thick,->](0,0)--(-1,0);
%\draw[thick,->](0,0)--(1,0);
\draw[thick,->](0,0)--(-1,-1);
%\draw[thick,->](0,0)--(0,-1);
\draw[thick,->](0,0)--(1,-1);
\end{tikzpicture}
&
\begin{tikzpicture}[scale=.2, baseline=(current bounding box.center)]
\foreach \x in {-1,0,1} \foreach \y in {-1,0,1} \fill(\x,\y) circle[radius=2pt];
\draw[thick,->](0,0)--(-1,1);
\draw[thick,->](0,0)--(0,1);
%\draw[thick,->](0,0)--(1,1);
\draw[thick,->](0,0)--(-1,0);
%\draw[thick,->](0,0)--(1,0);
\draw[thick,->](0,0)--(-1,-1);
\draw[thick,->](0,0)--(0,-1);
\draw[thick,->](0,0)--(1,-1);
\end{tikzpicture}
& 
\begin{tikzpicture}[scale=.2, baseline=(current bounding box.center)]
\foreach \x in {-1,0,1} \foreach \y in {-1,0,1} \fill(\x,\y) circle[radius=2pt];
%\draw[thick,->](0,0)--(-1,1);
%\draw[thick,->](0,0)--(0,1);
\draw[thick,->](0,0)--(1,1);
\draw[thick,->](0,0)--(-1,0);
%\draw[thick,->](0,0)--(1,0);
\draw[thick,->](0,0)--(-1,-1);
\draw[thick,->](0,0)--(0,-1);
%\draw[thick,->](0,0)--(1,-1);
\end{tikzpicture}
&
\begin{tikzpicture}[scale=.2, baseline=(current bounding box.center)]
\foreach \x in {-1,0,1} \foreach \y in {-1,0,1} \fill(\x,\y) circle[radius=2pt];
%\draw[thick,->](0,0)--(-1,1);
\draw[thick,->](0,0)--(0,1);
\draw[thick,->](0,0)--(1,1);
%\draw[thick,->](0,0)--(-1,0);
\draw[thick,->](0,0)--(1,0);
\draw[thick,->](0,0)--(-1,-1);
%\draw[thick,->](0,0)--(0,-1);
%\draw[thick,->](0,0)--(1,-1);
\end{tikzpicture}
& 
\begin{tikzpicture}[scale=.2, baseline=(current bounding box.center)]
\foreach \x in {-1,0,1} \foreach \y in {-1,0,1} \fill(\x,\y) circle[radius=2pt];
\draw[thick,->](0,0)--(-1,1);
\draw[thick,->](0,0)--(0,1);
\draw[thick,->](0,0)--(1,1);
%\draw[thick,->](0,0)--(-1,0);
\draw[thick,->](0,0)--(1,0);
%\draw[thick,->](0,0)--(-1,-1);
%\draw[thick,->](0,0)--(0,-1);
\draw[thick,->](0,0)--(1,-1);
\end{tikzpicture} 
&
\begin{tikzpicture}[scale=.2, baseline=(current bounding box.center)]
\foreach \x in {-1,0,1} \foreach \y in {-1,0,1} \fill(\x,\y) circle[radius=2pt];
\draw[thick,->](0,0)--(-1,1);
\draw[thick,->](0,0)--(0,1);
%\draw[thick,->](0,0)--(1,1);
%\draw[thick,->](0,0)--(-1,0);
\draw[thick,->](0,0)--(1,0);
%\draw[thick,->](0,0)--(-1,-1);
%\draw[thick,->](0,0)--(0,-1);
\draw[thick,->](0,0)--(1,-1);
\end{tikzpicture} 
&
\begin{tikzpicture}[scale=.2, baseline=(current bounding box.center)]
\foreach \x in {-1,0,1} \foreach \y in {-1,0,1} \fill(\x,\y) circle[radius=2pt];
\draw[thick,->](0,0)--(-1,1);
\draw[thick,->](0,0)--(0,1);
%\draw[thick,->](0,0)--(1,1);
%\draw[thick,->](0,0)--(-1,0);
%\draw[thick,->](0,0)--(1,0);
%\draw[thick,->](0,0)--(-1,-1);
%\draw[thick,->](0,0)--(0,-1);
\draw[thick,->](0,0)--(1,-1);
\end{tikzpicture} 
& 
\begin{tikzpicture}[scale=.2, baseline=(current bounding box.center)]
\foreach \x in {-1,0,1} \foreach \y in {-1,0,1} \fill(\x,\y) circle[radius=2pt];
\draw[thick,->](0,0)--(-1,1);
%\draw[thick,->](0,0)--(0,1);
\draw[thick,->](0,0)--(1,1);
%\draw[thick,->](0,0)--(-1,0);
%\draw[thick,->](0,0)--(1,0);
%\draw[thick,->](0,0)--(-1,-1);
%\draw[thick,->](0,0)--(0,-1);
\draw[thick,->](0,0)--(1,-1);
\end{tikzpicture} 
&
\begin{tikzpicture}[scale=.2, baseline=(current bounding box.center)]
\foreach \x in {-1,0,1} \foreach \y in {-1,0,1} \fill(\x,\y) circle[radius=2pt];
\draw[thick,->](0,0)--(-1,1);
\draw[thick,->](0,0)--(0,1);
\draw[thick,->](0,0)--(1,1);
%\draw[thick,->](0,0)--(-1,0);
%\draw[thick,->](0,0)--(1,0);
%\draw[thick,->](0,0)--(-1,-1);
%\draw[thick,->](0,0)--(0,-1);
\draw[thick,->](0,0)--(1,-1);
\end{tikzpicture}&
\end{array}$
\end{trivlist}
\caption{Classification of the $79$ models with respect to the $x$ and $y$-variables.}\label{figcas}
\end{figure}

\subsection*{Statement of the main result}
In this paper, we  address the question of the classification with respect to the variable $t$ and we prove that this classification coincides with the classification with respect to $x$ and $y$. There is a priori no relation between the $\frac{d}{dx}$ and $\frac{d}{dt}$ differential algebraic properties of a function in these two variables. For instance,   the function $t\Gamma(x)$ is holonomic with respect to $t$ but not differentially algebraic with respect to $x$, thanks to Hölder's result. In that case, the fibration induced by $t$ is ``isotrivial''. The main difficulty in our case is to show that such a situation does not happen and that the $x$ and $t$-algebraic behavior are intrinsically connected.

 Our main result is as follows:

 \begin{thmintro}[Theorem \ref{thm:genre0} and Corollary \ref{cor:diffalgxequivalentdiffalgtgenusone} below]\label{thm:introtdiffalgimpliesxdiffalg}
For any of the $79$ models of Figure~\ref{figcas},  the complete generating function
  is $\frac{d}{dt}$-differentially algebraic  over $\Q$ if and only if 
 it is $\frac{d}{dx}$-differentially algebraic over $\Q$. 
 \end{thmintro}
 
Theorem \ref{thm:introtdiffalgimpliesxdiffalg} is the corollary  of the following proposition proved in the more general setting of 
weighted walks that are walks whose directions are weighted  (see \S \ref{sec1}). To any such a  walk, one attaches an algebraic curve of genus zero or one called the \emph{kernel curve} and a group of automorphisms of that curve called \emph{the group of the walk} (see \S \ref{sec1}). The following holds.

\begin{thmintro}[Theorems \ref{thm:genre0} and \ref{theo2} below]\label{propintro:gtransimpliesxtrans}
For a  genus zero kernel curve attached to the models \eqref{G0}, the generating series is $\frac{d}{dt}$-differentially transcendental over $\Q$. For a genus one kernel curve with infinite group of the walk,  if the generating series is $\frac{d}{dt}$-differentially algebraic over $\Q$, then it is $\frac{d}{dx}$-differentially algebraic over $\Q$.
\end{thmintro}

In \cite{DreyfusHardouinRoquesSingerGenuszero}, the authors proved that for a genus zero kernel curve attached to the models \eqref{G0}, the generating series was $\frac{d}{dx}$-differentially transcendental over $\Q$.
The authors of \cite{BBMR16} proved that the nine nonholomic $\frac{d}{dx}$-differentially algebraic models  of Figure~\ref{figcas} were also $\frac{d}{dt}$-differentially algebraic over $\Q$ by giving an   explicit description of the series in terms of  analytic invariants. In \S \ref{sec54}, we will discuss how the construction of \cite{BBMR16} and  the results of \cite{hardouin2020differentially} should imply that  the second statement of  Theorem~\ref{propintro:gtransimpliesxtrans} is in fact an equivalence. 

\subsection*{Strategy of the proof}

The classification results of Figure \ref{figcas} come from  many approaches: probabilistic methods,  combinatorial classification,  computer algebra and ``Guess and Prove'',  analysis and boundary value problems, and more recently  difference Galois theory and algebraic geometry. The 
analytic approach consists in studying the asymptotic growth of the coefficients of the generating function, or else showing that it has an infinite number of singularities, in order to prove its nonholonomicity. This approach also allows for the study of some important specializations of the complete generating function as for instance $Q(1,1,t)$ the generating function for the number of nearest neighbor walks in the quarter plane (see \cite{MelcMish,MR09}). Though very powerful, these analytic techniques are unable to detect the differentially algebraic generating functions among the nonholonomic ones. For instance, the generating function $\prod_{k=1}^\infty \frac{1}{(1-x^k)}$ counting the number of partitions has an infinite number of singularities, and yet is $\frac{d}{dx}$-differentially algebraic. \par 

In order to detect these more subtle kinds of functional dependencies it is necessary to use new arguments that focus on the functional equation satisfied by the complete generating function.
Indeed, the combinatorial decomposition of a walk into a shorter walk followed by an admissible step translates into a functional equation for the generating function.   Following the ideas of Fayolle, Iasnogorodski and Malyshev  \cite{FIM}, the authors of \cite{KurkRasch} and \cite{DreyfusHardouinRoquesSingerGenuszero} specialized this functional equation  to the so-called \emph{kernel curve} to find a linear discrete equation: a linear $\q$-difference equation in the genus zero case and a shift difference equation in genus one. Difference Galois theory allowed then to characterize the differentially transcendental complete generating function (\cite{DHRS,DreyfusHardouinRoquesSingerGenuszero})  whereas the clever use of Tutte invariants produces explicit differential algebraic relations  for the $9$ nonholomic but differentially algebraic cases (\cite{BBMR16}). 
Unfortunately, all the above methods for proving the differential transcendence are only valid for a fixed value of the parameter $t$ in the field of complex numbers. This allowed 
the authors to consider the kernel curve as a complex algebraic curve but prevented them to study the variations of the parameter $t$.

Our work relies on  an nonarchimedean uniformization of the kernel curve, which we consider as an algebraic curve over $\Q(t)$. 
We use here the formalism of Tate curves over $\Q(t)$ as in   \cite{roquette1970analytic} to show that for both situations, genus one and zero, the differential algebraic properties of the complete generating functions are encoded by the differential algebraic properties of a solution of a rank one  nonhomogeneous linear $\q$-difference equation which  unifies the genus zero and the genus one cases. 
Then, we generalize some Galoisian criterias for $\q$-difference equations of \cite{HS} to  
 prove Theorem \ref{thm:genre0} and Theorem \ref{theo2} below.

\subsection*{Organization of the paper}
The paper is organized as follows. In Section \ref{sec1} we present some reminders and notations for  walks in the quarter plane. In Section \ref{secgenre0} we consider  walks with genus zero kernel curve while  Section \ref{secgenre1}  deals with  the genus one case. Since this paper combines many different fields, nonarchimedian uniformization, combinatorics, and Galois theory, we choose to postpone many technical  intermediate results to the appendices. This should allow the reader to understand the articulation of our proofs in Sections \ref{secgenre0} and~\ref{secgenre1}  in three steps without being lost in  too many  technicalities. These three steps are the uniformization of the kernel and the  construction  of a linear $\q$-difference equation, the Galoisian criteria, and  finally, the resolution of telescoping problems. Appendix \ref{sec:nonarcheestimates} is devoted to the nonarchimedean estimates that we used in the uniformization procedure. Appendix \ref{sec:nonarchimedianpreleminaries} contains some reminders on  special functions on Tate curves and their normal forms. Appendix \ref{sec:differenceGaloistheory} proves  the Galoisian criteria mentioned above.  Finally, Appendix \ref{sec:merofunctiontate} studies the transcendence properties of  special functions on Tate curves which  will be used for the descent of our telescoping equations.
 
%%%%%%%%%%%%%%%%%%%%%%%%%%%%%%%%%%%%%%%%%%%%%%%%%%%%%%%%%%%%%%%%%%%%%%%%%%%%%%%%%%%%%%%%%%%
%%%%%%%%%%%%%%%%%%%%%%%%%%%%%%%%%%%%%%%%%%%%%%%%%%%%%%%%%%%%%%%%%%%%%%%%%%%%%%%%%%%%%%%%%%%%
\section{The walks in the quadrant}\label{sec1}
%%%%%%%%%%%%%%%%%%%%%%%%%%%%%%%%%%%%%%%%%%%%%%%%%%%%%%%%%%%%%%%%%%%%%%%%%%%%%%%%%%%%%%%%%%%
%%%%%%%%%%%%%%%%%%%%%%%%%%%%%%%%%%%%%%%%%%%%%%%%%%%%%%%%%%%%%%%%%%%%%%%%%%%%%%%%%%%%%%%%%%%%

The goal of this section is to introduce some basic properties of walks in the quarter plane. In $\S \ref{sec:notationwalk}$, we introduce  the generating function $Q(x,y,t)$ of a walk confined in the quarter plane. In $\S \ref{sec:Kernelcurve}$, we attach to any  walk a \emph{kernel curve}, which is an algebraic curve 
defined over $\Q[t]$. This curve has been intensively studied as an algebraic curve over $\C$ by fixing a morphism from $\Q[t]$ to $\C$. For instance,  \cite{FIM} is concerned
with $t=1$ whereas the papers \cite{DHRS} and \cite{dreyfus2019differential} focus respectively on $t \in \C$ transcendental over $\Q$ and $t \in ]0,1[$. Unfortunately, 
specializing $t$ even generically does not allow to study the $t$-dependencies of the generating function. In this paper, we do not work with a specialization of $t$. This 
forces us to  move away from the archimedean framework of the field of complex numbers  and to consider the kernel curve over a suitable valued field extension
of $\Q(t)$ endowed with the valuation at $0$.

\subsection{The walks}\label{sec:notationwalk}
The cardinal directions  of the plane  $\{
\begin{tikzpicture}[scale=0.3]
;
\draw[thick,->](0,0)--(-1,0);
\end{tikzpicture}
,
\begin{tikzpicture}[scale=0.3]
;
\draw[thick,->](0,0)--(-1,1);
\end{tikzpicture}
 ,
 \begin{tikzpicture}[scale=0.3]
;
\draw[thick,->](0,0)--(0,1);
\end{tikzpicture}
 ,\begin{tikzpicture}[scale=0.3]
;
\draw[thick,->](0,0)--(1,1);
\end{tikzpicture},
\begin{tikzpicture}[scale=0.3]
;
\draw[thick,->](0,0)--(1,0);
\end{tikzpicture}
,
\begin{tikzpicture}[scale=0.3]
;
\draw[thick,->](0,0)--(1,-1);
\end{tikzpicture}
,
\begin{tikzpicture}[scale=0.3]
;
\draw[thick,->](0,0)--(0,-1);
\end{tikzpicture}
,
\begin{tikzpicture}[scale=0.3]
;
\draw[thick,->](0,0)--(-1,-1);
\end{tikzpicture} \}$   are identified with  pairs of integers  $(i,j)\in\{0,\pm 1\}^{2}\backslash\{(0,0)\}$. A walk $\mathcal{W}$ in the quarter plane $\Z_{\geq 0}^{2}$ is a sequence of points $(M_{n})_{n \in \Z_{\geq 0}}$ such that 
\begin{itemize} 
\item it starts at $(0,0)$, that is, $M_0=(0,0)$;
\item for all $n \in \Z_{\geq 0}$, the point $M_n$ belong to the quadrant $ \Z_{\geq 0} \times  \Z_{\geq 0}$;
\item  for all $n \in \Z_{\geq 0}$, the vector $M_{n+1}-M_n$ belongs to a given subset $\mathcal{D}$ of the  set of cardinal directions.
\end{itemize}
Fixing  a family of elements $(d_{i,j})_{(i,j)\in\{0,\pm 1\}^{2}}$ of $\Q\cap [0,1]$ such that $\sum_{i,j} d_{i,j}=1$, one can choose to weight the model of the walk in order to add a probabilistic flavor to our study. For $(i,j)\in\{0,\pm 1\}^{2}\backslash\{(0,0)\}$ (resp. $(0,0)$), the  element   $d_{i,j}$ can be viewed as  the probability for the walk  to go in the direction $(i,j)$ (resp.  to stay  at the same position). In that case,  the $d_{i,j}$ are called the weights and the model is called a weighted model.   \emph{The set of steps} $\mathcal{D}$  of the walk  is  the 
set of cardinal  directions with nonzero weight, that is, 
$$
\mathcal{D}=\{(i,j) \in\{0,\pm 1\}^{2}| d_{i,j} \neq 0 \}.
$$

A  model  is \emph{unweighted} if $d_{0,0}=0$ and if the nonzero $d_{i,j}$'s all have the same value. 
\begin{rmk}
In what follows we will represent model of walks with arrows. For instance, the family of models  represented by 
$$\begin{tikzpicture}[scale=0.6, baseline=(current bounding box.center)]
\draw[thick,->](0,0)--(-1,-1);
\draw[thick,->](0,0)--(1,-1);
\draw[thick,->](0,0)--(1,1);
\draw[thick,->](0,0)--(0,1);
\end{tikzpicture} \hbox{ or } \left\{\begin{tikzpicture}[scale=.6, baseline=(current bounding box.center)]
\draw[thick,->](0,0)--(1,1);
\end{tikzpicture}, \begin{tikzpicture}[scale=.6, baseline=(current bounding box.center)]
\draw[thick,->](0,0)--(1,-1);
\end{tikzpicture}, \begin{tikzpicture}[scale=.6, baseline=(current bounding box.center)]
\draw[thick,->](0,0)--(0,1);
\end{tikzpicture}, \begin{tikzpicture}[scale=.6, baseline=(current bounding box.center)]
\draw[thick,->](0,0)--(-1,-1);
\end{tikzpicture} \right\},$$
correspond to models with $d_{1,1},d_{1,-1},d_{0,1},d_{-1,-1}\neq 0$, $d_{1,0}=d_{0,-1}=d_{-1,1}=d_{-1,0}=0$, and where nothing is assumed on the value of $d_{0,0}$. In the following results, 
 the behavior of the  kernel curve never depends on $d_{0,0}$. This is the reason why, to reduce the amount of notations, we have decided not to mention $d_{0,0}$ in the graphical representation of the model.\end{rmk}

The {\it weight of the walk} is defined to be the product of the weights of its component steps. For any $(i,j)\in \Z_{\geq 0}^{2}$ and any $k\in \Z_{\geq 0}$, we let $q_{i,j,k}$ be the sum of the weights of all walks reaching  the position $(i,j)$ from the initial position $(0,0)$ after $k$ steps.  We introduce the corresponding trivariate generating function
$$
Q(x,y,t):=\displaystyle \sum_{i,j,k\geq 0}q_{i,j,k}x^{i}y^{j}t^{k}.
$$

Note that the generating function is not exactly the same as the one that we defined in the introduction. To recover the latter, we should take  $d_{i,j}\in \{0,1\}$ and $d_{i,j}=1$ if and only if  the corresponding direction belongs to $\mathcal{D}$. Fortunately, the assumption $\sum_{i,j} d_{i,j}=1$ can be relaxed by rescaling the $t$-variable, and the results of the present paper stay valid for the generating function of the introduction since both generating functions have the same nature.
 
\begin{rmk}For simplicity, we assume that the weights $d_{i,j}$ belong to  $\Q$. However, we would like to mention that any of the arguments and statements  below will hold with arbitrary real weights in $[0,1]$. One just needs to  replace the field $\Q$ with the field $\Q(d_{i,j})$. \end{rmk}

The {\it kernel} polynomial of a  weighted model  $(d_{i,j})_{i,j\in \{0,\pm 1\}^{2}}$  is defined by 
\begin{equation}\label{eq:equationforthekernel}
K(x,y,t):=xy (1-t S(x,y))
\end{equation}

where 
\begin{equation}\label{eq:defiAiBi}
\begin{array}{lll}
S(x,y) &=&\sum_{(i,j)\in \{0,\pm 1\}^{2}} d_{i,j}x^i y^j\\
&=& A_{-1}(x) \frac{1}{y} +A_{0}(x)+ A_{1}(x) y\\
&= & B_{-1}(y) \frac{1}{x} +B_{0}(y)+ B_{1}(y) x,
\end{array}
\end{equation}

and $A_{i}(x) \in x^{-1}\Q[x]$, $B_{i}(y) \in y^{-1}\Q[y]$.

By  \cite[Lemma 1.1]{DreyfusHardouinRoquesSingerGenuszero}, see also \cite[Lemma 4]{BMM}, the generating function $Q(x,y,t)$ satisfies the following functional equation:
\begin{equation}\label{eq:fundamentalkernelequationseries}
K(x,y,t)Q(x,y,t)=xy+{F}^{1}(x,t) +{F}^{2}(y,t)+td_{-1,-1} Q(0,0,t),
\end{equation}

where 
$$
{F}^{1}(x,t):= K(x,0,t)Q(x,0,t), \ \mbox{ and } \ {F}^{2}(y,t):= K(0,y,t)Q(0,y,t).$$

\begin{rmk}\label{rem1}
We shall often use the following symmetry argument between $x$ and $y$. Exchanging  $x$ and $y$ in the kernel polynomial amounts to consider the kernel  polynomial  of a weighted model  $\mathcal{D}':=\{ (i,j) \mbox { such that }  (j,i) \in \mathcal{D}\}$ with  weights $d'_{i,j}:=d_{j,i}$.\end{rmk}

\subsection{The kernel curve} \label{sec:Kernelcurve}

The \emph{kernel} polynomial may be seen as a bivariate polynomial in $x,y$ with coefficients in $\Q(t)$. The latter is a  valued field  endowed with the valuation at zero. It  is neither  algebraically closed   nor complete. In order to use the theory of Tate curves, one needs to consider a complete algebraically closed field  extension of $\Q(t)$. The field of Puiseux series with coefficients in $\overline{\Q}$ is algebraically closed but not complete. We may consider the field $C$ of Hahn series or  Malcev-Neumann series with coefficients in  $\overline{\Q}$, and monomials from $\Q$. We recall that  a  Hahn series $f$
is a formal power series $\sum_{ \gamma \in \Q }c_\gamma t^\gamma$ with coefficients $c_\gamma$ in $\overline{\Q}$ and such that the subset $\{\gamma | c_\gamma \neq 0 \}$ is a well ordered subset of $\Q$. The valuation $v_0( f)$ of $f$ is the smallest element of the subset $\{\gamma | c_\gamma \neq 0 \}$. The field $C$ is algebraically closed and complete 
with respect to the valuation at zero, see \cite[Ex. 3.2.23 and p.~151]{AschenbrennerVandenDriesVanDerHoeven}. One can endow 
$C$ with a derivation $\partial_t$  as follows
$$ \partial_t\left(\sum_{ \gamma \in \Q }c_\gamma t^\gamma \right)= \sum_{ \gamma \in \Q }c_\gamma \gamma t^\gamma.$$
Then, $\partial_t$ extends the derivation $t\frac{d}{dt}$ of $\Q(t)$, see \cite[Ex.(2), \S 4.4]{AschenbrennerVandenDriesVanDerHoeven}. 

Let us fix once for all $\alpha \in \R$ such that $0< \alpha <1$. For any $f \in C$, we define the norm of $f$ as $|f|= \alpha^{v_0(f)}$. For any  Hahn series $f$ such that $|f|<1$, we have $|\partial_t (f)| <1$. This is not true when $\partial_t$ is replaced by $\frac{d}{dt}$.

\par We need to discard some degenerate cases. Following \cite{FIM}, we have the following  definition. 
 
 \begin{defi}\label{defi:degenerate}
A weighted model is called {\it {degenerate}} if one of the following holds:
\begin{itemize}
\item $K(x,y,t)$ is reducible as an element of the polynomial ring $C[x,y]$, 
\item $K(x,y,t)$ has $x$-degree less than or equal to $1$,
\item $K(x,y,t)$ has $y$-degree less than or equal to $1$.
\end{itemize}
 \end{defi}

 \begin{rmk}
 In \cite{DreyfusHardouinRoquesSingerGenuszero}, the authors specialize the variable $t$
 as a transcendental complex number. Then, they study the kernel curve as a complex algebraic curve in $\P1(\C) \times \P1(\C)$. In this work, we shall use any algebraic geometric result
 of \cite{DreyfusHardouinRoquesSingerGenuszero} by appealing to Lefschetz Principle:  every true statement  about  an algebraic variety defined  over  $\C$  remains true when $\C$ is replaced by an algebraically closed field of characteristic zero.
 \end{rmk}
The following proposition gives very simple conditions on $\mathcal{D}$ to decide whether  a weighted model is degenerate or not.

 \begin{prop}[Lemma 2.3.2 in  \cite{FIM}] \label{prop:degeneratecases}
A weighted model is {degenerate} if and only if at least one of the following holds:
\begin{enumerate}
\item \label{case1}There exists $i\in \{- 1,1\}$ such that $d_{i,-1}=d_{i,0}=d_{i,1}=0$. This corresponds to walks with steps supported in one of the following configurations
$$\begin{tikzpicture}[scale=.4, baseline=(current bounding box.center)]
\foreach \x in {-1,0,1} \foreach \y in {-1,0,1} \fill(\x,\y) circle[radius=2pt];
\draw[thick,->](0,0)--(0,-1);
\draw[thick,->](0,0)--(1,-1);
\draw[thick,->](0,0)--(1,0);
\draw[thick,->](0,0)--(1,1);
\draw[thick,->](0,0)--(0,1);
\end{tikzpicture}\quad 
\begin{tikzpicture}[scale=.4, baseline=(current bounding box.center)]
\foreach \x in {-1,0,1} \foreach \y in {-1,0,1} \fill(\x,\y) circle[radius=2pt];
\draw[thick,->](0,0)--(0,-1);
\draw[thick,->](0,0)--(-1,-1);
\draw[thick,->](0,0)--(-1,0);
\draw[thick,->](0,0)--(-1,1);
\draw[thick,->](0,0)--(0,1);
\end{tikzpicture}
$$
\item \label{case2} There exists $j\in \{-1, 1\}$ such that $d_{-1,j}=d_{0,j}=d_{1,j}=0$. This corresponds to walks with steps supported in one of the following configurations
$$
\begin{tikzpicture}[scale=.4, baseline=(current bounding box.center)]
\foreach \x in {-1,0,1} \foreach \y in {-1,0,1} \fill(\x,\y) circle[radius=2pt];
\draw[thick,->](0,0)--(-1,0);
\draw[thick,->](0,0)--(-1,-1);
\draw[thick,->](0,0)--(0,-1);
\draw[thick,->](0,0)--(1,-1);
\draw[thick,->](0,0)--(1,0);
\end{tikzpicture}
 \quad
\begin{tikzpicture}[scale=.4, baseline=(current bounding box.center)]
\foreach \x in {-1,0,1} \foreach \y in {-1,0,1} \fill(\x,\y) circle[radius=2pt];
\draw[thick,->](0,0)--(-1,0);
\draw[thick,->](0,0)--(-1,1);
\draw[thick,->](0,0)--(0,1);
\draw[thick,->](0,0)--(1,1);
\draw[thick,->](0,0)--(1,0);
\end{tikzpicture} 
$$
\item \label{case3} All the weights are zero except maybe  $\{d_{1,1},d_{0,0},d_{-1,-1}\}$ or  $\{d_{-1,1},d_{0,0},d_{1,-1}\}$. This corresponds to walks with steps supported in one of the following configurations
$$\begin{tikzpicture}[scale=.4, baseline=(current bounding box.center)]
\foreach \x in {-1,0,1} \foreach \y in {-1,0,1} \fill(\x,\y) circle[radius=2pt];
\draw[thick,->](0,0)--(-1,-1);
\draw[thick,->](0,0)--(1,1);
\end{tikzpicture}
\quad
\begin{tikzpicture}[scale=.4, baseline=(current bounding box.center)]
\foreach \x in {-1,0,1} \foreach \y in {-1,0,1} \fill(\x,\y) circle[radius=2pt];
\draw[thick,->](0,0)--(1,-1);
\draw[thick,->](0,0)--(-1,1);
\end{tikzpicture}
$$
\end{enumerate}
\end{prop}
 
 Note that we only discard one dimensional problems as explained in \cite{BMM}. For all the degenerate cases, the generating function $Q(x,y,t)$ is algebraic.\par

\emph{From now on, we shall always assume that the weighted model under consideration is nondegenerate.}

To any weighted model, we attach a curve $E$, called the \emph{kernel curve}, that is defined as  the zero set in $\P1(C)\times \P1(C)$ of the following homogeneous polynomial
$$
\widetilde{K}(x_0,x_1,y_0,y_1,t)= x_0x_1y_0y_1 -t \sum_{i,j=0}^2 d_{i-1,j-1} x_0^{i} x_1^{2-i}y_0^j y_1^{2-j}={x_1^2y_1^2K\left(\frac{x_0}{x_1},\frac{y_0}{y_1},t\right)}. 
$$

 Let us write $\widetilde{K}(x_0,x_1,y_0,y_1,t)=\sum_{i,j=0}^2 A_{i,j}x_0^ix_1^{2-i}y_0^jy_1^{2-j}$ where $A_{i,j}= -t d_{i-1,j-1}$ if ${(i,j) \neq (1,1)}$ and $A_{1,1}= 1-td_{0,0}$.  The partial discriminants of $\widetilde{K}(x_0,x_1,y_0,y_1,t)$ are defined as the discriminants of the second degree homogeneous polynomials $y \mapsto \widetilde{K}(x_0,x_1,y,1,t)$ and  $x \mapsto \widetilde{K}(x,1,y_0,y_1,t)$, respectively, i.e.
$$
 \Delta_x (x_0,x_1)= \left(\sum_{i=0}^2 x_0^i x_1^{2-i} A_{i,1}\right)^2-4\left(\sum_{i=0}^2 x_0^i x_1^{2-i} A_{i,0}\right)\times \left(\sum_{i=0}^2 x_0^i x_1^{2-i} A_{i,2}\right)
$$
and
$$
 \Delta_y (y_0,y_1)= \left(\sum_{j=0}^2 y_0^jy_1^{2-j} A_{1,j}\right)^2-4\left(\sum_{j=0}^2 y_0^j y_1^{2-j} A_{0,j}\right)\times \left(\sum_{j=0}^2 y_0^j y_1^{2-j} A_{2,j}\right).
$$

Introduce
\begin{equation}
\label{eq:expression_D_0}
     \mathfrak{D}(x):=\Delta_x(x,1)=\sum_{j=0}^{4}\alpha_{j}x^{j}\quad \text{and}\quad \mathfrak{E}(y):=\Delta_y(y,1)=\sum_{j=0}^{4}\beta_{j}y^{j},
\end{equation}    
where
\begin{equation}\label{eq:alphaibetai}
\begin{array}{lll}
\alpha_4&=&\big(d_{1,0}^{2}-4d_{1,1}d_{1,-1}\big)t^{2}\\
\alpha_3&=&2t^{2}d_{1,0}d_{0,0}-2td_{1,0}-4t^{2}(d_{0,1}d_{1,-1}+d_{1,1}d_{0,-1})\\
\alpha_2&=&1+t^{2}d_{0,0}^{2}+2t^{2}d_{-1,0}d_{1,0}-4t^{2}(d_{-1,1}d_{1,-1}+d_{0,1}d_{0,-1}+d_{1,1}d_{-1,-1}) -2td_{0,0}\\
\alpha_1&=&
2t^{2}d_{-1,0}d_{0,0}-2td_{-1,0}-4t^{2}(d_{-1,1}d_{0,-1}+d_{0,1}d_{-1,-1})\\
\alpha_0&=&
\big( d_{-1,0}^{2}-4d_{-1,1}d_{-1,-1}\big) t^{2}\\ 
&&\\
\beta_4&=&\big(d_{0,1}^{2}-4d_{1,1}d_{-1,1}\big)t^{2}\\
\beta_3&=&2t^{2}d_{0,1}d_{0,0}-2td_{0,1}-4t^{2}(d_{1,0}d_{-1,1}+d_{1,1}d_{-1,0})\\
\beta_2&=&1+t^{2}d_{0,0}^{2}+2t^{2}d_{0,-1}d_{0,1}-4t^{2}(d_{1,-1}d_{-1,1}+d_{1,0}d_{-1,0}+d_{1,1}d_{-1,-1}) -2td_{0,0}\\
\beta_1&=&
2t^{2}d_{0,-1}d_{0,0}-2td_{0,-1}-4t^{2}(d_{1,-1}d_{-1,0}+d_{1,0}d_{-1,-1})\\
\beta_0&=&
\big( d_{0,-1}^{2}-4d_{1,-1}d_{-1,-1}\big) t^{2}.
\end{array}
\end{equation}

The discriminants  $\Delta_x (x_0,x_1), \Delta_y (y_0,y_1)$ are homogeneous polynomials of degree $4$.  
 Their Eisenstein invariants can be defined as follows:
 \begin{defi}[\S 2.3.5 in \cite{DuistQRT}]
 For any homogeneous polynomial of the form $${f(x_0,x_1)= a_0 x_1^4 +4a_1 x_0x_1^3 +6a_2 x_0^2 x_1^2 +4 a_3 x_0^3x_1 +a_4 x_0^4}\in C[x_0,x_1],$$ we define the Eisenstein invariants of $f(x_0,x_1)$ as \begin{itemize}
\item $D(f)= a_0a_4 +3a_2^2-4a_1a_3$ 
\item $E(f)=a_0a_3^2+a_1^2a_4-a_0a_2a_4-2a_1a_2a_3+a_2^3$
\item $F(f)= 27 E(f)^2-D(f)^3$.
\end{itemize} 
 \end{defi}

 Since $C$ is algebraically closed  of characteristic zero, we can apply \cite[\S 2.4]{DuistQRT} to the kernel curve. The following proposition characterizes the smoothness  of the kernel curve  in terms of the    invariants $F(\Delta_x)$, $F(\Delta_y)$.
 
 \begin{prop}[Proposition 2.4.3 in \cite{DuistQRT} and Proposition 2.1 in \cite {DreyfusHardouinRoquesSingerGenuszero2}]\label{prop:genusofthe Kerneljinvariant}
 	The following statements are equivalent 
 	\begin{itemize}
 	\item The  kernel curve $E$ is smooth, i.e. it has no singular point;
 	\item $F(\Delta_x) \neq 0$;
 	\item $F(\Delta_y) \neq 0$.
	
\end{itemize} 	 
 Furthermore, if   $E$ is smooth then it is an elliptic curve with  $J$-invariant given by the element $J(E)\in C$ such that $$J(E)=12^3 \frac{D(\Delta_y)^3}{-F(\Delta_y)}.$$ 
Otherwise, if $E$ is nondegenerate and singular,  $E$ has a unique singular point and is a genus zero curve.
 \end{prop}

We define the genus of a weighted model  as the genus of the associated kernel curve $E$.  We recall the results obtained in \cite[Theorem 6.1.1]{FIM} and \cite[Corollary~2.6]{DreyfusHardouinRoquesSingerGenuszero2},  that classify  all the weighted models  attached to a genus zero kernel.
 
 \begin{thm}\label{theo:genuszerocharac} Any nondegenerate  weighted model of genus zero has steps included in one of the following 4 sets of steps:
$$ \begin{tikzpicture}[scale=.4, baseline=(current bounding box.center)]
\foreach \x in {-1,0,1} \foreach \y in {-1,0,1} \fill(\x,\y) circle[radius=2pt];
\draw[thick,->](0,0)--(-1,1);
\draw[thick,->](0,0)--(0,1);
\draw[thick,->](0,0)--(1,1);
\draw[thick,->](0,0)--(1,0);
%\draw[thick,->](0,0)--(1,0);
%\draw[thick,->](0,0)--(-1,-1);
\draw[thick,->](0,0)--(1,-1);
%\draw[thick,->](0,0)--(1,-1);
\end{tikzpicture}\quad 
 \begin{tikzpicture}[scale=.4, baseline=(current bounding box.center)]
\foreach \x in {-1,0,1} \foreach \y in {-1,0,1} \fill(\x,\y) circle[radius=2pt];
%\draw[thick,->](0,0)--(-1,1);
%\draw[thick,->](0,0)--(0,1);
\draw[thick,->](0,0)--(1,1);
%\draw[thick,->](0,0)--(-1,0);
\draw[thick,->](0,0)--(1,0);
\draw[thick,->](0,0)--(-1,-1);
\draw[thick,->](0,0)--(0,-1);
\draw[thick,->](0,0)--(1,-1);
\end{tikzpicture}\quad\begin{tikzpicture}[scale=.4, baseline=(current bounding box.center)]
\foreach \x in {-1,0,1} \foreach \y in {-1,0,1} \fill(\x,\y) circle[radius=2pt];
\draw[thick,->](0,0)--(-1,1);
%\draw[thick,->](0,0)--(0,1);
\draw[thick,->](0,0)--(1,1);
\draw[thick,->](0,0)--(-1,0);
\draw[thick,->](0,0)--(0,1);
\draw[thick,->](0,0)--(-1,-1);
%\draw[thick,->](0,0)--(0,-1);
%\draw[thick,->](0,0)--(1,-1);
\end{tikzpicture}\quad\begin{tikzpicture}[scale=.4, baseline=(current bounding box.center)]
\foreach \x in {-1,0,1} \foreach \y in {-1,0,1} \fill(\x,\y) circle[radius=2pt];
\draw[thick,->](0,0)--(-1,1);
%\draw[thick,->](0,0)--(0,1);
%\draw[thick,->](0,0)--(1,1);
\draw[thick,->](0,0)--(-1,0);
%\draw[thick,->](0,0)--(1,0);
\draw[thick,->](0,0)--(-1,-1);
\draw[thick,->](0,0)--(0,-1);
\draw[thick,->](0,0)--(1,-1);
\end{tikzpicture}\quad$$
Otherwise, for any other nondegenerate  weighted model, the kernel curve $E$ is an elliptic curve. 
\end{thm}

\begin{rmk}\label{remgenre0}
The walks corresponding to the fourth configuration never enter the quarter-plane.  As described in \cite[Section 2.1]{BMM}, if we consider walks corresponding to the second and third configurations we are in the situation where one of the quarter plane constraints implies the other. In the last three configurations, the generating function is algebraic. So the only interesting  nondegenerate genus zero weighted models have steps included in
$$ \begin{tikzpicture}[scale=.4, baseline=(current bounding box.center)]
\foreach \x in {-1,0,1} \foreach \y in {-1,0,1} \fill(\x,\y) circle[radius=2pt];
\draw[thick,->](0,0)--(-1,1);
\draw[thick,->](0,0)--(0,1);
\draw[thick,->](0,0)--(1,1);
\draw[thick,->](0,0)--(1,0);
%\draw[thick,->](0,0)--(1,0);
%\draw[thick,->](0,0)--(-1,-1);
\draw[thick,->](0,0)--(1,-1);
%\draw[thick,->](0,0)--(1,-1);
\end{tikzpicture} $$
Note that due to Proposition \ref{prop:degeneratecases}, the anti-diagonal steps  have nonzero attached weights.

Moreover, by Theorem \ref{theo:genuszerocharac}, combined with Proposition \ref{prop:degeneratecases}, the nondegenerate weighted models of genus one are the walks where there are no three consecutive  cardinal directions with weight zero. Or equivalently, this corresponds to the situation where the  set of steps is not included in any half plane (See~\eqref{G0} below).
\end{rmk} 
 
Thanks to Theorem \ref{theo:genuszerocharac}, one can  reduce our study to two cases depending on the genus of the kernel curve attached to 
 a nondegenerate weighted model. The following lemma proves that when the kernel curve is of genus one, its  $J$-invariant  has modulus strictly greater than $1$. This property 
 allows us   to use the theory of Tate curves in order to analytically uniformize the kernel curve.

 \begin{lemma}\label{lemma:jinvKernel}
When $E$ is smooth, the invariant $J(E)$ belongs to  $\Q(t)$  and is such that $|J(E)|>1$, where  $|~|$ denotes the norm of $(C, |~|)$.
 \end{lemma}
  \begin{proof}
At $t=0$, $\Delta_y (y_0,y_1)$ reduces  to $y_{0}^{2}y_{1}^{2}$.  This proves that the reduction of $D(\Delta_y)$ (resp. $E(\Delta_y)$)  at $t=0$ is $\frac{1}{12}$ (resp. $\frac{1}{6^{3}}$). One concludes that  $F(\Delta_y)$ vanishes for $t=0$.
By Proposition~\ref{prop:genusofthe Kerneljinvariant}, $J(E) \in \Q(t)$ has a strictly negative valuation at $t=0$. Thus, $|J(E)|>1$. 
 \end{proof}

\subsection{The automorphism of the walk}\label{sec:autoofthewalks}

Following \cite[Section 3]{BMM} or \cite[Section 3]{KauersYatchak}, we introduce the involutive birational transformations of $\P1(C)\times \P1(C)$ given by 
$$
i_1(x,y) =\left(x, \frac{A_{-1}(x) }{A_{1}(x)y}\right) \text{ and }  i_2(x,y)=\left(\frac{B_{-1}(y)}{B_{1}(y)x},y\right),
$$ 
(see \S \ref{sec:notationwalk} for the significance of the $A_i,B_i$'s).

 They induce two involutive automorphisms 
$ 
\iota_{1}, \iota_{2} : E \dashrightarrow E
$ given by 
$$
\begin{array}{llll}
&\iota_1([x_0: x_1],[y_0:y_1]) &=&\left([x_0: x_1], \left[\dfrac{A_{-1}(\frac{x_{0}}{x_{1}}) }{A_{1}(\frac{x_{0}}{x_{1}})\frac{y_{0}}{y_{1}}}:1\right]\right),\\ \text{ and } & \iota_2([x_0: x_1],[y_0:y_1])&=&\left(\left[\dfrac{B_{-1}(\frac{y_{0}}{y_{1}})}{B_{1}(\frac{y_{0}}{y_{1}})\frac{x_{0}}{x_{1}}}:1\right],[y_0:y_1]\right).
\end{array}
$$

Note that $\iota_{1}$ and $\iota_{2}$ are nothing but the vertical and horizontal  switches of $E$, see Figure \ref{figiota}. That is,  for any $P=(x,y) \in E$, we have 
$$
\{P,\iota_1(P)\} = E \cap (\{x\} \times \P1(C))
\text{ and }
\{P,\iota_2(P)\} = E \cap (\P1(C) \times \{y\}).
$$

\begin{figure}[h]
\begin{center}
\includegraphics[scale=0.7]{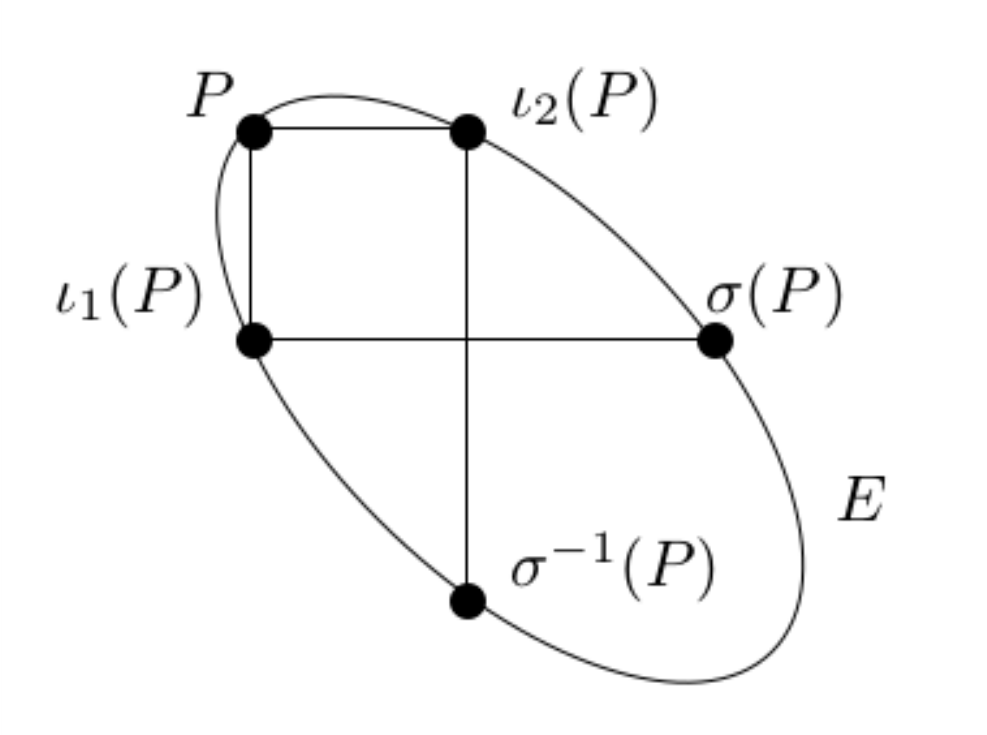}
%\begin{center}
%\begin{tikzpicture}
%\draw[rotate=-45]  (0,0) ellipse (2cm and 1cm);
    
%\put(-40,40){\line(0,-1){32.5}}
%\put(-40,40){\line(1,0){32.5}}
%\put(-40,7.5){\line(1,0){71}}
%\put(-6.5,40){\line(0,-1){72}}

%\put(-43,36){{\LARGE$\bullet$}}
%\put(-43,4){{\LARGE$\bullet$}}
%\put(-10,36){{\LARGE$\bullet$}}
%\put(-10,-36){{\LARGE$\bullet$}}
%\put(28,4){{\LARGE$\bullet$}}

%\put(-50,42){{$P$}}
%\put(0,42){{$\iota_{2}(P)$}}
%\put(-70,12){{$\iota_{1}(P)$}}
%\put(30,12){{$\sigma(P)$}}
%\put(0,-30){{$\sigma^{-1}(P)$}}
%\put(50,-20){{$E $}}
%\end{tikzpicture}
\caption{The maps $\iota_{1},\iota_{2}$ restricted to the kernel curve $E$}\label{figiota}
\end{center}
\end{figure}

The automorphism of the walk $\sigma$ is defined
 by 
$$
\sigma=\iota_2 \circ \iota_1. 
$$ 

The following  holds.

\begin{lemma}[Lemma 3.3 in   \cite{DreyfusHardouinRoquesSingerGenuszero2}]\label{lem:genus1nofixedpoint}
Let $P \in E$. The following statements are equivalent: 
\begin{itemize}
\item $P$ is fixed by $\sigma$;
\item $P$ is fixed by $\iota_1$ and $\iota_2$;
\item $P$ is the only singular point of $E$, and  $E$ is of genus zero. 
\end{itemize}
\end{lemma}

%%%%%%%%%%%%%%%%%%%%%%%%%%%%%%%%%%%%%%%%%%%%%%%%%%%%%%%%%%%%%%%%%%%%%%%%%%%%%%%%%%%%%%%
%%%%%%%%%%%%%%%%%%%%%%%%%%%%%%%%%%%%%%%%%%%%%%%%%%%%%%%%%%%%%%%%%%%%%%%%%%%%%%%%%%%%%%%%%%
%%%%%%%%%%%%%%%%%%%%%%%%%%%%%%%%%%%%%%%%%%%%%%%%%%%%%%%%%%%%%%%%%%%%%%%%%%%%%%%%%%
%%%%%%%%%%%%%%%%%%%%%%%%%%%%%%%%%%%%%%%%%%%%%%%%%%%%%%%%%%%%%%%%%%%%%%%%%%%%%%%%%

\section{Generating functions for walks, genus zero case}\label{secgenre0}
 %%%%%%%%%%%%%%%%%%%%%%%%%%%%%%%%%%%%%%%%%%%%%%%%%%%%%%%%%%%%%%%%%%%%%%%%%%%%%%%%
 %%%%%%%%%%%%%%%%%%%%%%%%%%%%%%%%%%%%%%%%%%%%%%%%%%%%%%%%%%%%%%%%%%%%%%%%%%%%%%%

 In this section, we fix a  nondegenerate weighted model of genus zero. Following Remark~\ref{remgenre0}, after eliminating duplications of trivial cases and  the interchange of $x$ and $y$, we should focus  on  walks $\cW$  arising  from the following 5 sets of steps:

\begin{equation}\label{G0}
\begin{tikzpicture}[scale=.4, baseline=(current bounding box.center)]
\foreach \x in {-1,0,1} \foreach \y in {-1,0,1} \fill(\x,\y) circle[radius=2pt];
\draw[thick,->](0,0)--(-1,1);
\draw[thick,->](0,0)--(0,1);
%\draw[thick,->](0,0)--(1,1);
%\draw[thick,->](0,0)--(-1,0);
%\draw[thick,->](0,0)--(1,0);
%\draw[thick,->](0,0)--(-1,-1);
%\draw[thick,->](0,0)--(0,-1);
\draw[thick,->](0,0)--(1,-1);
\end{tikzpicture}\quad 
\begin{tikzpicture}[scale=.4, baseline=(current bounding box.center)]
\foreach \x in {-1,0,1} \foreach \y in {-1,0,1} \fill(\x,\y) circle[radius=2pt];
\draw[thick,->](0,0)--(-1,1);
%\draw[thick,->](0,0)--(0,1);
\draw[thick,->](0,0)--(1,1);
%\draw[thick,->](0,0)--(-1,0);
%\draw[thick,->](0,0)--(1,0);
%\draw[thick,->](0,0)--(-1,-1);
%\draw[thick,->](0,0)--(0,-1);
\draw[thick,->](0,0)--(1,-1);
\end{tikzpicture}\quad \begin{tikzpicture}[scale=.4, baseline=(current bounding box.center)]
\foreach \x in {-1,0,1} \foreach \y in {-1,0,1} \fill(\x,\y) circle[radius=2pt];
\draw[thick,->](0,0)--(-1,1);
\draw[thick,->](0,0)--(0,1);
\draw[thick,->](0,0)--(1,1);
%\draw[thick,->](0,0)--(-1,0);
%\draw[thick,->](0,0)--(1,0);
%\draw[thick,->](0,0)--(-1,-1);
%\draw[thick,->](0,0)--(0,-1);
\draw[thick,->](0,0)--(1,-1);
\end{tikzpicture}\quad
\begin{tikzpicture}[scale=.4, baseline=(current bounding box.center)]
\foreach \x in {-1,0,1} \foreach \y in {-1,0,1} \fill(\x,\y) circle[radius=2pt];
\draw[thick,->](0,0)--(-1,1);
\draw[thick,->](0,0)--(0,1);
\draw[thick,->](0,0)--(1,1);
%\draw[thick,->](0,0)--(-1,0);
\draw[thick,->](0,0)--(1,0);
%\draw[thick,->](0,0)--(-1,-1);
%\draw[thick,->](0,0)--(0,-1);
\draw[thick,->](0,0)--(1,-1);
\end{tikzpicture}\quad 
\begin{tikzpicture}[scale=.4, baseline=(current bounding box.center)]
\foreach \x in {-1,0,1} \foreach \y in {-1,0,1} \fill(\x,\y) circle[radius=2pt];
\draw[thick,->](0,0)--(-1,1);
\draw[thick,->](0,0)--(0,1);
%\draw[thick,->](0,0)--(1,1);
%\draw[thick,->](0,0)--(-1,0);
\draw[thick,->](0,0)--(1,0);
%\draw[thick,->](0,0)--(-1,-1);
%\draw[thick,->](0,0)--(0,-1);
\draw[thick,->](0,0)--(1,-1);
\end{tikzpicture} \tag{G0}
\end{equation} 

A function $f(x,y,t)\in \Q[[x,y,t]]$ is $(\frac{d}{dx},\frac{d}{dt})$-differentially algebraic over $\Q$ if there exists  a nonzero polynomial $P$ with coefficients in $\Q$ such that $  P(f(x,y,t), \frac{d}{dx} f(x,y,t),\frac{d}{dt} f(x,y,t)   , \dots)=0$. The function  $f(x,y,t)$ is $(\frac{d}{dx},\frac{d}{dt})$-differentially transcendental over $\Q$ otherwise. Note that if $f$ is $\frac{d}{dt}$-differentially algebraic over $\Q$ then it is $(\frac{d}{dx},\frac{d}{dt})$-differentially algebraic over $\Q$. We define similarly the notion of   $(\frac{d}{dy},\frac{d}{dt})$-differential algebraicity.

 In this section, we   prove the following theorem: 
\begin{thm}\label{thm:genre0}
For any   weighted model  listed in  \eqref{G0}, the generating function $Q(x,0,t)$ is $(\frac{d}{dx},\frac{d}{dt})$-differentially transcendental over $\Q$.\par 
For any   weighted model  listed in \eqref{G0}, the generating function $Q(0,y,t)$ is $(\frac{d}{dy},\frac{d}{dt})$-differentially transcendental over $\Q$.
\end{thm}

 Theorem  \ref{thm:genre0} implies the $\frac{d}{dt}$-differential transcendence of the complete generating function.
\begin{cor}\label{cor3}
For any  weighted model  listed in \eqref{G0}, the generating function $Q(x,y,t)$ is $(\frac{d}{dx},\frac{d}{dt})$ and  $(\frac{d}{dy},\frac{d}{dt})$-differentially transcendental over $\Q$. Therefore, $Q(x,y,t)$ is $\frac{d}{dt}$-differentially transcendental over $\Q$.
\end{cor} 

\begin{proof}[Proof of Corollary \ref{cor3}]
Suppose to the contrary  that $Q(x,y,t)$ is $(\frac{d}{dx},\frac{d}{dt})$-algebraic over $\Q$. Let $P$  be a nonzero polynomial with coefficients in $\Q$ such that $  P(Q(x,y,t), \frac{d}{dx} Q(x,y,t),\frac{d}{dt} Q(x,y,t)   , \dots)=0$. Specializing  at $y=0$ this  relation and noting that $\frac{d ^i}{dx^i}\frac{d^j}{dt^j}(Q(x,0,t))$ is the specialization of $\frac{d ^i}{dx^i}\frac{d^j}{dt^j}(Q(x,y,t))$, one finds  a nontrivial 
differential algebraic relations for $Q(x,0,t)$ in the derivatives $ \frac{d}{dx}$ and $\frac{d}{dt}$. This contradicts   Theorem \ref{thm:genre0}. The proof for the $(\frac{d}{dy},\frac{d}{dt})$-differential transcendence is similar.
\end{proof}

As detailed in the introduction, our proof has three major steps: 
\begin{itemize}
\item[Step $1$:] we attach to the incomplete generating functions $Q(x,0,t)$ and $Q(0,y,t)$  some  auxiliary functions which share the same differential behavior than the generating series but
satisfy simple $\q$-difference equations. This is done via the uniformization of the kernel curve (see \S \ref{sec41} and \S \ref{sec:uniformizationgenuszero}).
\item[Step $2$:] we apply difference Galois theory to the $\q$-difference equations satisfied by the auxiliary functions in order to relate the differential algebraicity of the incomplete generating functions to the existence of  \emph{telescoping relations}. These telescoping relations are of the form \eqref{eq9} below.
\item[Step $3$:] we prove that there is no such telescoping relation. This allows us to conclude that the generating series is  $\frac{d}{dt}$-differentially transcendental over $\Q$  (see \S \ref{sec43}).
\end{itemize}

%%%%%%%%%%%%%%%%%%%%%%%%%%%%%%%%%%%%%%%%%%%%%%%%%
%%%%%%%%%%%%%%%%%%%%%%%%%%%%%%%%%%%%%%%%%%%%%%%%
\subsection{Uniformization of the kernel curve}\label{sec41}
%%%%%%%%%%%%%%%%%%%%%%%%%%%%%%%%%%%%%%%%%%%%%%%%
%%%%%%%%%%%%%%%%%%%%%%%%%%%%%%%%%%%%%%%%%%%%%%%
With the notation of $\S\ref{sec1}$,  especially \eqref{eq:alphaibetai},  any  weighted model   listed in \eqref{G0}  satisfies $\alpha_0=\alpha_1=\beta_0=\beta_1=0$. Moreover, since the  weighted model  is nondegenerate,  one finds that the  product $d_{1,-1}d_{-1,1}$ is nonzero. Furthermore, $$-1+d_{0,0}t\pm \sqrt{(1-d_{0,0}t)^{2}-4d_{1,-1}d_{-1,1}t^{2}} \neq 0.$$

The uniformization of the kernel curve  of a weighted model  listed in \eqref{G0} is given by the following proposition.
\begin{prop}[Propositions 1.5 in \cite{DreyfusHardouinRoquesSingerGenuszero}]\label{prop:unifgenre0}
Let us consider a weighted model  listed in \eqref{G0} and let $E$ be its kernel curve. There exist $\lambda\in C^{*}$ and a parametrization $\phi :\P1(C) \rightarrow E$ with $$\phi(s)=(x(s),y(s))=
\left(\dfrac{4\a_{2}}{\sqrt{\a_{3}^{2}-4\a_{2}\a_{4}}( s +\frac{1}{s}) -2\a_{3}}, 
\dfrac{4\b_{2}}{\sqrt{\b_{3}^{2}-4\b_{2}\b_{4}}( \frac{s}{\lambda}+\frac{\lambda}{s}) -2\b_{3}}\right),$$ such that
\begin{itemize}
\item 
$\phi: \P1(C) \setminus \{ 0, \infty\} \rightarrow E \setminus \{(0,0) \}$ is a bijection and $\phi^{-1}((0,0))=\{0,\infty\}$;
\item The automorphisms $\iota_1,\iota_2,\sigma$ of $E$ induce  automorphisms $\iup_{1},\iup_{2},\sigma_\q$ of $\P1 (C)$ {\it via} $\phi$ that satisfy $
\iup_1(s)=\frac{1}{s}$, ${\iup_2 (s)= \frac{\q}{s}
}$, $\sigma_{\q} (s)=\q s$, with $\lambda^{2}=\q\in \{ \widetilde{\q},\widetilde{\q}^{-1}\}$ and $$ 
\widetilde{\q}=\dfrac{-1+d_{0,0}t-\sqrt{(1-d_{0,0}t)^{2}-4d_{1,-1}d_{-1,1}t^{2}}}{-1+d_{0,0}t+\sqrt{(1-d_{0,0}t)^{2}-4d_{1,-1}d_{-1,1}t^{2}}} \in C^{*}.
$$ Thus,  we have the commutative diagrams 
 $$
\xymatrix{
    E  \ar@{->}[r]^{\iota_k} & E  \\
    \mathbb{P}^{1}(C) \ar@{->}[u]^\phi \ar@{->}[r]_{\iup_k} & \mathbb{P}^{1}(C) \ar@{->}[u]_\phi 
  }
  \text{ and }
  \xymatrix{
    E  \ar@{->}[r]^{\sigma} & E  \\
    \mathbb{P}^{1}(C) \ar@{->}[u]^\phi \ar@{->}[r]_{\sigma_{\q}} & \mathbb{P}^{1}(C) \ar@{->}[u]_\phi 
  }
$$
\end{itemize}

\end{prop}

The following estimate on the norm of $\widetilde{\q}$ holds:

\begin{lemma}\label{lemma:boundsonnorm}
We have  $|\widetilde{\q}| >1$.
\end{lemma}
\begin{proof}
We consider the expansion as a Puiseux series of $\widetilde{\q}$. It is then easily seen that its valuation is negative, which gives  $|\widetilde{\q}|>1$.
\end{proof}

\subsection{Meromorphic continuation of the generating functions}\label{sec:uniformizationgenuszero}

In this paragraph,  we combine the functional equation \eqref{eq:fundamentalkernelequationseries} with the uniformization of the kernel curve obtained 
above to meromorphically continue the generating function. \par 
We define the norm of an element $b=[b_0:b_1] \in \P1(C)$ as follows: if $b_1 \neq 0$, we set $|b|=|\frac{b_0}{b_1}|$ and $|[1:0]|=\infty$ by convention.
Since  $|t|<1$, 
 the generating function $Q(x,y,t)$  as well as $F^{1}(x,t),F^{2}(y,t)$ converge for  any  $(x,y) \in   \P1(C) \times \P1(C)$ such that 
$|x|$ and $|y|$ are smaller than or equal to $1$. On that domain, they satisfy
\begin{equation}\label{eq:funceqgenuszero}
K(x,y,t)Q(x,y,t)=xy+{F}^{1}(x,t) +{F}^{2}(y,t)+td_{-1,-1} Q(0,0,t).
\end{equation}

We claim that there exist two  positive real numbers $c_0,c_{\infty}$ such that $\phi$ maps the disks ${U_0=\{ s \in \P1(C) | |s|< c_0\}}$ and $U_\infty=\{ s \in \P1(C) | |s| > c_{\infty} \}$  into   the domain $\cU$ defined by  $\{ 
(x,y) \in E \mbox{ such that } |x| \leq 1 \mbox{ and }|y|\leq 1 \}$. Indeed,
 the $\a_i$ and $\b_i$ are of norm smaller than or equal to $1$ and $|\a_2|=1$ (see \eqref{eq:alphaibetai}). Thus, if  ${|s|< min(1, |\sqrt{\a_{3}^{2}-4\a_{2}\a_{4}}|)}$, then 
$$|x(s)|=\left|\dfrac{4\a_{2}s}{\sqrt{\a_{3}^{2}-4\a_{2}\a_{4}}( s^2 +1) -2\a_{3}s}\right|=\frac{|4\alpha_2 s |}{|\sqrt{\a_{3}^{2}-4\a_{2}\a_{4}}| } < 1.$$
An analogous reasoning  for $y(s)$ shows that when $|s|$ is sufficiently small,  we find  ${|x(s)|,|y(s)| \leq 1}$. Similarly, one can  prove that, when $|s|$ is sufficiently big, one has  ${|x(s)|,|y(s)| \leq 1}$.   This proves our claim.

We set $\breve{F}^{1}(s)=F^{1}(x(s),t)$ and $\breve{F}^{2}(s)=F^{2}(y(s),t)$. Based on  the above, these functions are  well  defined on $U_0 \cup U_\infty$. Evaluating \eqref{eq:funceqgenuszero} for $(x,y)=(x(s),y(s))$, one finds
\begin{equation}\label{eq:specializationfuncequkernelgenus0}
0=x(s)y(s)+\breve{F}^{1}(s) +\breve{F}^{2}(s)+td_{-1,-1} Q(0,0,t).
\end{equation}

The following lemma shows that one can use the above equation to meromorphically continue
the functions $\breve{F}^{i}(s)$  so that they satisfy a $\q$-difference equation.

\begin{lemma}\label{lemma:Analyticcontinuationandfuncequ}
For $i=1,2$, the restriction of the function $\breve{F}^i(s)$ to $U_{0}$ can be continued to a meromorphic function $\widetilde{F}^i(s)$ on $C$ such that
$$
\widetilde{F}^1(\q s)-\widetilde{F}^1(s)= b_1=(x(\q s) -x( s))y(\q s) 
$$
and 
$$
\widetilde{F}^2(\q s)-\widetilde{F}^2(s)=b_2=(y( \q s) -y(s))x(s).
$$

\end{lemma}
\begin{proof}
We just give a sketch of a proof  since  the arguments are the exact analogue in our ultrametric context  of those  employed in \cite[\S 2.1]{DreyfusHardouinRoquesSingerGenuszero}.
Since $\iup_1 (s)=\frac{1}{s}$ and $\iup_2(s)=\frac{\q}{s}$, we can assume  without  loss of generality that $\iup_1(U_0) \subset U_\infty$ and $\iup_2(U_\infty) \subset U_0$. 
Then one can evaluate  \eqref{eq:specializationfuncequkernelgenus0} at any $s \in U_0 $. We obtain 
$$
0=x(s)y(s)+\breve{F}^{1}(s) +\breve{F}^{2}(s)+td_{-1,-1} Q(0,0,t). 
$$
Evaluating \eqref{eq:specializationfuncequkernelgenus0} at $\iup_1(s) \in U_\infty$, we find 
$$
0=x(\iup_1( s))y(\iup_1( s))+\breve{F}^{1}(\iup_1( s)) +\breve{F}^{2}(\iup_1(s))+td_{-1,-1} Q(0,0,t).$$
Using the invariance of $x(s)$ (resp. $y(s)$) with respect to $\iup_1$ (resp. $\iup_2$), the second equation is 
$$
0=x(s)y(\q s)+\breve{F}^{1}(s) +\breve{F}^{2}(\q s)+td_{-1,-1} Q(0,0,t).$$
Subtracting this last equation to the first, we find that, for any $s \in U_0$, we have  
\begin{equation}\label{eq:funceqonopenset}
\breve{F}^{2}(\q s)-\breve{F}^2(s)= (y(\q s) -y(s))x(s).
\end{equation} By Lemma \ref{lemma:boundsonnorm}, the norm of $\widetilde{\q}$
is strictly greater than one  and therefore  the norm of $|\q|$ is distinct from $1$. This allows us 
to use \eqref{eq:funceqonopenset} to meromorphically continue $\breve{F}^{2}$  to  $C$
so that it satisfies \eqref{eq:funceqonopenset} everywhere. The proof for $\breve{F}^{1}$ is similar.
\end{proof}
Note that,  for $i=1,2$,   the function $\widetilde{F}^i(s)$   does not coincide  a priori with $\breve{F}^i(s)$ in the neighborhood of infinity. 

%%%%%%%%%%%%%%%%%%%%%%%%%%%%%%%%%%%%%%%%%%%%

\subsection{Differential transcendence in the genus zero case}\label{sec43}
%%%%%%%%%%%%%%%%%%%%%%%%%%%%%%%%%%%%%%%%
%%%%%%%%%%%%%%%%%%%%%%%%%%%%%%%%%%%%%%%%%

We recall that any holomorphic function $f$ on $C^*$ can be represented as an everywhere convergent Laurent series with coefficients in $C$, see \cite[Theorem 2.1, Chapter 5]{lang2013complex}. Moreover any nonzero meromorphic function on $C^*$ can be written as the quotient of two holomorphic functions on $C^*$ with no common zeros.  We denote by $\cM er(C^*)$ the field of meromorphic functions over $C^*$ and  by $\s_\q$ the $\q$-difference operator that maps a meromorphic function $g(s)$ onto $g(\q s)$. Finally, let  $C_\q$ be the 
the field  formed by the  meromorphic functions over $C^*$  fixed by $\s_\q$.\par 
We now define the $\q$-logarithm. If $|\q|>1$,  the Jacobi Theta function is the meromorphic function defined by 
${\theta_{\q}(s)=\sum_{n\in\Z}{\q}^{-n(n+1)/2}s^n\in \cM er(C^*)}$. It satisfies the the $\q$-difference equation
$$
\theta_{\q}(\q s)=s\theta_\q(s).
$$
Its logarithmic derivative $\ell_{\q}(s)=\frac{\partial_s(\theta_{\q})}{\theta_{\q}}\in \cM er(C^*)$ satisfies
 $\ell_\q(\q s)=\ell_\q( s)+1$. If $|\q|<1$ then the meromorphic function  $-\ell_{1/\q}$ is solution of $\s_\q (-\ell_{1/\q})=-\ell_{1/\q}+1$.  Abusing the  notation, we still denote by $\ell_\q$ the  function $-\ell_{1/\q}$  when $|\q|<1$.  \par 
Since we want to use the $\q$-difference equations of Lemma \ref{lemma:Analyticcontinuationandfuncequ} as a constraint
   for the form of the differential algebraic relations satisfied by the functions  $\widetilde{F}^i(s)$, we need to consider
   derivations that are compatible with $\sigma_\q$ in the sense that they commute with $\sigma_\q$. This is not the case for the derivation $\partial_t=t\frac{d}{d t}$. By Lemma~\ref{lemma:goodderivationgenus1}, the derivations $\partial_s=s \frac{d}{ds}$ and $\Delta_{t,\q}=\partial_t(\q)\ell_\q(s)\partial_s +\partial_t$ commute
with $\s_\q$. The  following lemma relates the  differential transcendence of the incomplete generating functions  $Q(x,0,t)$ and $Q(0,y,t)$ 
to the differential transcendence	of the auxiliary functions $\widetilde{F}^i (s)$. We refer to Definition~\ref{defi:diffalgtwoderivations} for the notion of $\left(\partial_s,\Delta_{t,\q}\right)$-differential algebraicity over a field.

\begin{lemma}\label{lem3}
If  the generating function  $Q(x,0,t)$ is $\left(\frac{d}{dx},\frac{d}{dt}\right)$-differentially algebraic over $\Q$, then $\widetilde{F}^1(s)$ is $\left(\partial_s,\Delta_{t,\q}\right)$-differentially algebraic over $\widetilde{K}=C_{\q}(s, \ell_{\q}(s))$.\par  If  the generating function is $Q(0,y,t)$ is $\left(\frac{d}{dy},\frac{d}{dt}\right)$-differentially algebraic over $\Q$, then $\widetilde{F}^2(s)$ is $\left(\partial_s,\Delta_{t,\q}\right)$-differentially algebraic over $\widetilde{K}=C_{\q}(s, \ell_{\q}(s))$.
\end{lemma}
\begin{proof}
The statement being symmetrical in $x$ and $y$, we  prove it only for  $Q(x,0,t)$.
Assume  that the generating function is $Q(x,0,t)$ is $\left(\frac{d}{dx},\frac{d}{dt}\right)$-differentially algebraic over $\Q$.  Since $F^1(x,t)$ is the product of $Q(x,0,t)$ by the polynomial $K(x,0,t)\in \Q[x,t]$, the function $F^1(x,t)$  is $\left(\frac{d}{dx},\frac{d}{dt}\right)$-differentially algebraic over $\Q$. It is therefore $\left(\frac{d}{dx},\partial_{t}\right)$-differentially algebraic over $\Q(t)$, and finally $\left(\frac{d}{dx},\partial_{t}\right)$-differentially algebraic over $\Q$, since $t$ is $\partial_{t}$-differentially algebraic over $\Q$.  Remember that  $\widetilde{F}^1(s)$ coincides with  $F^1(x(s),t)$ for $s \in U_0$ where $x(s)$ is defined thanks to  Proposition~\ref{prop:unifgenre0}. Therefore, we need  to understand the relations between the $x$ and $t$ derivatives of $F^1(x,t)$ and the derivatives of 
$F^1(x(s),t)$  with respect to $\partial_s$ and $\Delta_{t,\q}$.

 Let us study these relations for an arbitrary  bivariate function  $G(x,t)$ which   converges on $|x|,|y| \leq 1$. Denote   by  $\delta_x$  the derivation $\frac{d}{dx}$ and  by $\widetilde{G}(s)= G(x(s),t)$.  From the equality $(\partial_s \widetilde{G}(s))=\partial_{s}(x(s))(\delta_x G)(x(s),t)$, we conclude  that 
$$\partial_t( \widetilde{G}(s))=(\partial_t G)(x(s),t)+\partial_t(x(s))(\delta_x G)(x(s),t)=
(\partial_t G)(x(s),t)+  c \partial_s(\widetilde{G}(s)),$$ where $c =\frac{\partial_t(x(s))}{\partial_s(x(s))} $. The element $c$ belongs to $\widetilde{K}$ because $x(s) \in \widetilde{K}$ and $\widetilde{K}$ is stable by $\partial_s, \Delta_{t,\q}$ and  thereby by $\partial_t= \Delta_{t,\q}- \partial_t(\q)\ell_{\q}(s)\partial_s$, see  Lemma~\ref{lemma:fielddefinitiongenus1}. An easy induction shows that 
\begin{equation}\label{eq:derivitpartialt}
(\partial_t^n G)(x(s),t)= \partial_t^n(\widetilde{G}(s)) + \sum_{i \leq n ,j <n} b_{i,j} \partial_t^j\partial_s^i(\widetilde{G}(s)),
\end{equation}
where the $b_{i,j}$'s belong to  $\widetilde{K}$. By Lemma \ref{lemma:goodderivationgenus1}, we have $\partial_s \Delta_{t,\q}-\Delta_{t,\q}\partial_s=f\partial_s$,  where  ${f=\partial_t(\q) \partial_s(\ell_\q)}\in\widetilde{K}$.
Combining \eqref{eq:derivitpartialt} with  $\partial_t= \Delta_{t,\q}- \partial_t(\q)\ell_{\q}(s)\partial_s$, we find that 
\begin{equation}\label{eq:derivpartialtfonctiondeltatq}
(\partial_t^n G)(x(s),t)= \Delta_{t,\q}^n(\widetilde{G}(s)) + \sum_{i \leq 2n,j <n} d_{i,j} \Delta_{t,\q}^j\partial_s^i(\widetilde{G}(s)),
\end{equation}
for some  $d_{i,j}$'s  in $\widetilde{K}$.
Moreover, an easy induction shows that, for any $m \in \N^*$, we have  
\begin{equation}\label{eq:xderivationsderivation}
(\delta^m_x G)(x(s),t)= \frac{1}{\partial_s(x(s))^m}\partial_s^m(\widetilde{G}(s)) + \sum_{i=1}^{m-1} a_i \partial_s^i(\widetilde{G}(s)),
\end{equation}
where $a_i \in \widetilde{K}$.
Applying \eqref{eq:derivpartialtfonctiondeltatq} with $G$ replaced by $\delta^m_x G$, we find that  for every $m,n\in \N$,
$$
(\partial_t^n \delta^{m}_x G)(x(s),t)=\Delta_{t,\q}^n((\delta^m_x G) (x(s),t)  ) + \sum_{i \leq 2n,j <n} d_{i,j} \Delta_{t,\q}^j\partial_s^i((\delta^m_x G )(x(s),t) ). 
$$
Combining this equation with \eqref{eq:xderivationsderivation}, we conclude that
$$
(\partial_t^n \delta^{m}_x G)(x(s),t)=
 \frac{1}{\partial_s(x(s))^m} \Delta_{t,\q}^n \partial_s^m (\widetilde{G}(s)) + \sum_{i \leq 2n +m,j <n} r_{i,j} \Delta_{t,\q}^j\partial_s^i(\widetilde{G}(s)),
$$
where the $r_{i,j}$'s are elements of  $\widetilde{K}$.

 Applying the computations above to $G =F^1(x,t)$, we find that  any nontrivial polynomial equation in the derivatives $\delta^{m}_x\partial_t^n F^1(x,t)$ over $\Q$ yields to a nontrivial polynomial equation over $\widetilde{K}$ between the derivatives  $ \Delta_{t,\q}^j\partial_s^i (\widetilde{F}^1(s))$. \end{proof}

Thus, we have reduced the proof of Theorem \ref{thm:genre0} to the following proposition:

\begin{prop}\label{prop:Galoiscriteriagenus0}
The functions $\widetilde{F}^1(s)$  and $\widetilde{F}^2(s)$ are $\left(\partial_s,\Delta_{t,\q}\right)$-differentially transcendental over $\widetilde{K}$.
\end{prop}
\begin{proof}

Suppose to the contrary that  $\widetilde{F}^1(s)$ is $\left(\partial_s,\Delta_{t,\q}\right)$-differentially algebraic over $\widetilde{K}$. By Lemma~\ref{lemma:Analyticcontinuationandfuncequ}, the meromorphic function $\widetilde{F}^{1}(s)$ satisfies
$\widetilde{F}^1(\q s)-\widetilde{F}^1(s)= b_1=(x( \q s) -x(s))y(\q s)$ with $b_1 \in C(s)\subset C_\q(s)$. We now apply difference Galois theory to this $\q$-difference equation. More precisely, 
 by   Proposition~\ref{prop2}  and Corollary \ref{cor1} with ${K=C_\q(s)}$,  there exist  $m \in \N$, $d_0,\dots,d_m \in C_{\q}$ not all zero and $h \in C_{\q}(s)$ such that 
\begin{equation}\label{eq9}
d_0 b_1+d_1\partial_s(b_1)+\dots +d_m \partial_s^m(b_1)=\s_{\q}(h)-h.
\end{equation}
Let   $(e_\beta)_{\beta \in B}$ be  a $C$-basis of $C(s)$. Then,  $(e_\beta)_{\beta \in B}$ is   a $C_\q$-basis of $C_\q(s)$ by \cite[Lemma~1.1.6]{wibmerthesis}.  Now, decompose  the $d_k$'s and $h$ over $(e_\beta)_{\beta \in B}$. Since $b_{1}\in C(s)$, it is easily seen that \eqref{eq9} amounts into a collection of polynomial equations  with coefficients in $C$ that should satisfy the coefficients of the $d_k$'s and $h$ with respect to the basis $(e_\beta)_{\beta \in B}$.
Since this collection  of polynomial equations has a nonzero solution in $C_\q$, we can conclude   that it has a nonzero solution in $C$ because  $C$ is algebraically closed. 
Therefore,
there exists $c_k \in C$ not all zero and $g \in C(s)$ such that 
$$
\sum_{k} c_{k}  \partial_s^{k}(b_1) = \s_{\q}(g)-g. 
$$ 
By \cite[Lemma 6.4]{HS}  there exist $f \in C(s)$ and $c\in C$, such that $$\widetilde{F}^1(\q s)-\widetilde{F}^1(s)=b_1=\s_{\q} (f)-f +c.$$ Since $\widetilde{F}^1$ is meromorphic at $s=0$, we conclude that  $c$ must be equal to zero. Finally, we have shown that there exist $f \in C(s)$ such that
\begin{equation}\label{eq:telescopgenrezero}
b_1=\s_\q(f)-f.
\end{equation}

By duality, the morphism $\phi :\P1 \rightarrow E$ gives rise to a field isomorphism $\phi^*$ from  the field $C(E)=C(x,y)$\footnote{Here  $x$ and $y$ denote the coordinate functions on the curve $E$. } of rational functions on  $E$ and the field $C(s)$  of rational functions  on $\P1$.
Moreover, one has $\sigma_\q \phi^*= \phi^* \sigma^*$, where $\sigma^*$ is the action induced by the automorphism of the walk on $C(E)$. Then, it is easily seen that the equation \eqref{eq:telescopgenrezero} is equivalent to 
\begin{equation}\label{eq:telscopequationfunctionfieldgenuszero}
(\sigma(x)-x)\sigma(y) = \sigma(\tilde{f})-\tilde{f},
\end{equation}
where $\tilde{f} \in C(x,y)$ is the rational function corresponding to $f$ via $\phi^*$.
The coefficients of $\tilde{f}$ as a rational function over $E$ belong to a finitely generated extension $F$ of $\Q(t)$. 

There exists a $\Q$-embedding $\psi$ of $F$ into $\C$  that maps $t$ onto a transcendental complex number. 
Since $\sigma$ and $E$ are defined over $\Q(t)$, we   apply $\psi$ to  \eqref{eq:telscopequationfunctionfieldgenuszero} and  we find 
$$(\overline{\sigma}(x)-x)\overline{\sigma}(y)= \overline{\sigma} (\overline{f})-\overline{f},$$
where $\overline{f} $ belongs to  $\C(\overline{E})$ the field of rational functions on the complex algebraic curve $\overline{E}$ defined by the kernel polynomial $K(x,y,\psi(t))$ and  where  $\overline{\sigma}$ is the automorphism of $\C(\overline{E})$ induced by the automorphism of the walk corresponding to $\overline{E}$.  In \cite[\S 3.2]{DreyfusHardouinRoquesSingerGenuszero}, the authors  proved that there is no such equation. This concludes the proof by contradiction.
\end{proof}

%%%%%%%%%%%%%%%%%%%%%%%%%%%%%%%%%%%%%%%%%%%%%%%%%%%%%%%%%%%%%%%%%%%%%%%%%%%%%%%%%%%%%%%%%%%
%%%%%%%%%%%%%%%%%%%%%%%%%%%%%%%%%%%%%%%%%%%%%%%%%%%%%%%%%%%%%%%%%%%%%%%%%%%%%%%%%%%%%%%%%%%%

 \section{Generating functions of walks, genus one case}\label{secgenre1}
%%%%%%%%%%%%%%%%%%%%%%%%%%%%%%%%%%%%%%%%%%%%%%%%%%%%%%%%%%%%%%%%%%%%%%%%%%%%%%%%%%%%%%%%%%%
%%%%%%%%%%%%%%%%%%%%%%%%%%%%%%%%%%%%%%%%%%%%%%%%%%%%%%%%%%%%%%%%%%%%%%%%%%%%%%%%%%%%%%%%%%%%

In this section we consider the situation where the kernel curve  $E$ is an elliptic curve.  By Remark \ref{remgenre0}, this corresponds to the case where the set of steps is not included in an  half plane.  Unlike  the genus zero cases of \eqref{G0}, the group of the walk might be finite for genus one walks. For unweighted walks of genus one  with  finite group, it was proved in \cite{BMM,BostanKauersTheCompleteGenerating} that the series was  holonomic with respect to the three variables. More recently, the authors of  \cite{dreyfus2019differential}  studied weighted walks of genus one  with  finite group. They proved that the uniformization of the generating series was a product of zeta functions and elliptic functions over curve isogeneous  to the kernel curve. This allowed them to conclude that the generating series was holonomic with respect to the variables $x$ and $y$. Their description should also allow to conclude to the $\frac{d}{dt}$-differential algebraicity of the series but in that case the question of the holonomy with respect to the variable $t$ is still open.  \\
In this section, we shall focus on the weighted walks of genus one with infinite group and we will prove analogously to the genus zero case that the $(\frac{d}{dx},\frac{d}{dt})$-differential algebraicity of the series implies its $\frac{d}{dx}$-differential algebraicity. This result combined to \cite{BBMR16} shows that, for unweighted walks of genus one with infinite group, the series is $\frac{d}{dx}$-differentiallly algebraic if and only if it is $\frac{d}{dt}$-differentially algebraic (see Corollary \ref{cor:diffalgxequivalentdiffalgtgenusone} below).\\
Our strategy  follows the basic lines of   the one employed in the genus zero situation. However, the  uniformization procedure  in  the genus one case is more delicate and  differs from previous works such as \cite{FIM,KurkRasch, dreyfus2019differential} which relied on the uniformization of elliptic curves over $\C$ by a fundamental parallelogram of periods. Over a  nonarchimedean field $C$, there might be a lack 
of nontrivial lattices. One has to consider  multiplicative analogues, that is, discrete subgroups of $C^*$ of the form  $q^\Z$. Then, rigid analytic geometry gives a geometric meaning to the quotient $C^*/q^\Z$. This geometric quotient is    called  a Tate curve (see \cite{roquette1970analytic} for more details). For simplicity of exposition, we will not give here many details on this nonarchimedean geometry
The multiplicative uniformization of the kernel curve allows us as in \S \ref{sec:uniformizationgenuszero} to   attach to the incomplete generating functions $Q(x,0,t)$ and $Q(0,y,t)$ some meromorphic functions $\widetilde{F}^i(s)$ satisfying 
$$  \widetilde{F}^i(\q s) -\widetilde{F}^i(s)=b_i(s),$$
for some $\q \in C^*$  and  $b_i(s) \in C_q$, the field of $q$-periodic meromorphic functions over $C^*$. This 
process detailed in \S \ref{sec:uniformizationgenus1}, \ref{sec52} and \ref{sec53} has many advantages. Though technical, it is much more simple than the uniformization
by a fundamental parallelogram of periods since we only have to deal with one generator
of the fundamental group of the elliptic curve, precisely the loop around the origin in $C^*$. Moreover, it gives a unified framework to study the genus zero and one case, namely, the Galois theory of $\q$-difference equations. This is the content of \S \ref{sec54} 
where we apply the Galoisian  criteria of Appendix~\ref{sec:differenceGaloistheory}
to translate the differential algebraicity of the generating function in terms of the existence of  a telescoper. 

\subsection{Uniformization  of the kernel curve}\label{sec:uniformizationgenus1}

Let us fix a  weighted model of genus one. By Lemma \ref{lemma:jinvKernel}, the norm  of the $J$-invariant $J(E)$ of  the kernel curve is such that $|J(E)|>1$. By Proposition~\ref{prop:Tateuniformization}, there exists $q \in C$ such that $0<|q|<1$ and $J(E)=J(E_q)=\frac{1}{|q|}$, where $E_q$ is the elliptic curve attached to the Tate curve $C^*/ q^\Z$ (see Proposition~\ref{prop:Tatecurve}, Lemmas~\ref{lemma:tatetoweierstrass}, and~\ref{lem2}). The curve $E_q$ can be analytically uniformized by $C^*$  thanks to special functions, which have their origins in the theory of Jacobi $q$-theta functions (see Proposition \ref{prop:Tatecurve} below). Finally, since $E$ and $E_q$ have the same $J$-invariant, there exists an algebraic isomorphism between these two elliptic curves.  In order to describe the uniformization 
of the kernel curve $E$, one needs to explicit this algebraic isomorphism. This is not completely obvious since $E_q$ is given by its Tate normal form in $\bold{P}^2$, i.e. by an equation of the form
$$Y^2+ XY =X^3 +BX +\widetilde{C}.$$

Therefore, many  intermediate technical  results are  postponed    to the appendix \ref{sec:nonarchimedianpreleminaries}. The following proposition describes the multiplicative uniformization of an elliptic curve given by a Tate normal form.

Following \cite[Page 28]{roquette1970analytic}, we   set  $s_k =\sum_{n >0}\frac{n^k q^n}{1-q^n} \in C$  for $k \geq 1$.
\begin{prop}\label{prop:Tatecurve}
 The  series
\begin{itemize}
\item $X(s)=\sum_{n \in \Z} \frac{q^n s}{(1-q^n s)^2} -2s_1$; 
\item  $Y(s)=\sum_{n \in \Z} \frac{(q^{n}s)^2}{(1-q^n s)^3} +s_1$;
\end{itemize}
are $q$-periodic meromorphic functions over $C^*$. Furthermore $X(s)=X(1/s)$, and $X(s)$ has  a pole of order $2$ at any element of the form $q^\Z$. Moreover, the analytic map  
$$
\begin{array}{llll}
\pi: &C^* &\rightarrow &\bold{P}^2(C),\\
& s& \mapsto& [X(s):Y(s):1]\end{array}$$
is onto and his image is $E_q$, the elliptic curve defined by the following Tate normal form
\begin{equation}\label{eq:tatenormalform}
Y^2+XY= X^3 +BX +\widetilde{C},
\end{equation}
where $B=-5s_3$ and $\widetilde{C} =-\frac{1}{12}(5s_3 +7s_5)$. Moreover, $\pi(s_1)=\pi(s_2)$ if and only if $s_1 \in s_2 q^\Z$.
\end{prop}
\begin{proof}
This is \cite[Theorem 5.1.4, Corollary 5.1.5, and Theorem 5.1.10]{FresnelvanderPUt}.
\end{proof}
In the notation of Section \ref{sec:Kernelcurve},  set $\mathfrak{D}(x):=\Delta_x(x,1)$.  Let us write the kernel polynomial $${K\left(x,y,t\right)=\widetilde{A}_{0}(x)+\widetilde{A}_{1}(x)y+\widetilde{A}_{2}(x)y^{2}=\widetilde{B}_{0}(y)+\widetilde{B}_{1}(y)x+\widetilde{B}_{2}(y)x^{2}}$$ with ${\widetilde{A}_{i}(x)\in C[x]}$ and $\widetilde{B}_{i}(y)\in C[y]$. For $i\geq 1$, let $ \mathfrak{D}^{(i)}(x)$ denote the $i$-th derivative 
 with respect to $x$ of $\mathfrak{D}(x)$.
The analytic uniformization of the kernel curve is given by the following proposition.   
 \begin{thm}\label{cor:jinvariant}
There exists a root  $a $ of $\mathfrak{D}(x)$ in $C$  such that $|a|, |\mathfrak{D}^{(2)}(a)-2|,|\mathfrak{D}^{(i)}(a)| <1$ for $i=3,4$,  $|q|^{1/2}<|\mathfrak{D}^{(1)}(a)|<1$. For any such $a$, there exists   $u\in C^*$ with $|u|=1$ such that the  map $\phi$ given by 
$$
\begin{array}{llll}
\phi: & C^* & \rightarrow & E, \\
 & s &  \mapsto &  (\overline{x}(s),\overline{y}(s)),
\end{array} $$
is surjective where
 \begin{eqnarray}\label{eq:parametragex}
 \overline{x}(s)&=&a +\frac{\mathfrak{D}^{(1)}(a)}{ u^2X(s)+ \frac{u^2}{12}  -\frac{\mathfrak{D}^{(2)}(a)}{6}}\\
 \nonumber 
 \overline{y}(s)&=&\frac{\frac{\mathfrak{D}^{(1)}(a)\left( 2u^{3}Y(s) +u^{3}X(s)\right)}{2\left(u^2 X(s)+\frac{u^2}{12}-\frac{\mathfrak{D}^{(1)}(a)}{6}\right)^2}- \widetilde{A}_{1}\left( a +\frac{\mathfrak{D}^{(1)}(a)}{u^2 X(s)+\frac{u^2}{12}-\frac{\mathfrak{D}^{(2)}(a)}{6}}\right)    }{2 \widetilde{A}_{2}\left(a +\frac{\mathfrak{D}^{(1)}(a)}{u^2 X(s)+\frac{u^2}{12}-\frac{\mathfrak{D}^{(2)}(a)}{6}}\right)}. 
 \end{eqnarray} 
 \end{thm}
\begin{proof}
Lemma \ref{lemma:goodroot} and Lemma \ref{lem2} guaranty the existence of $a$.
The element $a$ allows us to write down the isomorphism between the kernel curve $E$ and one of its Weierstrass normal form $E_1$. More precisely, by Proposition~\ref{prop:uniformizationnondegeneratecase},  the application  $w_E$
$$
\begin{array}{llll}
& E_1 & \rightarrow & E \subset \P1 (C) \times \P1 (C) \\
 & [x_1:y_1:1] &  \mapsto &  (\overline{x},\overline{y})
\end{array} $$
where $$\overline{x}=a +\frac{\mathfrak{D}^{(1)}(a)}{x_1-\frac{\mathfrak{D}^{(2)}(a)}{6}} \hbox{ and }\overline{y}=\frac{\frac{\mathfrak{D}^{(1)}(a) y_1}{2(x_1-\frac{\mathfrak{D}^{(1)}(a)}{6})^2}- \widetilde{A}_{1}\left( a +\frac{\mathfrak{D}^{(1)}(a)}{x_1 -\frac{\mathfrak{D}^{(2)}(a)}{6}}\right)    }{ 2\widetilde{A}_{2}\left(a +\frac{\mathfrak{D}^{(1)}(a)}{x_1-\frac{\mathfrak{D}^{(2)}(a)}{6}}\right)},$$ is an isomorphism  between the elliptic curve $E_1 \subset \bold{P}^2(C)$ given by the  equation ${y_1^2=4x_1^3 -g_2x_1-g_3}$ and the kernel curve $E$. Now, it remains to explicit the isomorphism between $E_q$ and one of its Weierstrass normal form $\widetilde{E}_1$. By Lemma \ref{lemma:tatetoweierstrass},  the application 
$\begin{array}{llll}w_{T}:& E_q & \rightarrow &\widetilde{E}_1 ,\\ &  \hbox{}   [X:Y:1]& \mapsto &[X +\frac{1}{2}: 2Y +X:1]\end{array}$
induces  an isomorphism between  $E_q$ and the curve $\widetilde{E}_1$ given by $y^2=4x^3-h_2x-h_3$. Since $E$ and $E_q$ have the same $J$-invariants and are therefore isomorphic, the same holds for their Weierstrass normal forms.  Thus,   there exists $u \in C^*$
such that $\begin{array}{llll}\psi:& \widetilde{E}_1& \rightarrow &E_1,\\ & [x:y:1] &\mapsto &[u^2 x:u^3 y:1]\end{array}$ induces an isomorphism of elliptic curves (see Lemma~\ref{lemma:uniquenessweierstrassequation}). To conclude, we set $\phi= w_E \circ \psi \circ w_T \circ \pi$ where  $\pi$ is  the uniformization of $E_q$ by $C^*$ given in  Proposition  \ref{prop:Tatecurve}. The norm estimate on $u$ is Lemma \ref{lem2}.\end{proof}

\begin{rmk}\label{rmk:parametrizationcurvegenus1}
\begin{itemize}

\item Note that by construction $\phi(s_1)=\phi(s_2)$ if and only if if $s_1 \in s_2 q^\Z$ (see Proposition \ref{prop:Tatecurve}).
\item Via $\phi$, the field of rational functions over $E$ can be identified with  field of $q$-periodic meromorphic functions over $C^*$.
\item The conditions on $a$ are crucial to guaranty the meromorphic continuation of the generating function (see the proof of Lemma \ref{lem1}).
\item The symmetry  arguments between $x$ and $y$ of Remark \ref{rem1} can be pushed further and one can   construct another  uniformization of $E$ as follows. Denoting by $\mathfrak{E}(y)$ the polynomial $\Delta_y(y,1)$. One can prove that there exist a root  $b \in C^*$ of $\mathfrak{E}$ such that $|b|, |\mathfrak{E}^{(2)}(b)-2|,|\mathfrak{E}^{(i)}(b)| <1$ for $i=3,4$ and $|q|^{1/2}<|\mathfrak{E}^{(1)}(b)|<1$ and 
 $v\in C^*$ with $|v|=1$ such that the analytic  map $\psi$  given by 
$$
\begin{array}{llll}
\psi: & C^* & \rightarrow & E, \\
 & s &  \mapsto &  (\overline{x}(s),\overline{y}(s)),
\end{array} $$
is surjective with 
$ \overline{y}(s)=b +\frac{\mathfrak{E}^{(1)}(b)}{ v^2X(s)+ \frac{v^2}{12}  -\frac{\mathfrak{E}^{(2)}(b)}{6}}$ (see \cite[(2.16)]{dreyfus2019differential} for similar arguments).
\end{itemize}
\end{rmk}

%%%%%%%%%%%%%%%%%%%%%%%%%%%%%%%%%%%%%%%%%%%%%%%%%%%%%%%
%%%%%%%%%%%%%%%%%%%%%%%%%%%%%%%%%%%%%%%%%%%%%%%%%%%%%%%
\subsection{The group of the walk}\label{sec52}
%%%%%%%%%%%%%%%%%%%%%%%%%%%%%%%%%%%%%%%%%%%%%%%%%%%%%%
%%%%%%%%%%%%%%%%%%%%%%%%%%%%%%%%%%%%%%%%%%%%%%%%%%%%%%%

The following proposition gives an explicit form for the  automorphisms of $C^*$  induced via $\phi$ by  the automorphisms $\sigma,\iota_{1},\iota_{2}$ of $E$.   
\begin{prop}\label{prop1} 
There exists  $\q$ in $C^{*}$ such that the automorphism of $C^*$ defined by  $\s_\q : s\mapsto \q s$ induces via $\phi$ the automorphism $\sigma$, that is $\sigma \circ\phi= \phi \circ \sigma_\q$.
Similarly, the involutions $\iup_{1},\iup_{2}$  of $C^*$,  that are defined by $\iup_{1}(s)=1/s$ and $\iup_{2}(s)=\q /s$, induce via $\phi$ the automorphisms  $\iota_{1},\iota_{2}$.
\end{prop}

\begin{proof}
By \cite[Proposition 2.5.2]{DuistQRT}, the automorphism $\s$ corresponds to the addition by a prescribed point $\Omega$ of $E$. Let  $\pi: C^* \rightarrow E_q$ be the  surjective map  defined in Proposition \ref{prop:Tatecurve}.  By \cite[Exercise 5.1.9]{FresnelvanderPUt}, the map $\pi$ is  a group isomorphism between the multiplicative group $(C^*,*)$  and the Mordell-Weil group of $E_q$ \footnote{ This is the group whose underlying set is the set of points of $E_q$ and whose group law is given by the addition on the elliptic curve $E$.}. Moreover, since $E_q$ and $E$ are elliptic curves, any isomorphism between	$E_q$ and $E$ is a group morphism between their respective Mordell-Weil groups.  This proves that $ \phi $ is a group morphism.  Then, there exists $\q \in C^*$ 
 such that $\sigma \circ\phi= \phi \circ \sigma_\q$. Since $\phi$ is $q$-invariant, the element $\q$ is determined modulo $q^\Z$ (see Remark \ref{rmk:parametrizationcurvegenus1}). This proves the first statement.\\
Let us denote by   $\iup_{1},\iup_{2}$ some automorphisms of $ C^*$, obtained by   pulling back to $ C^*$ via $\phi$ the automorphisms $\iota_1,\iota_2$ of $E$. The automorphisms    $\iup_{1},\iup_{2}$ are uniquely determined up to multiplication by some power of $q$. The automorphisms of $C^*$ are of the form $s\mapsto l s^{\pm 1}$ with $l \in C^*$. Note that $\overline{x}(q^{\Z})=a$, and $(a, \frac{-B(a)}{2A(a)}) \in E$ is fixed by $\iota_1$. Indeed, by construction $\mathfrak{D}(a)=0$. This proves that  $\iup_1(1)$ belongs to $q^{\Z}$. Since $\iota_1$ is not the identity, we can modify $\iup_1$ by a suitable power of $q$ to get 
 $\iup_1 (s)=1/s$. The expression of $\iup_2$ follows with $\sigma=\iota_2 \circ \iota_1$.
\end{proof}
\begin{rmk}\label{rmk:choiceforqandsymmetry}
\begin{itemize}
\item The choice of the element $\q$ is unique up to multiplication by $q^\Z$.  Since $|q| \neq 1$, we can choose $\q$ such that  $|q|^{1/2}\leq |\q|<|q|^{-1/2}$. 

\item Pursuing the symmetry arguments of  Remark \ref{rmk:parametrizationcurvegenus1}, we easily note that  Proposition~\ref{prop1} has a straightforward analogue when one replaces $\phi$ by $\psi$ and one exchanges $\iup_1$ and $\iup_2$.
\end{itemize}
\end{rmk}

The proof of the following lemma is straightforward.
\begin{lemma}\label{lemma:qqmultindinfite order}The automorphism $\sigma$ has infinite order if and only if 
$\q$ and $q$  are \emph{multiplicatively independent}\footnote{Note that multiplicatively independent is sometimes  replaced in the literature  by noncommensurable  (see  \cite[\S 6]{roquette1970analytic}).}, that is,  there is  no $(r,l) \in \Z^2 \setminus (0,0)$ such that $q^r=\q^l$.
\end{lemma}

\subsection{Meromorphic continuation}\label{sec53}

In this section, we prove that the functions $$
{{F}^{1}(x,t):= K(x,0,t)Q(x,0,t)}, \hbox{ and }\quad {F}^{2}(y,t):= K(0,y,t)Q(0,y,t),$$ can be meromorphically  continued  to $C^*$. We follow some of the ideas initiated in \cite{FIM}. We note that, since $|t|<1$, the series ${F}^{1}(x,t)$ and ${F}^{2}(y,t)$ converge on the affinoid subset  ${U=\{(x,y) \in E \subset \P1 (C)\times \P1 (C)||x|\leq 1, |y|\leq 1 \}}$ of $E$. With Lemma~\ref{lemma:nonemptyopenset}, $U$ is not empty. For $(x,y) \in U$, we have 
$$
0=xy+F^{1}(x,t) +F^{2}(y,t)+td_{-1,-1} Q(0,0,t).
$$
Set $U_{x}=\{(x,y) \in E \subset \P1 (C)\times \P1 (C)||x|\leq 1\}$.  Note that $F^{1}(x,t)$ is analytic on $U_{x}$. We continue $F^{2}(y,t)$
 on $U_x$ by setting 
  $$
F^{2}(y,t)= -xy -F^{1}(x,t) -td_{-1,-1} Q(0,0,t).
$$

Composing  $F^i(x,t)$ with the  surjective map $$
\begin{array}{llll}
\phi: & C^* & \rightarrow & E \\
 & s &  \mapsto &  (\overline{x}(s),\overline{y}(s)),
\end{array}$$ we   define the functions  $\breve{F}^{1}(s)=F^{1}(\overline{x}(s),t)$ and $\breve{F}^{2}(s)=F^{2}(\overline{y}(s),t)$ for any $s$ in the set $$\mathcal{U}_{x}:= \phi^{-1}(U_{x})\cap \{s\in C^*| |s|\in [|q|^{1/2},|q|^{-1/2}[\}.$$

The goal of the following lemma is to prove that $\mathcal{U}_x$ is an annulus 
 whose size is large enough in order to continue the functions $\breve{F}^1, \breve{F}^2$, to the whole $C^*$ (see Figure \ref{figxs}).

\begin{lemma}\label{lem1}
Let $|s|\in [|q|^{1/2},|q|^{-1/2}[$. The following statements  hold:
\begin{itemize}
\item if $|s|\in ]|\mathfrak{D}^{(1)}(a)|,|\mathfrak{D}^{(1)}(a)|^{-1}[$, then $|\overline{x}(s)|< 1$;
\item if $|s|=|\mathfrak{D}^{(1)}(a)|^{\pm 1}$, then $|\overline{x}(s)|= 1$;
\item otherwise $|\overline{x}(s)|> 1$.
\end{itemize}
In conclusion,  $\mathcal{U}_x =[|\mathfrak{D}^{(1)}(a)|,|\mathfrak{D}^{(1)}(a)|^{-1}]$.
\end{lemma}

\begin{figure}
\begin{center}
\begin{tikzpicture}
\draw (0,0) [dashed] circle (0.5 cm);
\draw (0,0)  [dashed] circle (2cm);
\draw (0,0)  circle (0.3 cm);
\draw (0,0)  circle (3cm);
\put(11,0){\line(1,0){100}}
\put(120,0){{$|\overline{x}(s)|> 1$}}
\put(-20,65){{$|\overline{x}(s)|> 1$}}
\put(-20,30){{$|\overline{x}(s)|< 1$}}
\end{tikzpicture}
\caption{The plain circles correspond to $|s|=|q|^{\pm 1/2}$. The dashed circles correspond to $|\overline{x}(s)|= 1$.}\label{figxs}
\end{center}
\end{figure}
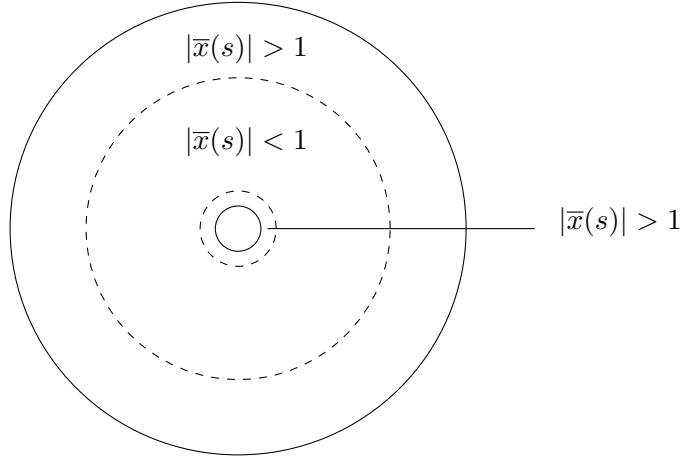
\begin{proof}
From the definition of $X(s)$, we have $X(s)=X(1/s)$ so that $\overline{x}(s)=\overline{x}(1/s)$. Using this symmetry, we just have to  prove Lemma \ref{lem1} for $|s|\in [|q|^{1/2},1]$.
We have 
 \begin{equation} \label{eq:normestimatexs}
|\overline{x}(s)| =\left|a +\frac{\mathfrak{D}^{(1)}(a)}{ u^2X(s)+ \frac{u^2}{12}  -\frac{\mathfrak{D}^{(2)}(a)}{6}}\right|\leq \max\left(|a|,\left|\frac{\mathfrak{D}^{(1)}(a)}{ u^2X(s)+ \frac{u^2}{12}  -\frac{\mathfrak{D}^{(2)}(a)}{6}}\right|\right),\end{equation} 

with equality if $|a| \neq \left|\frac{\mathfrak{D}^{(1)}(a)}{ u^2X(s)+ \frac{u^2}{12}  -\frac{\mathfrak{D}^{(2)}(a)}{6}}\right|$.
Remember  that  $|u|=1$, $|a|<1$, and  ${|q|^{1/2}<|\mathfrak{D}^{(1)}(a)|<1}$, see Theorem \ref{cor:jinvariant}.
Let us first assume that $|s|\in [|\mathfrak{D}^{(1)}(a)|,1[$. By Lemma~\ref{Xs}, $|u^{2}X(s)|=|s|$  and by Lemma~\ref{lem:normestimateu2mathfrakD},
  $|\frac{u^2}{12}  -\frac{\mathfrak{D}^{(2)}(a)}{6}|<|\mathfrak{D}^{(1)}(a)|$. Therefore   
 $$\left|\frac{\mathfrak{D}^{(1)}(a)}{ u^2X(s)+ \frac{u^2}{12}  -\frac{\mathfrak{D}^{(2)}(a)}{6}}\right|=\left|\frac{\mathfrak{D}^{(1)}(a)}{s}\right|.$$
Combining this equality with \eqref{eq:normestimatexs} and $|a|<1$, we find that $|\overline{x}(s)| <1$ if ${|s|\in ]|\mathfrak{D}^{(1)}(a)|,1[}$,  and $|\overline{x}(s)| =1$ if $|s|=|\mathfrak{D}^{(1)}(a)|$.  \par 
 Assume now that $|s|=1$. By construction,  $|\overline{x}(1)|=|a|<1$. So let us assume that $s\neq 1$. Since $|\frac{u^2}{12}  -\frac{\mathfrak{D}^{(2)}(a)}{6}|<|\mathfrak{D}^{(1)}(a)|<1$  and $ |u^{2}X(s)| \geq 1$ by Lemma \ref{Xs}, we find 
$$
\left| \frac{\mathfrak{D}^{(1)}(a)}{ u^2X(s)+ \frac{u^2}{12}  -\frac{\mathfrak{D}^{(2)}(a)}{6}} \right| =\left|\frac{\mathfrak{D}^{(1)}(a)}{u^2X(s)}\right| \leq |\mathfrak{D}^{(1)}(a)|<1
.$$ 

This concludes the proof of the first two points. \par 
Assume that $|s|\in ]|q|^{1/2},|\mathfrak{D}^{(1)}(a)|[$. By Lemma \ref{Xs}, $|u^{2}X(s)|=|X(s)|=|s|$. Since $$\left|\frac{u^2}{12}  -\frac{\mathfrak{D}^{(2)}(a)}{6}\right|<|\mathfrak{D}^{(1)}(a)|<1,$$ we find that $|u^{2}X(s)+\frac{u^2}{12}  -\frac{\mathfrak{D}^{(2)}}{6}|<|\mathfrak{D}^{(1)}(a)|$ and therefore, $|\overline{x}(s)|> 1$. If we have $|s|=|q|^{1/2}<|\mathfrak{D}^{(1)}(a) |$ then Lemma~\ref{Xs} implies  that ${|u^{2}X(s)|=|X(s)|\leq |s|<|\mathfrak{D}^{(1)}(a)|}$. Since  ${|\frac{u^2}{12}  -\frac{\mathfrak{D}^{(2)}(a)}{6}|<|\mathfrak{D}^{(1)}(a)|} $, we deduce that ${|u^{2}X(s)+\frac{u^2}{12}  -\frac{\mathfrak{D}^{(2)}(a)}{6}|<|\mathfrak{D}^{(1)}(a)|}$
 and therefore, $|\overline{x}(s)|> 1$. This concludes the proof.
\end{proof}

\begin{rmk}\label{rmk:symargcontinuation}
By symmetry between $x$ and $y$, one could have define $U_{y}=\{(x,y) \in E \subset \P1 (C)\times \P1 (C)||y|\leq 1\}$ and continue $F^{1}(x,t)$
 on $U_y$ by setting 
  $$
F^{1}(x,t)= -xy -F^{2}(y,t) -td_{-1,-1} Q(0,0,t).
$$  Then, the composition of the  $F^i$ with the surjective map $\psi$ defined in Remark \ref{rmk:parametrizationcurvegenus1} yields to functions $\breve{F}^{i}$ that are defined on
$\mathcal{U}_y:= \psi^{-1}(U_{y})\cap \{s\in C^*| |s|\in [|q|^{1/2},|q|^{-1/2}[\}$. The analogue of Lemma \ref{lem1} is as follows. For   $|s|\in [|q|^{1/2},|q|^{-1/2}[$, the following statements  hold:
\begin{itemize}
\item   if $|s|\in ]|\mathfrak{E}^{(1)}(b)|,|\mathfrak{E}^{(1)}(b)|^{-1}[$, then $|\overline{y}(s)|< 1$;
\item if $|s|=|\mathfrak{E}^{(1)}(b)|^{\pm 1}$ then $|\overline{y}(s)|= 1$;
\item otherwise $|\overline{y}(s)|> 1$.
\end{itemize}
\end{rmk}

By Proposition \ref{prop1}, the automorphism of the walk corresponds to 
the $\q$-dilatation on $C^*$. The following lemma shows that one can cover $C^*$ either  with the $\q$-orbit of the  set $\mathcal{U}_x$ or with  the $\q$-orbit of $\mathcal{U}_y$.

\begin{lemma}\label{lem5}
The following statement  hold:
\begin{itemize}
\item $|\q|\neq 1$;
\item moreover, up to replace $\q$ by some convenient $q^\Z$-multiple, the following hold: \begin{itemize}
\item if  either $d_{-1,1}= 0$ or $d_{1,-1}\neq 0$, then, 
$$\displaystyle \bigcup_{\ell\in \Z}\s_\q^{\ell} (\mathcal{U}_{x})=C^{*};$$
\item if either $d_{-1,1}\neq 0$ or $d_{1,-1}= 0$ then, 
$$\displaystyle \bigcup_{\ell\in \Z}\s_\q^{\ell} (\mathcal{U}_{y})=C^{*}.$$
\end{itemize}
\end{itemize}
\end{lemma}

\begin{proof}
Let us first prove that  $|\q|\neq 1$. By Remark \ref{rmk:choiceforqandsymmetry}, one can choose $\q$ so that we have
${|q|^{1/2} \leq |\q| < |q|^{-1/2}}$.
 By construction, $\overline{x}(1)=a$. Let $b\in \P1 (C)$ such that $(a,b)\in E$. Since $\iota_{1}(a,b)=(a,b)$ we have $\iota_{2}(a,b)\neq (a,b)$ by  Lemma \ref{lem:genus1nofixedpoint}. So let $a' \in \P1 (C)$ distinct from $a$ such that $\sigma (a,b)=(a',b)$. Then, $\overline{x}(\q)=a'$. By Lemma~\ref{lem1}, $|\overline{x}(s)|<1$ for $|s|=1$. Thus,  it suffices to prove that  $|\overline{x}(\q)|=|a'|\geq  1$ to conclude that $|\q| \neq 1$.\par 
Remember that  $K\left(x,y,t\right)=\widetilde{A}_{-1}(x)+\widetilde{A}_{0}(x)y+\widetilde{A}_{1}(x)y^{2}=\widetilde{B}_{-1}(y)+\widetilde{B}_{0}(y)x+\widetilde{B}_{1}(y)x^{2}$ with $\widetilde{A}_{i}(x)\in C[x]$ and $\widetilde{B}_{i}(y)\in C[y]$.
With $\iota_{1}(a,b)= (a,b)$ and the formulas in $\S \ref{sec:autoofthewalks}$, one finds that 
$$
b^2 = \frac{A_{-1}(a) }{A_{1}(a)}=\frac{\widetilde{A}_{-1}(a) }{\widetilde{A}_{1}(a)}.
$$ 

Let $\nu$ be the valuation at  $X=0$   of $\frac{\widetilde{A}_{-1} (X)}{\widetilde{A}_{1}(X)} $. Lemma \ref{lem:normestimatecompositionfraction} with $|a|<1$ gives $|b|^2=|a|^\nu$. Note that $\widetilde{A}_{1}$ and $\widetilde{A}_{-1}$ are polynomial of degree at most two  in $X$, so the integer  $\nu$ belongs to $\{-2,-1,0,1,2\}$. We have
\begin{equation}\label{eq:a'functiona}
a' = \frac{\widetilde{B}_{-1}(b) }{\widetilde{B}_{1}(b)a}.
\end{equation}
We will prove that $|a'|\geq1$ with a case by case study of the values of $\nu$.

Remember that 
\begin{equation}\label{eq:expressionwidetildeAiBi}
\begin{array}{lll}
\widetilde{A}_{-1}&=&d_{-1,-1}+d_{0,-1}x+d_{1,-1}x^{2} \\
\widetilde{A}_{1}&=&d_{-1,1}+d_{0,1}x+d_{1,1}x^{2} \\
\widetilde{B}_{-1}&=&d_{-1,-1}+d_{-1,0}y+d_{-1,1}y^{2} \\
\widetilde{B}_{1}&=&d_{1,-1}+d_{1,0}y+d_{1,1}y^{2}.
\end{array}
\end{equation}

\textbf{Case $\nu \geq 1$.} Then, $|b|=|a|^{\nu/2} <1$. Combining \eqref{eq:a'functiona} and Lemma \ref{lem:normestimatecompositionfraction}, we find
$|a||a'|=|b|^l$ where $l$ is the valuation at $X=0$ of $\frac{\widetilde{B}_{-1} (X)}{\widetilde{B}_{1}(X)}$. This gives $|a'|=|a|^{l\nu/2 -1}$. Since $l$ belongs to $\{ -2,\dots,2\}$ and $\nu$ is in $ \{1,2\}$, we get  $-3 \leq l \nu/2 -1 \leq 1$. If $l\nu/2 -1$ equals $1$ then $\nu$ must be equal to $2$  and by \eqref{eq:expressionwidetildeAiBi}, we must have $d_{-1,-1}=d_{0,-1}=0$ and $d_{-1,1} \neq 0$. By Remark \ref{remgenre0}, we must have $d_{-1,0}d_{1,-1} \neq 0$ so that $l=1$ and $l\nu/2-1=0$. A contradiction. Then, $l\nu/2 -1 \leq 0$ and 
$|a'| \geq 1$.

\textbf{Case $\nu=0$.} Then, $|b|=1$. With Lemma \ref{lemma:nonemptyopenset} and $|a|<1$, we obtain $|a'|>1$. \par 
\textbf{Case $\nu \leq -1$.} Then $|b |= |a|^{\nu/2}>1$. Combining \eqref{eq:a'functiona} and Lemma \ref{lem:normestimatecompositionfraction}, we find
${|a'|=|a|^{l\nu/2 -1}}$ where $l \in \{-2,\dots,2\}$ is  the degree in $X$ of $\frac{\widetilde{B}_{-1} (X)}{\widetilde{B}_{1}(X)}$.  Since $l$ belongs to $\{ -2,\dots,2\}$ and $\nu$ is in $\{-1,-2\}$, we get  $1 \geq l\nu/2 -1 \geq -3$. If $l \nu/2 -1=1$ then $\nu =-2$   and by \eqref{eq:expressionwidetildeAiBi}, we must have $d_{-1,1}=d_{0,1}=0$ and $d_{-1,-1} \neq 0$. By Remark \ref{remgenre0}, we must have $d_{-1,0}d_{1,1} \neq 0$ so that $l=-1$ and $l\nu/2-1=0$. A contradiction. Then, $l\nu/2 -1 \leq 0$ and 
$|a'| \geq 1$.\\ \par

Assume that either $d_{-1,1}= 0$ or $d_{1,-1}\neq 0$ and let us prove that
$$\displaystyle \bigcup_{\ell\in \Z}\s_\q^{\ell} (\mathcal{U}_{x})=C^{*}.$$
By Lemma \ref{lem:intersectionsigmaunitdisknonempty}, there exists  $(a_0,b_0) \in E$ such that $|a_0|=1$ and $\s(a_0,b_0)=(a_1,b_1)$ with $|a_1| \leq 1$.
By Lemma \ref{lem1}, there exists  $s_0\in C^*$ with  $|s_0|=|\mathfrak{D}^{(1)}(a)|^{\pm 1}$  such that 
$\overline{x}(s_0)=a_0$. Since $|q|^{1/2}\leq |\q|<|q|^{-1/2}$ and $|q|^{1/2}<|\mathfrak{D}^{(1)}(a)|<1$, we find that   ${|q|<|\q s_0| < |q|^{-1/2}}$. Since $|\overline{x}(\q s_0)|=|a_1| \leq 1$, we conclude using  Lemma \ref{lem1} that 
\begin{itemize}
\item either $|\q s_0 | \in  \cU_x$. This proves that  $$\cU_x \cap \sigma_\q(\cU_x)=[|\mathfrak{D}^{(1)}(a)|, |\mathfrak{D}^{(1)}(a)|^{-1}]\cap \sigma_{\q}([|\mathfrak{D}^{(1)}(a)|, |\mathfrak{D}^{(1)}(a)|^{-1}])\neq \varnothing .$$ Since $|\q| \neq 1$, we deduce that
$$\displaystyle \bigcup_{\ell\in \Z}\s_\q^{\ell} (\mathcal{U}_{x})=C^{*}.$$
\item or $|\q s_0 | \in [|q||\mathfrak{D}^{(1)}(a)|, |q||\mathfrak{D}^{(1)}(a)|^{-1}]$. Replacing $\q$ by $\q/q$ allows to conclude.
\item or $|\q s_0 | \in [|q|^{-1}|\mathfrak{D}^{(1)}(a)|, |q|^{-1}|\mathfrak{D}^{(1)}(a)|^{-1}]$. Replacing $\q$ by $q\q$ allows to conclude.
\end{itemize}
The proof for  $\mathcal{U}_{y}$ is obtained by a symmetry argument using   Lemma \ref{lem:intersectionsigmaunitdisknonempty} and  Remark~\ref{rmk:symargcontinuation}.
\end{proof}

According to Lemma \ref{lem5}, we define some auxiliary functions    as follows 
\begin{itemize}
\item  if  $d_{-1,1}= 0$, we define, for $i=1,2$,  the function  $\widetilde{F}^{i}(s)$ on $ \mathcal{U}_x$ as  $F^i(\phi(s),t)$; 
\item if  $d_{-1,1}\neq 0$,   the function  $\widetilde{F}^{i}(s)$ is defined on  $\mathcal{U}_y$  as  $F^i(\psi(s),t)$.
\end{itemize}
A priori the auxiliary functions $\widetilde{F}^{1}(s), \widetilde{F}^{2}(s)$ are defined on $\cU_x $ if  $d_{-1,1}= 0$ and on $\cU_y$ otherwise.   Theorem \ref{thm:funceqgenus1} below shows that one can meromorphically 
continue the functions $\widetilde{F}^i(s)$ on $C^*$ so that they satisfy 
some  nonhomogeneous rank $1$ linear $\q$-difference equations.

\begin{thm}\label{thm:funceqgenus1}
 The auxiliary functions $\widetilde{F}^{1}(s), \widetilde{F}^{2}(s)$ can be continued meromorphically  on $C^*$ so that they satisfy 
$$
\widetilde{F}^1(\q s)-\widetilde{F}^1(s)= b_1 
$$
and $$
\widetilde{F}^2(\q s)-\widetilde{F}^2(s)=b_2,
$$
where  $b_1=(x( \q s) -x( s))y(\q s)$ and $b_2 = (y(\q s) -y( s))x(s)$  are two $q$-periodic meromorphic functions over $C^*$. 

\end{thm}

\begin{proof}
The proof is completely similar to the proof of Lemma \ref{lemma:Analyticcontinuationandfuncequ} and relies on the fact that either the $\q$-orbit of $\mathcal{U}_x$ or the $\q$-orbit of  $\mathcal{U}_y$ covers $C^*$.
\end{proof}
Note that by Remark \ref{rmk:parametrizationcurvegenus1}, the coefficients $b_1,b_2$ of the $\q$-difference can be identified with rational functions on the algebraic curve $E$.

\subsection{Differential transcendence}\label{sec54}

The strategy to study the differential transcendence of generating functions  of nondegenerate weighted models of   genus one with infinite group is similar to the one employed in \S \ref{secgenre0}.  One first relate the differential behavior  of the incomplete generating functions to the differential algebraic properties of their associated auxiliary functions. Then,    one applies to these auxiliary functions the Galois theory of $\q$-difference equations. 
However, since the coefficients of the $\q$-difference equations satisfied by the auxiliary functions  are no longer rational but elliptic, the Galoisian criteria as well as the descent method to obtain some ``simple  telescopers'' are quite technical and postponed to Appendix  \ref{sec:differenceGaloistheory}. Theorem \ref{theo1} below gives a  first  criteria to guaranty the differential transcendence of the incomplete generating function.

\begin{thm}\label{theo1}

Assume that the weighted model is nondegenerate, of  genus one, and that the group of the walk is infinite. If 
$Q(x,0,t)$ is $\left(\frac{d}{dx},\frac{d}{dt}\right)$-differentially algebraic over $\Q$ then there exist $c_0, \dots,c_n \in C$ not all zero and $h \in C_q$ such that 
\begin{equation}\label{eq:telescopergenus1}
c_0 b_1+c_1 \partial_s(b_1)+ \dots +c_n \partial_s^n (b_1) =\s_\q(h)-h.
\end{equation}
\end{thm}
A symmetrical result holds for $Q(0,y,t)$ replacing $b_1$ by $b_2$.
\begin{proof}
Since the group of the walk is of infinite order, the 	automorphism $\sigma$ is of infinite order. Therefore
by Lemma \ref{lemma:qqmultindinfite order} the elements $\q$ and $q$ defined in  Proposition \ref{prop1} are multiplicatively independent. 
Assume that  $Q(x,0,t)$ is $\left(\frac{d}{dx},\frac{d}{dt}\right)$-differentially algebraic over $\Q$. Let $\widetilde{F}^{1}(s)$  be the auxiliary  function defined above.

 We  denote by $C_{\q}.C_{q}$  the compositum of  the  fields $C_q$ and $C_\q$ inside the field of meromorphic functions over $C^*$.
 We claim that  $\widetilde{F}^{1}(s)$
is $\left(\partial_s,\Delta_{t,\q}\right)$-differentially algebraic over $C_{\q}.C_{q}(\ell_\q,\ell_q)$.  Let us prove this claim when $d_{-1,1}=0$,  the proof when $d_{-1,1}\neq 0$ being similar.  Reasoning as in Lemma~\ref{lem3}, one can show that, for $n,m \in \N$, one has 
$$
(\partial_t^n \partial_x^{m} F^1)(\overline{x}(s),t)=
 \frac{1}{\partial_s(\overline{x}(s))^m} \Delta_{t,q}^n \partial_s^m (\widetilde{F}^1(s)) + \sum_{i \leq 2n +m,j <n} r_{i,j} \Delta_{t,q}^j\partial_s^i(\widetilde{F}^1(s)),
$$
where $r_{i,j} \in C_\q(\ell_\q)(\overline{x}(s),\partial_s^l\partial_t^k(\overline{x}(s)),\dots)$.  By construction, $\overline{x}(s)$ is in $ C_q$ so that  Lemma~\ref{lemma:fielddefinitiongenus1} implies  that $ \partial_s^l\partial_t^k(\overline{x}(s)) \in C_q(\ell_q)$ for any positive integers  $k,l$. Then,  the field  $C_\q(\ell_\q)(\overline{x}(s),\partial_s^l\partial_t^k(\overline{x}(s)),\dots)$ generated by $\overline{x}$
and its derivatives with respect to $\partial_s$ and $\partial_t$ is contained in $C_{\q}.C_{q}(\ell_\q,\ell_q)$. Thus,
any nontrivial polynomial relation between the $x$-$t$-derivatives of $Q(x,0,t)$ yields to a 
nontrivial polynomial relation between the derivatives of $\widetilde{F}^{1}(s)$ with respect to $\partial_s$ and $\Delta_{t,\q}$
over $C_{\q}.C_{q}(\ell_\q,\ell_q)$. This proves the claim. 

By Theorem \ref{thm:funceqgenus1}, the function $\widetilde{F}^{1}(s)$ satisfies $\widetilde{F}^{1}(\q s)- \widetilde{F}^{1}(s) =b_1(s)$ with $b_1(s) \in C_q \subset C_{\q}.C_{q}(\ell_\q,\ell_q)$. Since $ \widetilde{F}^{1}(s)$ is $\left(\partial_s,\Delta_{t,\q}\right)$-differentially algebraic over $C_{\q}.C_{q}(\ell_\q,\ell_q)$, Proposition~\ref{prop2} and 
Corollary \ref{cor1} imply that 
there exist $m \in \N$ and $d_0,\dots,d_m \in C_{\q}$ not all zero and ${g \in C_{\q}.C_{q}(\ell_{q})}$ such that 
$$
d_0 b_1+d_1\partial_s(b_1)+\dots +d_m \partial_s^m(b_1)=\s_{\q}(g)-g .
$$
Since $b_1$ is in $C_q$, Lemma \ref{Lemma:telescoperdescentgenus1} allows to perform a descent on the coefficients of the telescoping relation above. Thus, there exist $c_0, \dots,c_n \in C$ not all zero and $h \in C_q$ such that 
$$
c_0 b_1 +c_1\partial_s( b_1)+ \dots +c_n \partial_s^n (b_1) =\s_\q(h)-h .
$$
This  concludes the proof. The symmetry argument between $x$ and $y$ gives the proof for $Q(0,y,t)$.
\end{proof}
Theorem \ref{theo1} has an easy corollary concerning  the differential transcendence of the complete generating function for  weighted models of genus one with infinite group.

\begin{thm}\label{theo2}
For any nondegenerate  weighted model of genus one with  infinite group, the following statements are equivalent:
\begin{enumerate}
\item the series $ Q(x,0,t)$ is $\left(\frac{d}{dx},\frac{d}{d t}\right)$-differentially algebraic over $\Q$;

\item the series $ Q(x,0,t)$ is $\frac{d}{dx}$-differentially algebraic over $\C$.
\end{enumerate}

\end{thm}

\begin{rmk}\label{rem2}
An analogous result holds for $Q(0,y,t)$ replacing the derivation $\frac{d}{dx}$ by $\frac{d}{dy}$.
\end{rmk}
\begin{proof}

Since the group is infinite, the 	automorphism $\sigma$ is of infinite order. Therefore
by Lemma~\ref{lemma:qqmultindinfite order} the elements $\q$ and $q$ defined in  Proposition \ref{prop1} are multiplicatively independent. \par 
 Assume that $(1)$ holds. By Theorem  \ref{theo1}, there exist $c_0, \dots,c_n \in C$ not all zero and $h \in C_q$ such that 
\begin{equation}\label{eq:telescopergenus1proof}
c_0b_1+c_1 \partial_s (b_1)+ \dots +c_n \partial_s^n (b_1) =\s_\q(h)-h.
\end{equation}
Combining \eqref{eq:telescopergenus1proof} with the functional equation satisfied by $\widetilde{F}^1(s)$ and using the commutativity of $\s_\q$ and $\partial_s$, one finds that 
\begin{equation}\label{eq:teklescoperimpliesdiffalgrelations}
\s_\q\left[ c_0 \widetilde{F}^1(s) +\dots + c_n\partial_s^n( \widetilde{F}^1(s))-h \right]= c_0 \widetilde{F}^1(s) +\dots + c_n\partial_s^n( \widetilde{F}^1(s))-h.
\end{equation}
Since $\widetilde{F}^1$ and $h$ are meromorphic over $C^*$,   there exists $g \in C_\q$ such that 
$$c_0 \widetilde{F}^1(s) +\dots + c_n\partial_s^n( \widetilde{F}^1(s))-h =g.$$
Therefore, 
$\widetilde{F}^1(s)$ is $\partial_s$-differentially algebraic over $C_q$. Reasoning as in Lemma \ref{lem3}, one finds a nontrivial algebraic relation with coefficients in $C_q$ between the  first $n$-th derivatives of  $ F^1 $ with respect to $\partial_x$ evaluated in $(\overline{x}(s),t)$. Any  element of $C_q=C(\overline{x}(s),\overline{y}(s))$ is algebraic over $C(\overline{x}(s))$. Therefore, the 
first $n$-th derivatives of  $ F^1 $ with respect to $\partial_x$ evaluated in $(\overline{x}(s),t)$ are still algebraically dependent over  $C(\overline{x}(s))$. We conclude that   $F^1(x,t)=K(x,0,t)Q(x,0,t)$ is $\frac{d}{dx}$-differentially algebraic over $C(x)$ and therefore over $\Q$ by Remark \ref{rmk:CdifftransQdifftrans}.  This proves that  $(1) \Rightarrow(2)$. Statement  $(2)$ implies obviously  $(1)$. 
\end{proof}

A corollary of Theorem \ref{theo1} is that the $\frac{d}{dt}$-differential algebraicity of the series implies the $\frac{d}{dx}$-algebraicity of the series of the series. One of the major breakthrough of \cite{BBMR16} is to show that for unweighted walks, the series was $\frac{d}{dx}$-differentially algebraic over $\Q$ if and only if the models was decoupled, that is, there exist $f,g \in \Q(t)(X)$ such that 
\begin{equation}\label{eq:decoupled}
xy=f(x)+g(y) \mbox{ modulo } K(x,y,t).
\end{equation} 
The authors  of \cite{BBMR16} used   boundary value problems and the notion of  analytic invariants to deduce from \eqref{eq:decoupled}  a closed form of the generating series allowing them to conclude that the series was also $\frac{d}{dt}$-algebraic (see \cite[\S 6.4]{BBMR16}). Combining our result to \cite{BBMR16}, one finds the following corollary

\begin{cor}\label{cor:diffalgxequivalentdiffalgtgenusone}
If the walk is unweighted of genus one with infinite group, the following statement are equivalent
\begin{itemize}
\item the generating series is $\frac{d}{dx}$-differentially algebraic over $\Q$;
\item the generating series is $\frac{d}{dt}$-differentially algebraic over $\Q$;
\end{itemize}

\end{cor}

In a recent publication  \cite{hardouin2020differentially}, M.F.Singer and the second author generalized the results of \cite{BBMR16} and proved that a weighted model of genus one with infinite group was decoupled if and only if the series was $\frac{d}{dx}$-differentially algebraic. There is no doubt that 
the arguments of \cite{BBMR16} proving that if a model is decoupled then the series is $\frac{d}{dt}$-algebraic over $\Q$ will hold in a weighted situation. Combined to Theorem \ref{theo2}, this  will prove that  Corollary \ref{cor:diffalgxequivalentdiffalgtgenusone} is also true for weighted walks.

\begin{appendix}

%%%%%%%%%%%%%%%%%%%%%%%%%%%%%%%%%%%%%%%%%%%%%%%%%%%%%%%%%%%%%%%%%%%%%%%%%%%%%%%%%%%%%%%%%%%
%%%%%%%%%%%%%%%%%%%%%%%%%%%%%%%%%%%%%%%%%%%%%%%%%%%%%%%%%%%%%%%%%%%%%%%%%%%%%%%%%%%%%%%%%%%%

\section{Nonarchimedean estimates}\label{sec:nonarcheestimates}
%%%%%%%%%%%%%%%%%%%%%%%%%%%%%%%%%%%%%%%%%%%%%%%%%%%%%%%%%%%%%%%%%%%%%%%%%%%%%%%%%%%%%%%%%%%
%%%%%%%%%%%%%%%%%%%%%%%%%%%%%%%%%%%%%%%%%%%%%%%%%%%%%%%%%%%%%%%%%%%%%%%%%%%%%%%%%%%%%%%%%%%%
In this section, we give some nonarchimedean estimates, which will be crucial to uniformize the kernel curve.
\subsection{Discriminants of the kernel equation}
Lemma \ref{lemma:goodroot} relates the genus of the kernel curve to the simplicity of the roots of the discriminant 
of the kernel polynomial. It also ensures the existence of a root with convenient norm estimates. Let us remind, see \eqref{eq:expression_D_0}, that we have defined $\mathfrak{D}(x):=\Delta_x(x,1)$, where  $\Delta_x(x_0,x_1)$ is the discriminants of the second degree homogeneous polynomials $y \mapsto \widetilde{K}(x_0,x_1,y,1,t)$.
     
\begin{lemma}\label{lemma:goodroot}
For any nondegenerate weighted model  of genus one, the following holds:
 \begin{itemize}
 \item  all the roots of $\Delta_x(x_0,x_1)$ in $\P1(C)$ are simple;
 \item the discriminant $\mathfrak{D}(x):=\Delta_x (x,1)$ has a root $a \in C$ such that $|a|<1$, $|\mathfrak{D}^{(2)}(a) -2|<1$, and $|\mathfrak{D}^{(1)}(a)|,|\mathfrak{D}^{(3)}(a)|,|\mathfrak{D}^{(4)}(a)|<1$ where $ \mathfrak{D}^{(i)}$ denote the $i$-th derivative 
 with respect to $x$ of $\mathfrak{D}(x)$.
 \end{itemize}
 A symmetric statement holds for $\Delta_y(y_0,y_1)$ by replacing $\mathfrak{D}$ by $\mathfrak{E}$.
 \end{lemma}
 \begin{proof}
The first assertion is \cite[Proposition 2.1]{DreyfusHardouinRoquesSingerGenuszero2}. First, let us prove the existence of a root $a \in C$ of $\mathfrak{D}(x)$ such that $|a|<1$.  Suppose to the contrary that all the roots of $\mathfrak{D}(x)$ have a norm greater than or equal to $1$.  If $\alpha_0 $ is zero  then  zero is a root: a contradiction. Thus, we can assume that $\alpha_0$ is nonzero.
 
 Let us first assume that $\alpha_4 \neq 0$.  The product of  the roots of $\mathfrak{D}(x)$ equals $$\frac{\alpha_0}{\alpha_4}=\frac{t^2(d_{-1,0}^2 -4d_{-1,-1}d_{-1,1})}{t^2(d_{1,0}^2 -4d_{1,-1}d_{1,1})}.$$
Then  we conclude that $|\frac{\alpha_0}{\alpha_4}|=1$  so that each of the roots must have norm  $1$. Then, considering the symmetric functions of the roots of $\mathfrak{D}(x)$, we conclude that, for any $i=0,\dots, 3$, the element $\frac{\alpha_i}{\alpha_4}$ should have  norm smaller than or equal to $1$.  Since $$\frac{\alpha_2}{\alpha_4}=\frac{-4d_{-1, -1}d_{1, 1}t^2-4d_{0, -1}d_{0, 1}t^2-4d_{1, -1}d_{-1, 1}t^2+2d_{-1, 0}d_{1, 0}t^2+d_{0, 0}^2t^2-2td_{0, 0}+1}{t^2(d_{1,0}^2 -4d_{1,-1}d_{1,1})},$$
 has norm strictly greater than $1$, we find  a  contradiction.\par 
 Assume now that $\alpha_4 = 0$. Since the roots of $\Delta_x(x_0,x_1)$ in $\P1(C)$ are simple, the coefficient  $\alpha_3$ is nonzero.  The product of the roots of $\mathfrak{D}(x)$ equals $$-\frac{\alpha_0}{\alpha_3}=\frac{-t^2(d_{-1,0}^2 -4d_{-1,-1}d_{-1,1})}{2t^{2}d_{1,0}d_{0,0}-2td_{1,0}-4t^{2}(d_{0,1}d_{1,-1}+d_{1,1}d_{0,-1})}.$$
Then, it is clear that  $|\frac{\alpha_0}{\alpha_3}|\leq 1$  and that each of the roots has norm $1$.  Thus, the  symmetric function $\frac{\alpha_2}{\alpha_3}$
should also  have norm   smaller than  or equal to $1$.
But
$$-\frac{\alpha_2}{\alpha_3}=\frac{-4d_{-1, -1}d_{1, 1}t^2-4d_{0, -1}d_{0, 1}t^2-4d_{1, -1}d_{-1, 1}t^2+2d_{-1, 0}d_{1, 0}t^2+d_{0, 0}^2t^2-2td_{0, 0}+1}{2t^{2}d_{1,0}d_{0,0}-2td_{1,0}-4t^{2}(d_{0,1}d_{1,-1}+d_{1,1}d_{0,-1})},$$ has  norm strictly bigger than $1$. We find a contradiction again. \par 
 
Let  $a$ be a root of $\mathfrak{D}(x)$ in $C$ with $|a|<1$.  Since $a,\alpha_1,\alpha_{3},\alpha_{4}$ have norm smaller than $1$, $|\alpha_2 -1| <1$, and 
\begin{itemize}
\item $\mathfrak{D}^{(1)}(a)=\alpha_1 +2\alpha_2a +3 \alpha_3 a^2 +4 \alpha_4a^3$;
\item  $\mathfrak{D}^{(2)}(a)=2 \alpha_2+ 6 \alpha_3 a +12 \alpha_4 a^2$;
\item $\mathfrak{D}^{(3)}(a)=6 \alpha_3 +24 \alpha_4 a$;
\item $\mathfrak{D}^{(4)}(a)=24\alpha_4 $,
\end{itemize}
 we have $|\mathfrak{D}^{(2)}(a)-2|<1$, and $|\mathfrak{D}^{(1)}(a)|,|\mathfrak{D}^{(3)}(a)|,|\mathfrak{D}^{(4)}(a)|<1$.
The statement for $\Delta_y(y_0,y_1)$ is symmetrical and we omit its proof.
 \end{proof}
\subsection{Automorphisms of the  walk on the domain of convergence}
In this section, we study the action of the group of the walk on the product of the unit disks in $\bold{P}^1(C) \times \bold{P}^1(C) $. This product is the  fundamental  domain of convergence of the generating function.

We need a preliminary lemma that  explains how one can compute the norm of 
the values of a rational function.
\begin{lemma}\label{lem:normestimatecompositionfraction}
Let $f \in C(X)$ be a nonzero rational function and  let  ${a \in \P1(C)}$. Let $\nu$ (resp. $d$) be the valuation at $X=0$ (resp. $ \infty$) of $f$ with the convention that $\nu=+\infty$, $d=-\infty$ if $f=0$. The following statements  hold:
\begin{itemize}
\item if $|a| <1$, then $|f(a)| =|a|^\nu$;
\item if $|a|>1$, then $|f(a)|=|a|^d$.
\end{itemize}
\end{lemma}
\begin{proof}
Let us prove the first case, the second being completely symmetrical.  Let us write $f(X)$ as  $\frac{\sum_{i=\nu_1}^{r_1} c_i X^i}{\sum_{j=\nu_2}^{r_2} d_j X^j}$ with $c_{\nu_1}d_{\nu_2}  \neq0$. If $k>l$, we note that  $|a^k| <|a^l|$.  Then $$|f(a)|=\frac{|\sum_{i=\nu_1}^{r_1} c_i a^i|}{|\sum_{j=\nu_2}^{r_2} d_j a^j|}=|a|^{\nu_1-\nu_2}=|a|^{\nu}.$$ 
\end{proof}

The following lemma explains how  the fundamental involutions permute the interior and the exterior of the fundamental domain  of convergence. 

\begin{lemma}\label{lemma:nonemptyopenset}
For any nondegenerate weighted model, the  following statements  hold: 
\begin{enumerate}
\item  for any $a \in C$ with $|a|=1$, there exist $b_{\pm} \in \P1(C)$ with $|b_{-}|<1$, and $|b_{+}|>1$, such that  $K(a,b_{\pm},t)=0$;
\item for any $b \in C$ with $|b|=1$, there exist $a_{\pm} \in \P1(C)$ with $|a_{-}|<1$, and $|a_{+}|>1$, such that  $K(a_{\pm},b,t)=0$.
\end{enumerate}
\end{lemma}
\begin{proof}
See \cite[Section 1.3]{dreyfus2019differential} for a similar result in the situation where $C$ is replaced by $\C$.

The statements are symmetrical, so we only prove  the first one. Since $C$ is algebraically closed and   the model  is nondegenerate, Proposition \ref{prop:degeneratecases} implies that $K(x,y,t)$ is of degree $2$ in $y$. Then,  for any $a \in C$,  there are two elements $b_{\pm}\in \P1(C)$ such that 
$K(a,b_{\pm},t)=0$. let $a \in C$ with  $|a|=1$. We write \begin{equation}\label{eq:kernelaffine}
K(a,y,t) =t \alpha + \beta y + t\gamma y^2
\end{equation}

 where \begin{itemize}
\item $\alpha=-\sum_{i=-1}^{1} d_{i,-1}a^{i+1}$;
\item $\beta =a - t\sum _{i=-1}^{1} d_{i,0}a^{i+1}$;
\item $\gamma=- \sum_{i=-1}^{1} d_{i,1} a^{i+1}$.
\end{itemize}

Since $|a|=1$, we find $|\beta|=1$, $|\alpha |,|\gamma| \leq 1$. First let us prove that there is no point $(a_0,b_0) \in E$ such that $|a_0|=|b_0|=1$. Indeed, suppose to the contrary that $|a_0|=|b_0|=1$ and ${K(a_0,b_0,t)=0}$. Then, $|\beta|=|a_0|=1$ and  $|\gamma|,|\alpha | \leq 1$ so that the equality 
$|\beta b_0|= |t(\alpha + \gamma b_0^2)|$ implies
$|b_0| <1$. We find a contradiction. From the equation $K(a,b,t)=0$, we deduce that
\begin{equation}\label{pt1}
\hbox{ if }|b |<1,  \hbox{ then }|t \alpha| =|\beta b + t \gamma b^2|=|\beta b|\hbox{ which gives }|b|=|t \alpha|;
\end{equation}
\begin{equation}\label{pt2}
\hbox{ if }|b| >1,  \hbox{ then }|\frac{1}{b}| <1 \hbox{ and we find }|t \gamma|=\left|\frac{t\alpha }{b^{2}}+\frac{\beta}{b}\right|=\left|\frac{\beta}{b}\right|=\left|\frac{1}{b}\right|.
\end{equation}

Using $K(a,b_{\pm},t)=0$, we find 
\begin{equation}\label{eq:productroots}
b_{-}b_{+}=\frac{\alpha}{\gamma},
\end{equation}
with the convention that $b_{+}$  is $[1:0]$ if $\gamma=0$. If $\gamma=0$ then $b_{-}=\frac{-t \alpha}{\beta}$ has  norm smaller than $1$, which  concludes the proof in that case. Assume now  that $\gamma\neq 0$.
 Since $|b_{+}|$ and $ |b_{-}|$ cannot have norm $1$, we just need to discard the cases ``$|b_{+}| <1 $ and $|b_{-}| <1$'' or  ``$|b_{+}| >1 $ and $|b_{-}| >1$''. If $\alpha =0$, then one of the root  is zero, say  $b_{-}=0$, and $|b_{+}|=\frac{|\beta|}{|t \gamma| }>1$, which concludes the proof in that case. If  $\alpha \neq 0$ then one can  suppose to the contrary that $|b_{+}| <1 $ and $|b_{-}| <1$.  From  \eqref{pt1}, we obtain  $|b_{+}|=|b_{-}|=|t \alpha|$, which gives
 $$|b_{+}b_{-}|=|t \alpha|^2= \frac{|\alpha |}{|\gamma|}.$$
Then,
$|t^2 \alpha | =\frac{1}{|\gamma | } \geq 1$, which contradicts  $|t^2 \alpha |<1$.   Suppose to the contrary that $|b_{+}| >1 $ and $|b_{-}| >1$. By \eqref{pt2}, $|b_{+}|=|b_{-}|= \frac{1}{|t \gamma|}$ which gives
 $$|b_{+}b_{-}|= \frac{1}{|t \gamma|^2}= \frac{|\alpha |}{|\gamma|}.$$
Thus, $|t^2 \alpha | =\frac{1}{|\gamma | } \geq 1$,  and once again, we find a  contradiction.
\end{proof}
 Lemma \ref{lem:intersectionsigmaunitdisknonempty} explains how the  the intersection of the  fundamental domain of convergence of the generating function and its image by $\sigma$  is nonempty. This result is  therefore crucial in order  to continue the generating function to the whole  $C^*$.
\begin{lemma}\label{lem:intersectionsigmaunitdisknonempty}
For any nondegenerate weighted model, the  following statements hold:
\begin{itemize}
\item if $d_{-1,1}=0$ or $d_{1,-1} \neq 0$ there exists $(a,b) \in E$ with $|a|=1$ such that $\sigma(a,b)=(a',b')$ with $|a'| \leq 1$;
\item  if $d_{-1,1}\neq 0$ or  $d_{1,-1} =0$ there exists $(a,b) \in E$ with $|b|=1$ such that $\sigma(a,b)=(a',b')$ with $|b'| \leq 1$.
\end{itemize}

\end{lemma}

\begin{proof}
Using  the symmetry between $x$ and $y$ mentioned in Remark \ref{rem1}, we  only  prove the first statement of Lemma \ref{lem:intersectionsigmaunitdisknonempty}. 

Let $a \in \P1(C)$ such that $|a|=1$. By Lemma \ref{lemma:nonemptyopenset}, there exist $b_{+}\in \P1(C)$ with $|b_{+}|>1$ and $b_{-}\in C$ with $|b_{-}|<1$ such that $(a,b_{\pm}) \in E$. Let $B_i$ as in \eqref{eq:defiAiBi} and note that by Proposition \ref{prop:degeneratecases}, $B_1$ is not identically zero.
Let $\nu$ (resp. $d$) be the valuation at $0$ (resp. $\infty$) of the rational fraction $\frac{B_{-1}(y)}{B_1(y)}=\frac{ \sum_{j=-1}^1 d_{-1,j}y^j }{\sum_{j=-1}^1 d_{1,j}y^j } \in C(y)$. We claim that either $\nu \geq 0$ or $d \leq 0$.  If $d_{1,-1}\neq 0$ then $\nu \geq 0$. If $d_{-1,1} =0$ then either $d \leq 0$ or $d =1$. In the latter situation, we must have $d_{1,1}=d_{1,0}=0$ and $d_{-1,0}\neq 0$. Since the model is nondegenerate, we must have $d_{1,-1}\neq 0$ by Proposition \ref{prop:degeneratecases}. In that case, $\nu \geq 0$. This proves the claim.

Let $a_{+},a_{-} \in \P1(C)$ such that $\iota_2 (a,b_{+})= (a_{+}, b_{+})$ and  $\iota_2 (a,b_{-})= (a_{-}, b_{-})$. This gives 
\begin{equation}\label{eq:relationa+ab+}
a_+  = \frac{B_{-1}(b_+)}{B_1(b_+)a} \mbox{ and }    a_-  =\frac{B_{-1}(b_{-})}{B_1(b_-)a}.
\end{equation}

Since $\sigma(a,b_{-})= (a_+,b_+)$ (resp. $\sigma(a,b_{+})= (a_{-},b_{-})$),    it is enough to prove that either $a_{+}$ or $a_{-}$ has norm smaller or equal  to $1$.  If $d \leq 0$, we combine \eqref{eq:relationa+ab+}, Lemma \ref{lem:normestimatecompositionfraction} and $|b+|>1$ to find  $|a a_{+}|=|a_+|=|b_+|^d \leq 1$. If $\nu \geq 0$, we combine \eqref{eq:relationa+ab+}, Lemma \ref{lem:normestimatecompositionfraction} and  $|b_-|<1$ to find  $|a a_{-}|=|a_-|=|b_-|^\nu \leq 1$. This ends the proof.
\end{proof}
%%%%%%%%%%%%%%%%%%%%%%%%%%%%%%%%%%%%%%%
%%%%%%%%%%%%%%%%%%%%%%%%%%%%%%%%%%%%%%
\section{Tate curves and their normal forms}\label{sec:nonarchimedianpreleminaries}
%%%%%%%%%%%%%%%%%%%%%%%%%%%%%%%%%%%%%%%%%
%%%%%%%%%%%%%%%%%%%%%%%%%%%%%%%%%%%%%%%%
Let $(C,|~|)$ be a complete nonarchimedean algebraically closed valued field of zero characteristic and let $q \in C$ such that $0<|q| <1$.  In this section, we recall 
some of the basic properties of elliptic curves over nonarchimedean fields. The period  lattice is here replaced by a discrete  multiplicative group of the form $q^\Z$. Then, the quotient of $\C$ by a  period lattice is replace by the so called \emph{Tate curve}, which  corresponds to the naive quotient of the multiplicative group $C^*$ by $q^\Z$. 
However, in the nonarchimedean context, only elliptic curves with $J$-invariant of norm greater than equal to one can be analytically uniformized by Tate curves (see Proposition \ref{prop:Tateuniformization}). The analytic geometry behind is the rigid analytic geometry as developed in \cite{FresnelvanderPUt}. We will not introduce this theory here   but  we just recall briefly the algebraic geometrical and special functions aspects of Tate curves.

\subsection{Special functions on a Tate curve}\label{sec:merofunc}
 
 We recall that any holomorphic function $f$ on $C^*$ can be represented by an everywhere convergent Laurent series $\sum_{n \in \Z} a_n s^n$ with $a_n \in C$. Moreover any nonzero meromorphic function on $C^*$ can be written as $\frac{g}{h}$ such that the holomorphic functions $g$ and $h$ have no common zeros.  We shall denote by $\cM er(C^*)$ the field of meromorphic functions over $C^*$.

\begin{rmk}
 If $k$ is a complete nonarchimedian sub-valued field of $C$ and $q$ belongs to  $ k$, every result quoted above still holds over $k$. 
\end{rmk}
The analytification of the  elliptic curve $E_q$  is isomorphic to the Tate curve, which  is the rigid analytic space corresponding to the naive quotient of $C^*/q^\Z$. The curve $E_q$ is therefore a``canonical'' elliptic curve. A natural question is "Given an elliptic curve $E$  defined over $C$, is there a $q$ such that $E$ is isomorphic to $E_q$?" The answer is positive under certain assumption on the $J$-invariant $J(E)$ of $E$. 

\begin{prop}[Theorem 5.1.18 in \cite{FresnelvanderPUt}] \label{prop:Tateuniformization}
Let $E $ be an elliptic curve  over $C$  such that $|J(E)|>1$. Then, there exists $q \in C$ such that $0<|q|<1$ and  $E$ is isomorphic to the elliptic curve  $E_q$.
\end{prop}
Remind that we  have defined  $s_k =\sum_{n >0}\frac{n^k q^n}{1-q^n} \in C$  for $k \geq 1$, and 
$$X(s)=\sum_{n \in \Z} \frac{q^n s}{(1-q^n s)^2} -2s_1,\quad Y(s)=\sum_{n \in \Z} \frac{(q^{n}s)^2}{(1-q^n s)^3} +s_1.$$
 They are  $q$-periodic  meromorphic  functions over $C^*$.
By Proposition \ref{prop:Tatecurve}, the field $C_q$  of   $q$-periodic  meromorphic functions over $C^*$   coincides with the field generated over $C$ by $X(s)$ and $Y(s)$. \par 
 Since we need to understand what is the pullback of the fundamental domain of convergence of the generating function via this uniformization, we prove some basic properties on the norm of  $X(s)$. Remind that $X(s)=X(1/s)$ and $X(qs)=X(s)$. Thus it suffices to study $|X(s)|$ for $|q|^{1/2}\leq |s|\leq 1$.   The following study follows the arguments of \cite[\S V.4]{Silvermanadavancedtopic}.

\begin{lemma}\label{Xs}
Let  $s\in C^{*}$. The following holds:
\begin{itemize}
\item If $|q|^{1/2}<|s|<1$, then $|X(s)|= |s|$;
\item If $|s|=1$,  then $|X(s)|\geq 1$;
\item If $|s|=|q|^{1/2}$, then $|X(s)|\leq  |s|$.
\end{itemize}
\end{lemma}

\begin{proof}
Since $X(s)$ has a pole in $s=1$ we may further assume that $s\neq 1$. Let us rewrite  $X(s)$:
$$X(s)=\frac{s}{(1-s)^2}+\sum_{n>0} \frac{q^n s}{(1-q^n s)^2}+\frac{q^n s^{-1}}{(1-q^n s^{-1})^2}-2 \frac{q^n }{1-q^n}. $$
This means that we have 
\begin{equation}\label{eq:majorationnormXs}
|X(s)| \leq \max \left( \left|\frac{s}{(1-s)^2}\right|, \left| \sum_{n>0} \frac{q^n s}{(1-q^n s)^2}+\frac{q^n s^{-1}}{(1-q^n s^{-1})^2}-2 \frac{q^n }{1-q^n}\right|\right),
\end{equation}
with equality when $ |\frac{s}{(1-s)^2}| \neq  | \sum_{n>0} \frac{q^n s}{(1-q^n s)^2}+\frac{q^n s^{-1}}{(1-q^n s^{-1})^2}-2 \frac{q^n }{1-q^n}|$. 
Let us consider ${s\in C^{*}\setminus \{1\}}$ with $|q|^{1/2}\leq |s|\leq 1$.
Using $|q|<1$ we find that 
$|q^n s|\leq |qs|<1$  for every $n\geq 1$. This shows that the norm of $q^n s$ is strictly smaller than $1$. Then, $\left|\frac{q^n s}{(1-q^n s)^{2}}\right|=|q^n s|<|s|$. On the other hand, 
$|q^n |\leq |q|<|s|$ and 
$|\frac{q^n }{1-q^n}|<|s|$.
Finally, when ${|q|^{1/2}< |s|}$, we have  ${|q^n s^{-1}|\leq |qs^{-1}|<|qq^{-1/2}|<|s|}$ and therefore $\left|\frac{q^n s^{-1}}{(1-q^n s^{-1})^{2}}\right|=|q^n s^{-1}|<|s|$. 
This proves  that,  for any $s \in \P1(C)$ such that $|q|^{1/2}< |s|\leq 1$, we have

\begin{equation}\label{eq1}
\left|\sum_{n>0} \frac{q^n s}{(1-q^n s)^2}+\frac{q^n s^{-1}}{(1-q^n s^{-1})^2}-2 \frac{q^n }{1-q^n} \right|<|s|.
\end{equation}
When, $|q|^{1/2}=|s|$ and $n\geq 2$, we have  ${|q^n s^{-1}|\leq |q^{2}s^{-1}|=|q^{2}q^{-1/2}|<|s|}$, and therefore  ${\left|\frac{q^n s^{-1}}{(1-q^n s^{-1})^{2}}\right|=|q^n s^{-1}|<|s|}$.  Moreover, if   $|q|^{1/2}= |s|$ then ${|q s^{-1}| =|qq^{-1/2}|=|s|}$. Therefore $\left|\frac{q s^{-1}}{(1-q s^{-1})^{2}}\right|=|q s^{-1}|=|s|$.  We conclude that 
\begin{equation}\label{eq1b}
 \quad \left|\sum_{n>0} \frac{q^n s}{(1-q^n s)^2}+\frac{q^n s^{-1}}{(1-q^n s^{-1})^2}-2 \frac{q^n }{1-q^n} \right|=|s|.
\end{equation}
It remains to  consider  the term $\frac{s}{(1-s)^2}$.  If  $|s| <1$ then   we have ${\left|\frac{s}{(1-s)^2}\right|= |s|}$. Combining with \eqref{eq:majorationnormXs}, \eqref{eq1} and \eqref{eq1b} respectively, we obtain the result when $|q|^{1/2}<|s|<1$ and $|q|^{1/2}=|s|<1$ respectively.
If  $|s|=1$ and $s \neq 1$ then $|1-s|\leq 1$. Thus, $\left|\frac{s}{(1-s)^2}\right|\geq |s|=1$, which, combined with \eqref{eq:majorationnormXs} and \eqref{eq1} concludes the proof.\end{proof}

\subsection{Tate and Weierstrass normal forms}
In \cite{dreyfus2019differential}, the authors generalize the results  of \cite{KurkRasch} and attach a Weierstrass normal form to the kernel curve. The following proposition proves that, with some care, their result passes to a nonarchimedean framework.

Let us consider a nondegenerate weighted  model of genus one and let us write its kernel polynomial as follows: $K\left(x,y,t\right)=\widetilde{A}_{0}(x)+\widetilde{A}_{1}(x)y+\widetilde{A}_{2}(x)y^{2}=\widetilde{B}_{0}(y)+\widetilde{B}_{1}(y)x+\widetilde{B}_{2}(y)x^{2}$ with ${\widetilde{A}_{i}(x)\in C[x]}$ and $\widetilde{B}_{i}(y)\in C[y]$. The following proposition gives a Weierstrass normal form for the kernel curve.
\begin{prop}\label{prop:uniformizationnondegeneratecase}
 Let $a \in C$ be  as in Lemma \ref{lemma:goodroot}. Let $E_1$ be the elliptic curve defined by the Weierstrass equation
 \begin{equation}
 y_1^2=4x_1^3 -g_2x_1-g_3,
 \end{equation}
 with \begin{eqnarray}\label{eqn:ginvariant}
  g_2&=&\frac{\mathfrak{D}^{(2)}(a)^2}{3}-2\frac{\mathfrak{D}^{(1)}(a)\mathfrak{D}^{(3)}(a)}{3}  \\\nonumber
 g_3&=&-\frac{\mathfrak{D}^{(2)}(a)^3}{27}+\frac{\mathfrak{D}^{(1)}(a)\mathfrak{D}^{(2)}(a)\mathfrak{D}^{(3)}(a)}{9} -\frac{\mathfrak{D}^{(1)}(a)^2\mathfrak{D}^{(4)}(a)}{6}.
\end{eqnarray}  
Then, the rational map 
$$
\begin{array}{llll}
& E_1 & \rightarrow & E \subset \P1 (C) \times \P1 (C) \\
 & [x_1:y_1:1] &  \mapsto &  (\overline{x},\overline{y})
\end{array} $$
where $$\overline{x}=a +\frac{\mathfrak{D}^{(1)}(a)}{x_1-\frac{\mathfrak{D}^{(2)}(a)}{6}} \hbox{ and }\overline{y}=\frac{\frac{\mathfrak{D}^{(1)}(a) y_1}{2(x_1-\frac{\mathfrak{D}^{(1)}(a)}{6})^2}- \widetilde{A}_{1}\left( a +\frac{\mathfrak{D}^{(1)}(a)}{x_1 -\frac{\mathfrak{D}^{(2)}(a)}{6}}\right)    }{ 2\widetilde{A}_{2}\left(a +\frac{\mathfrak{D}^{(1)}(a)}{x_1-\frac{\mathfrak{D}^{(2)}(a)}{6}}\right)},$$ is an isomorphism of elliptic curves that sends the point $ \mathcal{O}=[1:0:0]$  in $E_1$ to the point $\left(a, \dfrac{-\widetilde{A}_{1}(a)}{2\widetilde{A}_{2}(a)}\right) \in E$.
 \end{prop}
 \begin{proof}
This is the same proof as in \cite[Proposition 18]{dreyfus2019differential}. Note that there is only one configuration here since we have chosen a root of the discriminant $|a|<1$ which can not be infinity.
 \end{proof}
 
We recall that the $J$-invariant  $J(E_1)$ of the elliptic curve $E_1$  given in a  Weierstrass form $ y_1^{2}=4x_1^{3}-g_2 x_1 -g_3$ equals to $J(E_{1})=12^{3}\frac{g_2^{3}}{g_2^{3}-27g_3^{2}}$.  For a weighted model of genus one, the $J$-invariant $J(E)$ of the  kernel curve  has modulus strictly greater than $1$ by  Lemma \ref{lemma:jinvKernel}.  Since $J(E)=J(E_1)$,
 Proposition \ref{prop:Tateuniformization}  shows that  there exists $q\in C^*$  such that $0<|q|<1$ and $E_1$ is isomorphic to $E_q$. In order to explicit this isomorphism, we 
need  to understand how one passes  from  to a Tate normal form to a Weierstrass normal form. This is the content of the following lemmas.

\begin{lemma}\label{lemma:tatetoweierstrass}[\S 6, Page 29 in \cite{roquette1970analytic}] 
In the notation of Proposition \ref{prop:Tatecurve}, the  change of variable $X=x- \frac{1}{12}$ and ${Y=\frac{1}{2}(y-x + \frac{1}{12})}$ maps the Tate equation
$$ Y^2 +XY =X^3 +BX +\widetilde{C}$$
onto the Weierstrass equation 
$$y^2= 4x^3 -h_2 x -h_3,$$
where $h_2= \frac{1}{12}+20s_3$ and $h_3=\frac{-1}{6^3}+ \frac{7}{3}s_5$. \end{lemma}

 As detailed above, the elliptic curves $E_1$ and $E_q$ are isomorphic. The following lemma gives the form of an explicit isomorphism between theses two curves. 
\begin{lemma}\label{lemma:uniquenessweierstrassequation}
Let $y^2= 4x^3 -h_2 x -h_3$ be the Weierstrass normal form   (resp. $ Y^2 +XY =X^3 +BX +\widetilde{C}$  its Tate normal form )    of $E_q$  as  in   Lemma \ref{lemma:tatetoweierstrass} and let 
$ y_1^2= 4x_1^3 -g_2 x_1-g_3$ be the Weierstrass normal form of $E_1$ as in  Proposition \ref{prop:uniformizationnondegeneratecase}.

There exists $u \in C^*$ such that the following map
$$\begin{array}{lll} E_q& \rightarrow &E_1, \\
(X,Y) &\mapsto& (u^2(X+\frac{1}{12}), u^3(2Y +X))\end{array}$$ is an isomorphism of elliptic curves.   Moreover, the following holds
 \begin{itemize}
 \item $h_2=\frac{g_2}{u^4}$ and $h_3=\frac{g_3}{u^6}$;
 \item $\Delta_{q}=\frac{\Delta_1}{u^{12}}$ where $\Delta_{1}$ and $\Delta_q$ denote the discriminants 
 of the Weierstrass equations of $E_{1}$ and $E_q$ respectively.
 \end{itemize}
\end{lemma}

\begin{proof}
 From \cite[Proposition 3.1, Chapter III]{silverman2009arithmetic}, we deduce that  any isomorphism between the  elliptic curves $E_ 1$ and $E_q$   is given by 
$x_1=u^2x+\alpha$ and $y_1=u^3 y+\beta u^{2}x+\gamma$ with $u \in C^*$, $\alpha,\beta,\gamma\in C$. Since both equations are in Weierstrass normal form, we necessarily have $\alpha=\beta=\gamma=0$. This  proves the first point. From $ y_1^2= 4x_1^3 -g_2 x_1-g_3$, we substitute $x_1,y_1$ by $x,y$ to find
$$u^6  y^2= 4  u^6 x^3 -g_2 u^2 x-g_3.$$
Dividing the both sides by $u^6$ we find $h_2=\frac{g_2}{u^4}$ and $h_3=\frac{g_3}{u^6}$. The assertion on the discriminants follows from  $\Delta_{q}=
h_2^3-27h_3^2$ and $\Delta_1=
g_2^3-27g_3^2$. 
\end{proof}

The  lemma below gives some  precise estimate  for the norms of  $\Delta_q=h_{2}^{3}-27h_{3}^{2}$ and ${\Delta_1 =g_{2}^{3}-27g_{3}^{2}}$,  the discriminants of the elliptic curves $E_q ,E_1$, and the element  $u$ defined in Lemma~\ref{lemma:uniquenessweierstrassequation}.

 \begin{lemma}\label{lem2}
 The following statement  hold:
 \begin{itemize}
 \item $|\Delta_{q}|=|q|$,  with $|h_2 -\frac{1}{12}|=|q|$ and $|h_3 -(-\frac{1}{6^{3}})| =|q|$;
 \item $|\Delta_1|=|q|$  with  $|g_2 -\frac{4}{3}|<1$,  $|g_3-(-\frac{8}{27})|<1$;
 \item $|u|=1$;
 \item $|\mathfrak{D}^{(1)}(a)|\in ]\, |q|^{1/2},1[$.
 \end{itemize}
 \end{lemma}
 
\begin{proof}
 Following \cite[Pages 29-30]{roquette1970analytic}, we  find $|\Delta_{q}|=|q|, |s_3|=|q|=|s_5|$.  Combining  the latter norm estimates with  Lemma \ref{lemma:tatetoweierstrass}, we find $|h_2 -\frac{1}{12}|=|q|$ and $|h_3 -(-\frac{1}{6^{3}})|=|q|$. \\
Let us prove the second point.  It follows from \eqref{eq:alphaibetai} that $|1-\a_2|<1$ and $|\a_i|<1$ for ${i=0,1,3,4}$. By Lemma~\ref{lemma:goodroot}, $|\mathfrak{D}^{(1)}(a)|,|\mathfrak{D}^{(3)}(a)|,|\mathfrak{D}^{(4)}(a)|<1, |\mathfrak{D}^{(2)}(a) -2|<1$. Combining these norm estimates with  \eqref{eqn:ginvariant}, we find  $|g_2 -\frac{4}{3}|<1$,  $|g_3-(-\frac{8}{27})|<1$. Since $|J(E_1)|=|J(E_q)|=|\frac{12^3g_2}{\Delta_1}|= |\frac{12^3h_2}{\Delta_{q}}|$ and $|g_2|=|h_2|=1$, we find $|\Delta_{q}|=|\Delta_1|=|q|$. By Lemma  \ref{lemma:uniquenessweierstrassequation},  $\Delta_{q}=\frac{\Delta_1}{u^{12}}$, and then $|u|=1$.
\\
Let us prove the last point. Let us expand $\Delta_1 =g_{2}^{3}-27g_{3}^{2}$ with the expression of $g_{2},g_{3}$ given in \eqref{eqn:ginvariant}:
$$\begin{array}{lll}
\Delta_1 &=&\left(\frac{\mathfrak{D}^{(2)}(a)^2}{3}-2\frac{\mathfrak{D}^{(1)}(a)\mathfrak{D}^{(3)}(a)}{3}\right)^{3}-27\left(\frac{-\mathfrak{D}^{(2)}(a)^3}{27}+\frac{\mathfrak{D}^{(1)}(a)\mathfrak{D}^{(2)}(a)\mathfrak{D}^{(3)}(a)}{9}-\frac{\mathfrak{D}^{(1)}(a)^2\mathfrak{D}^{(4)}(a)}{6}\right)^{2}\\
&=&\frac{\mathfrak{D}^{(2)}(a)^6}{27}-\frac{2\mathfrak{D}^{(1)}(a)\mathfrak{D}^{(2)}(a)^{4}\mathfrak{D}^{(3)}(a)}{9}+\frac{4\mathfrak{D}^{(1)}(a)^{2}\mathfrak{D}^{(2)}(a)^{2}\mathfrak{D}^{(3)}(a)^{2}}{9}-\frac{8\mathfrak{D}^{(1)}(a)^{3}\mathfrak{D}^{(3)}(a)^{3}}{27}  \\
&&-\frac{\mathfrak{D}^{(2)}(a)^{6}}{27}-\frac{\mathfrak{D}^{(1)}(a)^{2}\mathfrak{D}^{(2)}(a)^{2}\mathfrak{D}^{(3)}(a)^{2}}{3}-\frac{3\mathfrak{D}^{(1)}(a)^{4}\mathfrak{D}^{(4)}(a)^{2}}{4}+\frac{2\mathfrak{D}^{(1)}(a)\mathfrak{D}^{(2)}(a)^{4}\mathfrak{D}^{(3)}(a)}{9}\\
&&-\frac{\mathfrak{D}^{(1)}(a)^{2}\mathfrak{D}^{(2)}(a)^{3}\mathfrak{D}^{(4)}(a)}{3}+\mathfrak{D}^{(1)}(a)^{3}\mathfrak{D}^{(2)}(a)\mathfrak{D}^{(3)}(a)\mathfrak{D}^{(4)}(a)\\
&=& \frac{\mathfrak{D}^{(1)}(a)^{2}\mathfrak{D}^{(2)}(a)^{2}\mathfrak{D}^{(3)}(a)^{2}}{9}-\frac{8\mathfrak{D}^{(1)}(a)^{3}\mathfrak{D}^{(3)}(a)^{3}}{27}-\frac{3\mathfrak{D}^{(1)}(a)^{4}\mathfrak{D}^{(4)}(a)^{2}}{4}\\
&&-\frac{\mathfrak{D}^{(1)}(a)^{2}\mathfrak{D}^{(2)}(a)^{3}\mathfrak{D}^{(4)}(a)}{3}+\mathfrak{D}^{(1)}(a)^{3}\mathfrak{D}^{(2)}(a)\mathfrak{D}^{(3)}(a)\mathfrak{D}^{(4)}(a).
\end{array} $$
Since $|\mathfrak{D}^{(1)}(a)|,|\mathfrak{D}^{(3)}(a)|,|\mathfrak{D}^{(4)}(a)|<1, |\mathfrak{D}^{(2)} -2|<1$ , the previous expression is a sum of terms that are all strictly smaller in norm than $|\mathfrak{D}^{(1)}(a)|^2$. This proves that  $|\Delta_1| =|q| < |\mathfrak{D}^{(1)}(a)|^2$.\\
\end{proof}

The following estimate will be required to uniformize the generating function.
\begin{lemma}\label{lem:normestimateu2mathfrakD}
In the notation of Theorem \ref{cor:jinvariant}, we have  $|\frac{u^2}{12}  -\frac{\mathfrak{D}^{(2)}(a)}{6}|<|\mathfrak{D}^{(1)}(a)|$.
\end{lemma}
\begin{proof}
Using \eqref{eqn:ginvariant} and the norm estimate on the $\mathfrak{D}^{(i)}(a)$'s, we get
\begin{equation}\label{eqn:normestimateginvariant}
g_2 =\frac{\mathfrak{D}^{(2)}(a)^2}{3} + \mathfrak{D}^{(1)}(a) \omega, \quad g_3 =\frac{-\mathfrak{D}^{(2)}(a)^3}{27} + \mathfrak{D}^{(1)}(a) \omega',
\end{equation}
where $|\omega|, |\omega'|<1$.  This proves that 
$$
\frac{g_3}{g_2} =\frac{-\mathfrak{D}^{(2)}(a)}{9} +\mathfrak{D}^{(1)}(a) \omega''
$$
with $|\omega''|<1$. Then, we find
$$\left| \frac{u^2}{12} - \frac{\mathfrak{D}^{(2)}(a)}{6}\right|= \left| \frac{u^2}{12} + \frac{3g_3}{2g_2} -\frac{3g_3}{2g_2} -  \frac{\mathfrak{D}^{(2)}(a)}{6}\right| \leq \max\left(\left|\frac{u^2}{12}+\frac{3g_3}{2g_2}\right|, \left|\frac{3}{2} \mathfrak{D}^{(1)}(a) \omega''\right|\right).$$
Finally, with  the norm estimate of Lemma \ref{lem2}, it is sufficient to show that $ |\frac{u^2}{12}+\frac{3g_3}{2g_2}|\leq |q|$.
By Lemma \ref{lemma:uniquenessweierstrassequation}, we have $\frac{u^2}{12}=\frac{g_3h_2}{12g_2h_3}$. By Lemma \ref{lem2}, $|h_2 -\frac{1}{12}|=|q|$ and $|h_3 -(-\frac{1}{6^{3}})|=|q|$. Then, by Lemma \ref{lem2} again, we find 
\begin{multline*}
\left|\frac{u^2}{12}+\frac{3g_3}{2g_2}\right|=\left|\frac{g_3h_2}{12g_2h_3}+\frac{3g_3}{2g_2}\right|=\left|\frac{g_3}{g_2}\right| \left|\frac{h_2}{12 h_3}+\frac{3}{2}\right|=\left|\frac{h_2+18h_3}{12 h_3}\right|\\
=|h_2+18h_3 |=\left|\left(h_2-\frac{1}{12}\right)+18\left(h_3 -\left(-\frac{1}{6^{3}}\right)\right) \right|\leq \max  \left(\left|h_2-\frac{1}{12}\right|,\left|h_3 -\left(-\frac{1}{6^{3}}\right)\right|\right)\leq |q| .
\end{multline*} 

\end{proof}

%%%%%%%%%%%%%%%%%%%%%%%%%%%%%%%%%%%%%%%%%%%%%%%%%%%%%%%%%%%%%%%%%%%%%%%%%%%%%%%%%%%%%%%%%%%
%%%%%%%%%%%%%%%%%%%%%%%%%%%%%%%%%%%%%%%%%%%%%%%%%%%%%%%%%%%%%%%%%%%%%%%%%%%%%%%%%%%%%%%%%%%%
\section{Difference Galois theory}\label{sec:differenceGaloistheory}
%%%%%%%%%%%%%%%%%%%%%%%%%%%%%%%%%%%%%%%%%%%%%%%%%%%%%%%%%%%%%%%%%%%%%%%%%%%%%%%%%%%%%%%%%%%
%%%%%%%%%%%%%%%%%%%%%%%%%%%%%%%%%%%%%%%%%%%%%%%%%%%%%%%%%%%%%%%%%%%%%%%%%%%%%%%%%%%%%%%%%%%%
In this section, we establish some criteria to guaranty the transcendence of functions satisfying a difference  equation of order $1$. This criteria is based on the Galois theory of difference fields as developed in \cite{VdPS97} but generalizes some of the existing results in the literature, for instance the assumption  that the field of constants is   algebraically closed
 (see for instance Theorem \ref{thm:abstractdifftransGaloiscriteria}).

The  algebraic framework of this section is difference algebra and more precisely the notion of difference fields.
A difference field is a pair $(K,\s)$ where $K$ is a field and $\s$ is  an automorphism of $K$. The  field
$\s$-constants $K^\s$ of $(K,\s)$ is formed by the elements $f \in K$ such that $\s(f)=f$. An extension $(K,\s_K) \subset (L,\s_L)$ of difference fields is a field 
extension $K\subset L$ such that $\s_L$ coincides with $\s_K$ on $K$. If there is no confusion, we shall denote by $\s$ the automorphism $\s_K$ and $\s_L$. For a complete introduction on difference algebra, we shall refer to \cite{Cohndifferencealg}.\par 

\subsection{Rank one difference equations}

In this section, we focus  on  rank one difference equations.
\begin{lemma}\label{lemma:algoverdifferenceconstants}
Let $(K,\s) \subset (L,\s)$  be an extension of difference fields such that $L^\s=K^\s$. Let $x\in L$. The following statements are equivalent
\begin{enumerate}
\item $x$ is algebraic over $K^\s$;
\item there exists $r \in \N^*$ such that $\s^r(x)=x$.
\end{enumerate}
\end{lemma} 

\begin{proof}
Assume that $x$ is algebraic over $K^\s$. Then, $\s$ induces a permutation on the set of roots of the minimal polynomial of $x$ over $K^\s$. Thus, there exists $r \in \N^*$ such that $\s^r(x)=x$. Conversely, if there exists $r \in \N^*$ such that $\s^r(x)=x$, the polynomial $P(X)=\prod_{i=0}^{r-1} (X-\s^i(x)) \in L[X]$ is fixed by $\s$ and thereby $P(X) \in L^\s[X]=K^\s[X ]$. Since $P(x)=0$, we have proved that $x$ is algebraic over $K^\s$.
\end{proof}

\begin{lemma}\label{lem:transccriteriaforlog}
Let $(K,\s) \subset (L,\s)$  be an extension of difference fields such that $L^\s=K^\s$. Let $f \in L$ and $0\neq c \in K$, such that 
$\s(f)=f +c$. The following statements are equivalent
\begin{enumerate}
\item $ f \in K$;
\item $f$ is algebraic over $K$;
\item There exists  $\alpha \in K$ such that $\s(\alpha)=\alpha +c$.
\end{enumerate}
Moreover, let  $\overline{K}$ be the algebraic closure of $K$  endowed  with a structure of $\s$-field extension of $K$.  For all $\alpha \in \overline{K}, i \in \Z$ we denote by $\alpha_i$ the element of $\overline{K}$ such that $\sigma^{i}( f-\alpha)=f-\alpha_i$.  If $f$ is transcendental over $K$ then for $i,j \in \Z$ such that  $i \neq j$, the elements $\alpha_j$ and $\alpha_i$ are distinct.
\end{lemma}

\begin{proof}
Let us prove the first part of the proposition. The first statement implies trivially the second one. Assume that $f$ is algebraic over $K$ and let $P(X)=X^n +a_{n-1}X^{n-1}+\dots a_0 \in K[X]$ be its minimal polynomial over $K$. Note that $n \neq 0$. Using $\s(f)-f=c$ and $P(f)=0$, we find that 
$\s(P(f))-P(f)=0=( nc +\s(a_{n-1})-a_{n-1}) f^{n-1} +b_{n-2}f^{n-2}+\dots+b_0$ with $b_i \in K$ for $i=0,\dots,n-2$. By minimality of $P(X)$, we find that $ \s(a_{n-1})-a_{n-1}=-nc$ with $a_{n-1} \in K$. Then, $\s(\alpha)- \alpha=c$ with $\alpha=\frac{a_{n-1}}{-n} \in K$. We have shown that the second statement implies the third. Finally, assume that there exists $\alpha \in K$ such that $\s(\alpha)=\alpha +c$. With $\s(f)-f=c$, we find that $\s(\alpha-f)=\alpha -f$. This gives that $\alpha -f \in L^\s=K^\s$ and the element $f$ belongs to $K$. 

Now, let us assume that $f$ is transcendental over $K$. Suppose to the contrary that there exist $\alpha \in \overline{K}$ and $i >j \in \Z$ such that  $$\alpha_i=(\s^i(\alpha) -c - \sigma(c) -\dots-\sigma^{i-1}(c))= \alpha_j=(\s^j(\alpha) -c - \sigma(c) -\dots-\sigma^{j-1}(c)).$$ 
The latter equality gives $\s^r(\beta)- \beta= \gamma$ where $r=i-j >0$, $\beta =\s^j(\alpha)$ and $\gamma =\s^{i-1}(c) +\dots +\s^{j}(c)$. Since $\alpha$ is algebraic over $K$, the same holds for $ \beta$. Let ${P(X)=X^n +a_{n-1}X^{n-1}+\dots +a_0 \in K[X] \setminus K}$ be the  minimal polynomial of $\beta$ over $K$. Using the fact that $\s^r(\beta)-\beta=\gamma$ and the minimality of $P$, we conclude, as above, that $\s^r(a_{n-1})-a_{n-1}=-n\gamma$, that is $\s^r(\tilde{\beta})-\tilde{\beta}=\gamma$ where $\tilde{\beta}=\frac{a_{n-1}}{-n} \in K$. Combining this equality with 
$\s^r(\s^j(f))-\s^j(f)=\gamma$, we find that $\tilde{\beta}- \s^j(f) \in L$ is fixed by $\s^r$. By Lemma~\ref{lemma:algoverdifferenceconstants}, this means that $\tilde{\beta}- \s^j(f) $ is algebraic over $K^\s$, which yields to $f$ algebraic over $K$. We find a contradiction. 
\end{proof}

\begin{lemma}\label{lem:transclogpolynomialinlogcocycle}
Let $(K,\s) \subset (L,\s)$  be an extension of difference fields such that $L^\s=K^\s$. Let $f \in L$ and $0\neq c \in K$, such that 
$\s(f)=f +c$.  Assume that $f$ is transcendental over $K$. If there exists $g \in K(f)$ such that
 $\s(g)-g \in K[f]$, then  $g \in K[f]$.
\end{lemma}
\begin{proof}
Let $\overline{K}$ be an algebraic closure of $K$, endowed with a structure of $\s$-field extension of $K$.  Since  $f$ is transcendental over $K$, we can  write  a partial fraction decomposition of $g\in \overline{K}(f)$. Let  $R$ be  the  largest  integer such that there exists $\alpha \in \overline{K}$  so that the element $\frac{1}{(f -\alpha)^R}$ appears in the partial fraction decomposition of $g$.  Suppose to the contrary that $R >0$ and let  $\alpha \in \overline{K}$ such that $\frac{1}{(f -\alpha)^R}$ appears in the partial fraction decomposition of $g$. We deduce from  Lemma~\ref{lem:transccriteriaforlog}  applied to $K$ and $f$, that the elements  $ \{\alpha_i, i\in \Z\}$  are all distinct. Then, there exists $N$, the largest  integer such that $ \s^N(\frac{1}{(f -\alpha)^R})$ appears in the partial fraction decomposition of $g$. The element $ \s^{N+1}(\frac{1}{(f -\alpha)^R})$ appears in the partial fraction decomposition of $\s(g)$. This proves that   $ \s^{N+1}(\frac{1}{(f -\alpha)^R})  $ appears in the partial fraction decomposition of $\s (g)-g$. A contradiction with  $\s(g)-g \in K[f]$. This proves that $g \in K[f]$.
\end{proof}

%%%%%%%%%%%%%%%%%%%%%%%%%%
%%%%%%%%%%%%%%%%%%%%%%%%%%
\subsection{Differential transcendence criteria}\label{sec31}
%%%%%%%%%%%%%%%%%%%%%%%%
%%%%%%%%%%%%%%%%%%%%%%%

In this section, a $(\s,\partial,\Delta)$-field $K$ is a difference field $(K,\s)$ endowed with two   derivations $\partial,\Delta$ commuting with $\s$ such that $\partial\Delta-\Delta \partial =c_K\partial$ with $c_K \in K^\s$. We assume that $\partial$ is nontrivial on $K$, that is, it is not the zero derivation. The element $c_K$ has to be considered as part of the data of the notion of $(\s,\partial,\Delta)$-field. An extension of $(\s,\partial,\Delta)$-fields is an  inclusion of two $(\s,\partial,\Delta)$-fields  $(K,\s_K,\partial_K,\Delta_K ) \subset (L,\s_L,\partial_L,\Delta_L )$ such that
\begin{itemize}
\item $K\subset L$ is a field extension;
\item $\s_K,\partial_K,\Delta_K$ are the restrictions of $\s_L,\partial_L,\Delta_L$ to $K$;
\item $c_K=c_L$. 
\end{itemize}  
If there is no confusion, we shall omit the subscripts $~_K,~_L$. If $\s$ is the identity, we  shall speak of  $(\partial,\Delta)$-fields, $(\partial,\Delta)$-fields extension for short.

\begin{ex}
As proved in \S \ref{sec:merofunctiontate}, the following fields  are $(\sigma,\partial,\Delta)$-fields, that correspond respectively to the framework of the genus zero and genus one kernel curve. Remind that $\s_\q$ denote the automorphism of $\cM er(C^*)$ defined by $f(s)\mapsto f(\q s)$ and  $C_\q$ denote the field of meromorphic functions fixed by $\s_\q$.
 In the two examples, we have $\Delta_{\q,t}=\partial_t(\q)\ell_\q(s)\partial_s +\partial_t$ where $\ell_\q$ is the so called $\q$-logarithm. That is,    an element of $\cM er(C^*)$  satisfying $\s_\q (\ell_\q)=\ell_\q+1$, and $c_{K}=\partial_t(\q) \partial_s(\ell_\q) \in C_\q$.
 \begin{itemize}
\item Let $\q\in C^*$ with $|\q|\neq 1$. Then, the inclusion 
 $$(C_{\q}(s,\ell_{\q}),\sigma_{\q},\partial_s,\Delta_{t,\q} ) \subset (\cM er(C^*),\sigma_{\q},\partial_s,\Delta_{t,\q} )$$
 is an extension of $(\sigma,\partial,\Delta)$-fields.
\item  Let $\q$ and $q$ two elements of $C^*$ such that $|q|,|\q|\neq 1$, that are \emph{multiplicatively independent}, that is,  there are no $r,l \in \Z^{2} \setminus (0,0)$ such that $q^r=\q^l$. Since ${C_\q\subset \mathcal{M}er(C^{*})}$ and  $C_{q}\subset \mathcal{M}er(C^{*})$, we 
consider $C_{\q}.C_{q}\subset \mathcal{M}er(C^{*})$,   the field  compositum  of  $C_{\q}$ and $C_{q}$ inside $ \mathcal{M}er(C^{*})$. Then, the inclusion
$$(C_{\q}.C_{q}(\ell_{\q},\ell_{q}),\sigma_{\q},\partial_s,\Delta_{t,\q} ) \subset (\cM er(C^*),\sigma_{\q},\partial_s,\Delta_{t,\q} )$$
 is an extension of $(\sigma,\partial,\Delta)$-fields.
\end{itemize}
\end{ex}

\begin{defi}\label{defi:diffalgtwoderivations}
Let $(K,\partial,\Delta) \subset (L,\partial,\Delta)$. An element $f \in L$ is said to be $(\partial,\Delta)$-differentially algebraic over $K$ if there exists $N\in \N$, such that the elements 
\begin{itemize}
\item  $ \partial^i (f)$ for $i \leq N$ are algebraically dependent over $K$ if  $\Delta$ is a $K$-multiple of $\partial$; 
\item  $\partial^i\Delta^j(f)$ for $i,j\leq N$ are algebraically dependent over $K$ otherwise.\end{itemize}
 Otherwise, we will say that $f$ is  $(\partial,\Delta)$-transcendental over $K$.
\end{defi}

\begin{rmk}
Note that since $ \partial\Delta-\Delta \partial=c\partial$ with $c \in K^{\sigma}\subset K$, the $(\partial, \Delta)$-field extension of $K$ generated by some element $f \in L$ coincides with the field extension of $K$ generated by the set  
$\{ \partial^i\Delta^j(f), \mbox{ for } i,j \in \N \}$.
\end{rmk}

Let us make a remark concerning the field of definition of the coefficients of the differential polynomials. 

\begin{rmk}\label{rmk:CdifftransQdifftrans}
Let $(K,\partial,\Delta) \subset (K',\partial,\Delta)\subset (L,\partial,\Delta)$ and assume that $K'$ is a field generated over $K$ by elements that are $(\partial,\Delta)$-differentially algebraic  over $K$.
By \cite[Proposition 8, Page 101]{kolchin1973differential}, $f\in L$  is $(\partial,\Delta)$-differentially transcendental over $K$ if and only if it is  $(\partial,\Delta)$-differentially transcendental over $K'$.
\end{rmk}

The following lemma will be crucial in many arguments: 
\begin{lemma}\label{lem:lineardisjonctionoverconstant}
If $K \subset M$ is a $\s$-field extension  such that $M^{\sigma}=K$ and $K \subset L$ is a $\s$-field extension with $L^\s=L$. Then $M$ and $L$ are linearly disjoint over $K$.
\end{lemma}
\begin{proof}
Let $c_1,\dots,c_r \in L$ be $K$-linearly independent elements,  that become dependent 
over $M$. Up to a permutation of the $c_i$'s,  a minimal linear relation among these elements over $M$ has the following form
\begin{equation}\label{eq:minimalliaison}
c_1 +\sum_{i=2}^r \lambda_i c_i =0,
\end{equation}  
with $\lambda_i \in M$ for $i=2,\dots,r$. Computing $\s(\eqref{eq:minimalliaison})-\eqref{eq:minimalliaison}$, we find
$$
\sum_{i=2}^r (\s(\lambda_i)-\lambda_i) c_i =0.
$$ 
By minimality, $\s(\lambda_i)=\lambda_i$ and $\lambda_i \in M^{\s}=K$. By $K$-linear independence of the $c_i$, we find that $\lambda_i=0$ for $i=2,\dots,r$ and then $c_1=0$. A contradiction.
\end{proof}

The following statement, whose proof is due to Michael Singer, is a  version of an old theorem of Ostrowski \cite{Ost46, Kol68} and  its proof follows the lines of   the proof of \cite[Proposition~3.6]{DHRS}. In this last paper,  it was assumed that $K^\s$ is algebraically closed, which is not the case in  this article.  One could use the powerful scheme-theoretic tools  developed in  \cite{OvchinnikovWibmer} to prove the result in our more general setting.  Instead we will argue in a more elementary way to reduce  Theorem~\ref{thm:abstractdifftransGaloiscriteria} to the case where $K^\s$ is algebraically closed. 
\begin{thm}\label{thm:abstractdifftransGaloiscriteria} Let $(K,\s,\partial, \Delta) $ be a  $(\s,\partial, \Delta)$-field such that 
$K^\s$ is relatively algebraically closed in $K$, that is there are no proper algebraic extension of $K^\s$ inside $K$. Let $(L,\s,\partial, \Delta)$ be a $(\s,\partial, \Delta)$-ring extension of $(K,\s,\partial, \Delta) $.  Let $f \in L$ and $b \in K$ such that $\s(f)=f+b$. If $f$ is $(\partial,\Delta)$-differentially algebraic over $K$ then there exist $\ell_{1},\ell_{2}\in \N$,  $c_{i,j} \in K^\s$ not all zero and $g \in K$ such that 
\begin{equation}\label{eq:generaltelescoper}
\sum_{\substack{
0\leq i\leq \ell_{1},\\ 0\leq j\leq \ell_{2}}}c_{i,j} \partial^i\Delta^j(b)= \s(g)-g.
\end{equation}  
Furthermore, we may take $\ell_2=0$ in the case where $\partial$ and $\Delta$ are $K$-linearly dependent.  We call \eqref{eq:generaltelescoper} a telescoping relation for $b$.
\end{thm}

The proof of this result depends on results from the Galois theory of linear difference equations and we will refer to \cite[Appendix A]{DHRS} and the references given there for relevant facts from this theory. Let $(K, \sigma)$ be a difference field and consider the system of difference equations 
\begin{eqnarray} \label{eq:addsys}
\sigma(y_0) - y_0 = b_0,  \ldots , \sigma(y_n) - y_n = b_n ,\  \mbox{ with } \ b_0, \ldots , b_n  \in K.
\end{eqnarray}
Let us see \eqref{eq:addsys} as a system $\sigma (Y)=AY$, where $A\in \GL_{2(n+1)}(K)$ is a diagonal bloc matrix $A=\mathrm{Diag}(A_{0},\dots,A_{n})$ with $A_{i}=\begin{pmatrix}
1 & b_{i} \\ 
0& 1
\end{pmatrix}$ which correspond to the equation ${\sigma(y_i) - y_i = b_i}$. A Picard-Vessiot extension for $\sigma (Y)=AY$ is a difference ring extension $(R,\sigma)$ of $(K,\sigma)$ such that: 
\begin{itemize}
\item there exists $U\in \GL_{2(n+1)}(R)$ such that $\sigma (U)=AU$;
\item $R$ is generated as a $K$-algebra by the entries of $U$ and $\det(U)^{-1}$;
\item $R$ is a simple difference ring, that is, the $\sigma$-ideals of $R$ are $\{0\}$ and $R$.
\end{itemize}
 We will need the following result.
 
 \begin{lemma}[Proposition A.9 in \cite{DHRS}] \label{lemma:ost} Assume that $(K,\s)$ is a difference field with $K^\s$ algebraically closed. Let $R$ be a Picard-Vessiot extension  for the system \eqref{eq:addsys} and $z_0, \ldots , z_n \in R$ be solutions of this system.  If $z_0, \ldots , z_n $ are algebraically dependent over $K$, then there exist $c_i \in K^\s$, not all zero,  and $g \in K$ such that 
\[c_0 b_0 + \ldots + c_n b_n = \s(g) - g.\]  \end{lemma}

Before proving Theorem \ref{thm:abstractdifftransGaloiscriteria}, we give a slight generalization of Lemma \ref{lemma:ost}.
\begin{lemma}\label{lemma:relativclosedconstant}
Let $(K,\s)$ be  a difference field with $K^\s$ relatively  algebraically closed in $K$ and let $b_0,\dots, b_n$ be some elements in $K$. Let $(L,\s)$ be a $\s$-ring extension of $(K,\s)$. Let  $z_0, \ldots , z_n \in L$ be solutions of  $\sigma(z_i)-z_i=b_i$.  If $z_0, \ldots , z_n $ are algebraically dependent over $K$, then there exist $c_i \in K^\s$, not all zero,  and $g \in K$ such that 
\[c_0 b_0 + \ldots + c_n b_n = \s(g) - g.\]
\end{lemma}
\begin{proof}
 Let $\mathbf{k}$ be the algebraic closure of $K^\s$. We extend $\s$ to be the identity on $\mathbf{k}$\footnote{On the other hand, there is no unique procedure to extend a field automorphism of $K^\s$ to  the algebraic closure $\mathbf{k}$. Indeed, these extensions are controlled by the Galois group of the field $\mathbf{k}$ over $K^\s$.}.  Under the assumption  that $K^\s$ is relatively algebraically closed, the ring $\widetilde{K} = K\otimes_{K^\s}\mathbf{k}$ is an integral domain and in fact is a field.   We have $\widetilde{K}^\sigma = \mathbf{k}$. Let $\widetilde{L}= L\otimes_{K^\s}\mathbf{k}$. We then have a natural inclusion of $\widetilde{K} \subset \widetilde{L}$. Let $S=\widetilde{K}[z_0,\dots,z_n] \subset \widetilde{L}$. It is easily seen that $S$ is a $\s$-ring extension of $\widetilde{K}$. Let $I$ be a maximal difference ideal in $S$ and let $R = S/I$.  For each $r = 0, \ldots , n$, let $u_r$ be the image of $z_r$ in $R$. Since $\widetilde{K}^\sigma = \mathbf{k}$ is algebraically closed and  $R$ is a simple difference ring, we have that $R$ is a Picard-Vessiot ring for the system associated to  $\s(y_r) -y_r= b_r$, $r= 0 , \dots, n$, over $\widetilde{K}$. The elements $u_0,\dots,u_n$ are algebraically dependent over $K$ and solutions of $\s(
y_r) -y_r= b_r, r= 0 , \dots, n$. Lemma~\ref{lemma:ost}  proves that  there exist $c_{i} \in \mathbf{k}$, not all zero, and $g \in \widetilde{K}$ such that 
$$\sum_{0\leq i\leq n} c_{i} b_i= \s(g)-g.$$
 Let $\{d_r\} \subset \mathbf{k}$ be a $K^\s$-basis of $\mathbf{k}$. By Lemma \ref{lem:lineardisjonctionoverconstant}, it is also  a $K$-basis of $\widetilde{K}$.  We may write each $c_{i}$ and $g$ as
\[ c_{i} = \sum_r c_{i,r}d_r \text{    and    } g = \sum_r g_rd_r\]
for some $c_{i,r} \in K^\s$ and $g_r \in K$. Since not all the $c_{i}$ are zero, there exists $r$ such that $c_{i,r}$ are not all zero. For this $r$, we have
\[ \sum_{i \leq n} c_{i,r} b_i = \sigma(g_r) - g_r.\]
This yields the conclusion of the proof.
\end{proof}

\begin{proof}[ Proof of Theorem~\ref{thm:abstractdifftransGaloiscriteria}]

 Assuming that $f$ is $(\partial,\Delta)$-differentially algebraic over $K$, there is some finite set $\{\partial^{i_0}\Delta^{j_0}(f), \ldots , \partial^{i_n}\Delta^{j_n}(f)\} \subset L$ of elements that are algebraically dependent over $K$. Note that  $j_k=0$ for all $k$ if $\Delta$ is $K$-linearly dependent from $\partial$. Since $\s$ commutes with $\Delta$ and $\partial$, we have for all $r=0,\dots,n$,
\[ \sigma(\partial^{i_r}\Delta^{j_r}(f)) - \partial^{i_r}\Delta^{j_r}(f) = \partial^{i_r}\Delta^{j_r}(b).\]
To conclude it remains to apply Lemma \ref{lemma:relativclosedconstant} with $z_r= \partial^{i_r}\Delta^{j_r}(f)$ and $b_r =\partial^{i_r}\Delta^{j_r}(b)$ for $r=0,\dots,n$.
\end{proof}

%%%%%%%%%%%%%%%%%%%%%%%%%%%%%%%%%%%%
%%%%%%%%%%%%%%%%%%%%%%%%%%%%%%%%%
\section{Meromorphic functions on a Tate curve and their derivations}\label{sec:merofunctiontate}
%%%%%%%%%%%%%%%%%%%%%%%%%%%%%%%%%%
%%%%%%%%%%%%%%%%%%%%%%%%%%%%%%%%%%

In this section we translate the galoisian criteria of Theorem \ref{thm:abstractdifftransGaloiscriteria} in the context of elliptic functions field. We start by defining the derivations.
 Studying the transcendence properties of the $\q$-logarithm, we then perform a descent on the field of coefficients and on the number of derivations involved in the telescoping relation.
  
\subsection{Derivation on nonarchimedean elliptic functions field}\label{sec:tderivationtatecurve}
%%%%%%%%%%%%%%%%%%%%%%%%%%%%%%%%%%%%%%%
%%%%%%%%%%%%%%%%%%%%%%%%%%%%%%%%%%%%%%%
Let $\q \in C^{*}$ such that $|\q|\neq 1 $ and let   $\s_\q$ denote the automorphism of $\cM er(C^*)$ defined  by $\s_\q(f(s))=f(\q s)$. We denote 
by $C_\q$ the field of meromorphic functions fixed by $\s_\q$. By Proposition \ref{prop:Tatecurve}, it is the 
field of rational  functions on the Tate curve $E_{\q}$ or $E_{1/\q}$, depending whether $|\q|< 1 $ or $|\q|>1$.  In this section, we construct, as in   \cite[\S 2]{DiVizioHardouinPacific} a derivation of these functions that encode their $t$-depencies and commute with $\s_\q$. 

The fact that $\partial_s=s\frac{d}{ds}$ acts on $\cM er(C^*)$, and its  commutation with $\sigma_{\q}$ is straightforward. Unfortunately,  the $t$-derivative of $\q$ may be nontrivial, implying a more complicated commutation rule between  $\partial_t=t\frac{d}{dt}$ and $\sigma_{\q}$. 
More precisely, we have
$$\begin{array}{l}
\partial_s\circ\s_\q=\s_\q\circ\partial_s;\\
\partial_t\circ\s_\q= \partial_t(\q)\s_\q \circ \partial_s +\s_\q \circ \partial_t .
\end{array}
$$

The following statement holds.

\begin{lemma}\label{lemma:partialsconstant}
The $\partial_s$-constants $\cM er(C^*)^{\partial_s}=\{ f \in  \cM er(C^*) | \partial_s(f)=0\}$ of $\cM er(C^*)$ are precisely the constant functions $C$. 
\end{lemma} 

 Next Lemma introduces a twisted $t$-derivation that commutes with~$\s_\q$. Remind that the $q$-logarithm $\ell_{\q}$ has been defined in $\S \ref{sec43}$.

\begin{lemma}[Lemma 2.1 in \cite{DiVizioHardouinPacific}]    \label{lemma:goodderivationgenus1}
The following derivations of $\cM er(C^*)$
$$
\left\{
\begin{array}{l}
\partial_s\\
\Delta_{t,\q}= \partial_t(\q)\ell_\q(s)\partial_s+\partial_t,
\end{array}\right.
$$
commute with $\s_{\q}$. Moreover,  we have 
$$
\partial_s \Delta_{t,q} - \Delta_{t,q} \partial_s=\partial_t(\q) \partial_s(\ell_\q) \partial_s,
$$
where $\partial_t(\q) \partial_s(\ell_\q) \in C_\q$.
\end{lemma}

\begin{rmk}\label{rmk:derivateqlog}
Note that since $\partial_{s},\Delta_{t,\q}$ commute with $\s_\q$, we can derive the equation ${\s_\q(\ell_\q)=\ell_\q+1}$ to find  $\s_\q (\partial_s(\ell_\q))=\partial_s\left(\ell_\q\right)$ and $\s_\q(\Delta_{t,\q}(\ell_\q))=\Delta_{t,\q}(\ell_\q)$. We then conclude that   $\partial_s(\ell_\q),\Delta_{t,\q}(\ell_\q)$ belong to $ C_\q$.
\end{rmk}

The link with the iterates of $\Delta_{t,\q}$ and the derivatives $\partial_{s},\partial_{t}$ is now made in the following lemma.
\begin{lemma}\label{lemma:decompiterateDeltat}
For any $i \in \N$, there exist $c_{j,k,l} \in C_\q$ such that 
$$
\Delta_{t,\q}^i= (\partial_t(\q) \ell_\q)^i  \partial_s^i +\sum_{k=0}^{i-1}\sum_{j=0}^{k} \sum_{l=0}^i c_{j,k,l} \ell_\q^{j} \partial_s^k \partial_t^l.
$$
\end{lemma}
\begin{proof}
Let us prove the result by induction on $i$. For $i=1$, this comes from the fact that $\Delta_{t,\q}= \partial_t(\q)\ell_\q \partial_s +\partial_t$. Let us fix $i\in \N$ and assume that the result  holds for $i$.  We find
$$
\Delta_{t,\q}^{i+1}= ( \partial_t(\q)\ell_\q \partial_s +\partial_t)\left((\partial_t(\q) \ell_\q)^i  \partial_s^i +\sum_{k=0}^{i-1}\sum_{j=0}^{k} \sum_{l=0}^i c_{j,k,l} \ell_\q^{j} \partial_s^k \partial_t^l \right),$$
that is
\begin{multline*}
\Delta_{t,\q}^{i+1}  =  (\partial_t(\q) \ell_\q)^{i+1}  \partial_s^{i+1} +\Delta_{t,\q}((\partial_t(\q) \ell_\q)^i)  \partial_s^i +(\partial_t(\q) \ell_\q)^i  \partial_t \partial_s^i +  \\ 
 \sum_{k=0}^{i-1}\sum_{j=0}^{k} \sum_{l=0}^i \Delta_{t,\q}(c_{j,k,l} \ell_\q^{j}) \partial_s^k \partial_t^l +\sum_{k=0}^{i-1}\sum_{j=0}^{k} \sum_{l=0}^i c_{j,k,l} \partial_t(\q) \ell_\q^{j+1} \partial_s^{k+1} \partial_t^l+ \sum_{k=0}^{i-1}\sum_{j=0}^{k} \sum_{l=0}^i c_{j,k,l} \ell_\q^{j} \partial_s^k \partial_t^{l+1}  .
\end{multline*} 
Note that the commutation of $\sigma_{\q}$ with $\Delta_{t,\q}$ implies that  $C_\q$ is stabilized by $\Delta_{t,\q}$.
Since by Remark~\ref{rmk:derivateqlog},  $\Delta_{t,\q}(\ell_\q)$ belongs to $C_\q$, we get that,  for any integer $j$, any $\tilde{c}\in C_{\q}$, we have ${\Delta_{t,\q}(\tilde{c}(\ell_\q)^{j})=\Delta_{t,\q}(\tilde{c})(\ell_\q)^{j}+ \tilde{c}c(\ell_\q)^{j-1}}$ where ${c =j \Delta_{t,\q}(\ell_\q) \in C_\q}$. Therefore, with ${\Delta_{t,\q}(\tilde{c})\in  C_\q}$, we find that $\Delta_{t,\q}(\tilde{c}(\ell_\q)^{j})\in C_{\q}[\ell_\q]$ is of degree at most $j$ in $\ell_\q$. With
$\partial_t(\q),c_{j,k,l}\in C_\q$, this ends the proof.
\end{proof}

From now on, let us fix $q\in C^*$ with $|q|\neq 1$, that is multiplicatively independent to $\q$,
 that is there are no $r,l\in \Z^{2}\setminus(0,0)$  such that $q^r=\q^l$.
 Remind that $C_{\q}.C_{q}\subset \mathcal{M}er(C^{*})$ is the compositum of fields and $\ell_{\q}\in \mathcal{M}er(C^{*})$ is a solution of $\s_{\q} (\ell_{\q})=\ell_{\q}+1$. We now give examples of difference differential fields for $\sigma_{\q},\partial_{s}$ and $\Delta_{t,\q}$.
 
\begin{lemma}\label{lemma:fielddefinitiongenus1}  The following statement	 hold.
\begin{enumerate}
\item The field $C_\q(s,\ell_\q)$ is stabilized by $\s_\q$, $\partial_s$ and $\Delta_{t,\q}$. The field $C_\q(s)$ is stabilized by $\s_\q$, and $\partial_s$. The field $C(s)$ is stabilized by $\partial_{s}$, $\partial_{t}$.
\item The field $C_{\q}.C_{q}(\ell_\q, \ell_{q})$ is stabilized by $\s_\q$, $\partial_s$ and $\Delta_{t,\q}$. The field $C_{\q}.C_{q}( \ell_{q})$ is stabilized by $\s_\q$, and $\partial_s$. The field $C_{q}(\ell_q)$ is stabilized by $\partial_{s}$, $\partial_{t}$.
\end{enumerate}
\end{lemma}

\begin{proof}
\begin{trivlist}
\item (1) Since $\s_\q (\ell_\q )=\ell_\q +1$, we easily see that  $C_\q(s,\ell_\q),C_\q(s)$ are stabilized by $\s_\q$. Since $\s_{\q}$ commutes with $\partial_{s}$ and $\Delta_{t,\q}$, the field $C_\q$ is stabilized by $\partial_{s}$ and $\Delta_{t,\q}$. It is now clear that $C_\q(s)$ is stabilized by $\partial_{s}$ and $\Delta_{t,\q}(C_\q(s))\subset C_\q(s,\ell_{\q})$.
By Remark \ref{rmk:derivateqlog}, $\Delta_{t,\q}(\ell_{\q}),\partial_{s}(\ell_{\q})\in C_{\q}$. Combining the lasts assertions, we obtain the result for $C_\q(s,\ell_\q)$.  Finally, the field $C(s)$ is stable by $\partial_{s}$, $\partial_{t}$, since $C$ is stable by $\partial_{s},\partial_{t}$, and $\partial_{s}(s)=s$, $\partial_{t}(s)=0$.
 
\item (2) Let us prove that $C_{q}(\ell_{q})$ is stabilized by $\s_\q$. Using $\s_{q}(\ell_{q})=\ell_{q}+1$ and the commutation between $\s_\q$ and $\s_{q}$, we find that 
$\s_\q (\ell_{q})-\ell_{q}\in C_{q}$. Similarly,  $\s_\q (C_{q})\subset C_{q}$, proving that $C_{q}(\ell_{q})$ is stabilized by $\s_\q$. Using $\partial_{s}(C_{\q})\subset C_{\q}$ and $\partial_s (\ell_q) \in C_q$, we find that the field $C_{\q}.C_{q}( \ell_{q})$ is stabilized by $\s_\q$ and $\partial_s$. \par 
Let us now consider the field $C_{\q}.C_{q}(\ell_\q, \ell_{q})$. The field $C_\q(\ell_\q)$  is clearly stable by $\s_\q$. From what preceede, $C_{q}(\ell_{q})$ is stable by $\s_\q$, and therefore, $C_{\q}.C_{q}(\ell_\q, \ell_{q})$ is stable by $\s_\q$. The same arguments than those used in $(1)$, prove that  ${\Delta_{t,\q} (C_{\q}(\ell_{\q})) \subset C_{\q}.C_{q}(\ell_\q)}$  and $\partial_{s}(C_{\q}(\ell_{\q})) \subset C_\q(\ell_\q)$.
It remains to prove that ${\Delta_{t,\q} (C_{q}(\ell_{q})) \subset C_{\q}.C_{q}(\ell_\q,\ell_{q})}$. We note that  ${\partial_t(\q)\ell_{\q}\partial_{s}+\partial_{t} =\Delta_{t,\q}=\Delta_{t,q} +(\partial_t(\q) \ell_\q - \partial_t(q) \ell_{q})\partial_s}$. Since $C_{q}$ is stabilized by $\Delta_{t,q}$ and $\partial_s$, we find that $\Delta_{t,\q}(C_{q})\subset C_{\q}.C_{q}(\ell_\q,\ell_{q}) $. Moreover, since $\partial_{s}(\ell_{q}),\Delta_{t,q}(\ell_{q})$
 belong to $C_{q}$, see Remark~\ref{rmk:derivateqlog}, we find that $\Delta_{t,\q}(\ell_{q}) \in C_{\q}.C_{q}(\ell_\q,\ell_{q})$. We have shown the inclusion ${\Delta_{t,\q} (C_{q}(\ell_{q})) \subset C_{\q}.C_{q}(\ell_\q,\ell_{q})}$. This concludes the proof for $C_{\q}.C_{q}(\ell_\q, \ell_{q})$.\par 
  Let us now consider $C_{q}(\ell_q)$. By Remark~\ref{rmk:derivateqlog} and $\partial_t = \Delta_{t,q}- \partial_t(q)\ell_q \partial_s$, we find that the inclusion holds ${\partial_{s}(\ell_q), \partial_{t}(\ell_q)\in C_q(\ell_q)}$. Since $\partial_{s},\Delta_{t,q}$ commute with $\sigma_{q}$, $C_{q}$ is stable by   $\partial_{s},\Delta_{t,q}$. With $\partial_{t} = \Delta_{t,q}- \partial_t(q)\ell_q \partial_s$, it follows that $\partial_t (C_{q})\subset C_{q}(\ell_q)$. Finally, we obtain that the field $C_q(\ell_q)$ is stable by $\partial_{s}, \partial_{t}$.
 \end{trivlist}
\end{proof}

\subsection{Difference Galois theory  for elliptic function fields}\label{sec3}
%%%%%%%%%%%%%%%%%%%%%%%%%%%%%%%%%%
In this section, we  apply the results of \S \ref{sec:differenceGaloistheory} to the specific cases of elliptic function fields  introduced in Lemma \ref{lemma:fielddefinitiongenus1}. We recall that  the following fields extensions   are $(\sigma,\partial,\Delta)$-fields extensions. \begin{itemize}
\item Let $\q\in C^*$ with $|\q|\neq 1$. Then, let us consider 
 $$(C_{\q}(s,\ell_{\q}),\sigma_{\q},\partial_s,\Delta_{t,\q} ) \subset (\cM er(C^*),\sigma_{\q},\partial_s,\Delta_{t,\q} ).$$
\item  Let $\q$ and $q$ two elements of $C^*$ such that $|q|,|\q|\neq 1$, that are \emph{multiplicatively independent}. Let us consider
$$(C_{\q}.C_{q}(\ell_{\q},\ell_{q}),\sigma_{\q},\partial_s,\Delta_{t,\q} ) \subset (\cM er(C^*),\sigma_{\q},\partial_s,\Delta_{t,\q} ).$$
\end{itemize}
In that framework,  the  criteria obtained in \S \ref{sec:differenceGaloistheory} to guaranty  the $(\partial_s,\Delta_{t,\q})$-differential transcendence of a solution  of a rank one  $\q$-difference equation can be simplified and some descent arguments prove that the existence of a telescoping relation involving the two derivatives implies the existence of a telescoping relations involving only the derivation $\partial_s$. More precisely, we find the following proposition:

\begin{prop}\label{prop2}
Let $K\subset \mathcal{M}er(C^{*})$ be a $(\sigma_{\q},\partial_{s})$-field and let us assume that 
\begin{trivlist}
\item \textbf{(H1)}\quad $L=K(\ell_{\q})$ is a $(\s_{\q},\partial_{s},\Delta_{t,\q})$-field; 
\item \textbf{(H2)}\quad $K^{\s_{\q}}=L^{\s_{\q}} =C_\q$ is relatively algebraically closed in $L$;
\item \textbf{(H3)}\quad $\ell_{\q}$ is transcendental over $K$.
\end{trivlist}
Let $f\in  \mathcal{M}er(C^{*})$,  that satisfies $\s_{\q}(f)=f+b$, for some $b$ that belongs to a subfield of $K$ stable by $\partial_{s},\partial_{t}$.

If $f$ is $(\partial_s,\Delta_{t,\q})$-differentially algebraic over $L$ then, there exist $m \in \N$, $d_0,\dots,d_m \in C_{\q}$ not all zero, and $h \in K$ such that 
$$
d_0 b_1+d_1\partial_s(b)+\dots +d_m \partial_s^m(b)=\s_{\q}(h)-h.
$$
\end{prop} 

\begin{proof}
Since $f$ is $(\partial_s,\Delta_{t,\q})$-differentially algebraic over $L$ and $K^{\sigma_{\q}}$ is relatively algebraically closed, Theorem \ref{thm:abstractdifftransGaloiscriteria} yields that there exist  $M \in \N$, $c_{i,j} \in L^{\s_{\q}}$ not all zero, and $g \in L$ such that
\begin{equation}\label{eq2}
\sum_{i,j \leq M} c_{i,j} \partial_s^i\Delta_{t,\q}^{j} (b)= \s_{\q}(g)-g.
\end{equation}

By Lemma \ref{lemma:decompiterateDeltat}, for all $i \in \N$, there exist $c_{j,k,l} \in C_\q$ such that 
\begin{equation}\label{eq:decompiterateDeltat}
\Delta_{t,\q}^i= (\partial_t(\q) \ell_\q)^i  \partial_s^i +\sum_{k=0}^{i-1}\sum_{j=0}^{k} \sum_{l=0}^i c_{j,k,l} \ell_\q^{j} \partial_s^k \partial_t^l.
\end{equation} 

The left hand side of \eqref{eq2} is a polynomial in $\ell_\q$ with coefficients in $K$. By Lemma~\ref{lem:transclogpolynomialinlogcocycle} with $\textbf{(H2)}$ and $\textbf{(H3)}$,  we find that  $g\in K[\ell_\q]$ as well.\par
Thus, let us write  $g=\sum_{k=0}^R \alpha_k \ell_\q^k$ with $\alpha_k \in K$ and $\alpha_{R} \neq 0$.
Let $$N=\max\{j\in \N |\exists i \hbox{ such that } c_{i,j} \neq 0\}.$$ 
By  \eqref{eq:decompiterateDeltat}, the coefficient of highest degree in $\ell_\q$ of the left hand side of \eqref{eq2} is  
\begin{equation}\label{eq30}
\left(\sum_{i \leq M} c_{i,N} (\partial_t(\q))^N \partial_s^{N+i}(b)\right)  \ell_\q^N.
\end{equation}
 
On the other hand, we have 
\begin{equation}\label{eq40}
\s_{\q}(g)-g =\ell_\q^{R}( \s_{\q}(\alpha_R)-\alpha_R)) + \ell_\q^{R-1}(\s_{\q}(\alpha_{R-1}) -\alpha_{R-1} +R\s_{\q}(\alpha_R))  + P(\ell_\q),
\end{equation} 
where $P(X) \in K[X]$ is a polynomial of degree strictly smaller than $R-1$. Then, comparing \eqref{eq30} and \eqref{eq40},  we find that  
 \begin{itemize}
\item either $R< N$ so that 
\begin{equation}\label{eq5}
\sum_{i \leq M} c_{i,N} (\partial_t(\q))^N \partial_s^{N+i}(b) = 0,
\end{equation}
\item  either $R=N$ so that \begin{equation}\label{eq6}
\sum_{i \leq M} c_{i,N} (\partial_t(\q))^N \partial_s^{N+i}(b) = \s_{\q}(\alpha_N)-\alpha_N,
\end{equation}
\item or $R > N$ so that $R>0$,  $0\neq \alpha_R \in L^{\s_{\q}}$. We claim that $R=N-1$. Indeed, $R>N-1$ implies $\s_{\q}(\alpha_R)=\s_{\q}(\alpha_R)$, $\s_{\q}(\alpha_{R-1}) -\alpha_{R-1} +R \alpha_R=0$ and then $\s_{\q}(\frac{\alpha_{R-1}}{\alpha_R}) - \frac{\alpha_{R-1}}{\alpha_R} +R =0$  with $\frac{\alpha_{R-1}}{\alpha_R} \in K$ in contradiction  with  Lemma~\ref{lem:transccriteriaforlog}  applied to  $f=\ell_{\q}$. Thus, we get $R=N-1$ and
\begin{equation}\label{eq7}
\sum_{i \leq M} \frac{c_{i,N}}{\alpha_R} (\partial_t(\q))^N \partial_s^{N+i}(b)
 = \s_{\q}\left(\frac{\alpha_{R-1}}{\alpha_R}\right)-\frac{\alpha_{R-1}}{\alpha_R} +R. \end{equation}
\end{itemize} 

For all these cases, note that there exists $i_0$ such that $c_{i_0,N} \neq 0$ by definition of $N$.  Since $\partial_s$ commutes with $\s_{\q}$, we can derive  \eqref{eq7} with respect to   $\partial_s$ and obtain that in any case, there exists $d_k \in L^{\s_{\q}}= C_{\q}$ not all zero and $h \in K$ such that 
 \begin{equation}\label{eqn:telescopedxfinal}
\sum_{k\leq M+1} d_{k}  \partial_s^{k}(b) = \s_{\q}(h)-h.
\end{equation}
\end{proof}

\subsection{Transcendence properties}\label{sec:trsnforellipticfunctions}

The goal of this subsection is to prove  some transcendence properties of the $\q$-logarithm in order to perform some descent procedure on telescopers. More precisely, we need to prove that the assumptions \textbf{(H1)} to \textbf{(H3)} of Proposition \ref{prop2} are satisfied for the fields $C_{\q}(s)$ and $C_{\q}.C_{q}(\ell_\q,\ell_q)$ for  $ \q$ and $q$  two multiplicatively independent elements of $C^{*}$ with $|q|\neq 1$, $|\q|\neq 1$.  We  recall  that $q$ and $\q$ are multiplicatively independent if  there are no $(r,l) \in \Z^{2}\setminus(0,0)$  such that $q^r=\q^l$. Remind that $C_{\q}.C_{q}\subset \mathcal{M}er(C^{*})$ is the compositum of fields and $\ell_{\q}\in \mathcal{M}er(C^{*})$ is a solution of $\s_{\q} (y)=y+1$. With Lemma \ref{lemma:fielddefinitiongenus1},
 \textbf{(H1)} of Proposition \ref{prop2} is satisfied for $K=C_{\q}(s)$ and $K=C_{\q}.C_{q}(\ell_{q})$.
 
\begin{lemma}\label{lemma:tqconstantCq}
Any element  in a $\s_\q$-extension of   $C_q$\footnote{We recall that since  $\s_\q$ and $\s_q$ commute, the field $C_q$ is a $\s_\q$-field.} that is algebraic over $C_q$ and 
invariant by $\s_{\q}$  is in $C$.    Any element  in a $\s_q$-extension of   $C_\q$ that is algebraic over $C_\q$ and 
invariant by $\s_{q}$  is in $C$. \end{lemma}

\begin{proof}
The two statements are symmetrical, so let us only prove the first one. First let us prove that $C_q \cap C_\q=C$. Let $f$ be an element of $C_q$ that is $\s_{\q}$-invariant.
Suppose to the contrary that $f$ is nonconstant. Then $f$ has a nonzero pole $c$.
 Since $\s_{\q}(f)=f$, the multiplication by $\q$ induces a permutation of the poles of $f$ modulo $q$. Since the set of poles modulo $q$ is a finite set,  there exists $m \in\N$   such that $\q^m c =q^dc$ for some $d \in \Z$. A contradiction with the fact that $q$ and $\q$ are multiplicatively independent. 
 Now, let $f$ be  in a $\s_\q$-extension of   $C_q$  algebraic over $C_q$ and 
invariant by $\s_{\q}$. Let $\mu(X) \in C_q[X]$ be the monic minimal polynomial of $f$ above $C_q$. Since $\s_\q(f)=f$, we easily see that the coefficients of $\mu$ must be fixed by $\s_\q$. Then, 
these coefficients belong to $C_q\cap C_\q$, which is equal to $C$. Then, $f$ is algebraic over $C$. The latter field being algebraically closed, we conclude that $f \in C$.
 \end{proof}

\begin{lemma}\label{lemma:transcendanceofqlog1}
The following statements hold:
\begin{enumerate}
\item  the fields $C_\q$ and $C_{q}$ are linearly disjoint over $C$;
\item for all  $\alpha \in C_{\q}.C_{q}$, $\s_{q}(\alpha) \neq \alpha +1$ and 
$\s_\q(\alpha) \neq \alpha +1$;
\item for all  $\alpha \in C_\q(s)$,  
$\s_\q(\alpha) \neq \alpha +1$.
\end{enumerate}

\end{lemma}
\begin{proof}
\begin{trivlist}
\item (1) This is Lemmas   \ref{lemma:tqconstantCq} and \ref{lem:lineardisjonctionoverconstant}  with $K=C$, $M=C_q$ and 
$L=C_\q$, $\sigma=\sigma_\q$.

\item (2) Suppose to the contrary  that there exists $\alpha \in C_{\q}.C_{q}$, such that $\s_{q}(\alpha) = \alpha +1$. Since 
$C_{q}$ is by Proposition \ref{prop:Tatecurve}, the field of meromorphic functions over a Tate curve, there exist $x,y  \in C_{q}$ such that $x$ is transcendental over $C$, $y$ algebraic of degree $2$ over $C(x)$ and $C_{q}=C(x,y)$. Since $C_\q$ is linearly disjoint from $C_{q}$ over $C$, the field  $C_{\q}.C_{q}$ equals $C_\q(x,y)$ and there are ${P(X),Q(X) \in C_\q(X)}$ such that $\alpha =P(x)y+Q(x)$. Since $x,y$ are fixed by $\s_{q}$ and $y$ is of degree $2$ over $C_{\q}(x)$, we deduce from $\s_{q}(\alpha)=\alpha +1$ that
$P^{\s_{q}}(x)=P(x)$ and
 $Q^{\s_{q}}(x)-Q(x)=1$ where $P^{\s_{q}}(X)$ (resp.  $Q^{\s_{q}}(X)$) denotes the fraction obtained from $P(X)$ (resp. $Q(X)$) by applying $\s_{q}$ to the coefficients.
Let  $\overline{C_{\q}}$  be some algebraic closure of $C_{\q}$. We endow $\overline{C_{\q}}$ with a structure of $\s_q$-field extension of $C_\q$.  Let us  write  $Q(X)=\frac{c_r}{X^r}+\dots+\frac{c_1}{X} +R(X)$ with $R \in \overline{C_{\q}}(X)$ with no pole at $X=0$. Then, since $x$ is   transcendental over $\overline{C_{\q}}$ and fixed by $\s_q$
 $$Q^{\s_{q}}(x)-Q(x)=1=\frac{\s_q(c_r)-c_r}{x^r}+\dots+\frac{\s_q(c_1)-c_1}{x} +R^{\s_q}(x)-R(x).$$
Using the transcendence of $x$ over $\overline{C_{\q}}$, we find that $1=\s_q(\tilde{\beta})-\tilde{\beta}$ for $\tilde{\beta}=R(0) \in \overline{C_{\q}}$.  There exists a unique derivation extending  $\partial_s$ to $ \overline{C_{\q}}$ and this derivation commutes with $\s_q$. Denoting this derivation by $\partial_s$ and  
 deriving  $1=\s_q(\tilde{\beta})-\tilde{\beta}$, we conclude  that
 $\partial_s (\tilde{\beta}) \in C_q \cap C_{\q^r}$. Note that $q$ and $\q^{r}$ are multiplicatively independent. By Lemma \ref{lemma:tqconstantCq}, we  find that  $\partial_s(\tilde{\beta}) \in C$ which leads to $\tilde{\beta}=cs+d$ for some $c,d \in C$. A contradiction with $1=\s_q(\tilde{\beta})-\tilde{\beta}$.
The proof for $\q$ is similar. 

\item (3)  Let  $\alpha \in C_\q(s)$. Using the partial fraction decomposition of $\alpha$ in $\overline{C_{\q}}(s)$, the fact that ${\s_\q(s)=\q s}$ and  the transcendence of $s$ over $C_\q$, one can easily see that $\s_\q(\alpha)-\alpha \neq 1$.
 \end{trivlist}
\end{proof}

\begin{lemma}\label{lemma:transcendenceqlogfinal}
The following statements hold:
\begin{enumerate}
\item the function $\ell_{\q}$ (resp. $\ell_{q}$) is transcendental over $C_{\q}.C_{q}$;
\item  the function $\ell_{\q}$
is transcendental over $C_{\q}(s)$. In particular, \textbf{(H3)} of Proposition \ref{prop2} is satisfied for $K=C_{\q}(s)$.
\end{enumerate}

\end{lemma}
\begin{proof}
\begin{trivlist}
\item (1)
Since $\s_{\q}(\ell_\q) =\ell_\q +1$ and $C_\q\subset(C_{\q}.C_{q})^{\s_\q}\subset \cM er(C^*)^{\s_\q}= C_\q$, we can apply Lemma~\ref{lem:transccriteriaforlog} and find that $\ell_\q$ is algebraic over $C_{\q}.C_{q}$ if and only if there exists $\alpha \in C_{\q}.C_{q}$ such that
$\s_\q(\alpha) =\alpha +1$. We conclude by Lemma \ref{lemma:transcendanceofqlog1}. The proof for $\ell_{q}$ is symmetrical. 
\item (2)
Since $\s_{\q}(\ell_\q) =\ell_\q +1$ and $C_\q\subset (C_\q(s))^{\s_\q}\subset \cM er(C^*)^{\s_\q}=C_\q$, we can apply Lemma~\ref{lem:transccriteriaforlog} and find that $\ell_\q$ is algebraic over $C_\q(s)$ if and only if there exists $\alpha \in C_\q(s)$ such that
$\s_\q(\alpha)=\alpha +1$. We again conclude by Lemma \ref{lemma:transcendanceofqlog1}.
\end{trivlist}
\end{proof}

\begin{lemma}\label{lemma:descentcocycles}
 The following statement  hold:
 \begin{enumerate}
\item  let $f \in C_{q}$. If  there exists $\alpha \in C_{\q}.C_{q}$  satisfying $\s_{\q}(\alpha)-\alpha= f$, then there exists $\beta \in C_{q}$ such that 
$\s_{\q}(\beta)-\beta =f$; 
 \item  let $f \in C_{\q}.C_{q}$. If there exists $\alpha \in C_{\q}.C_{q}(\ell_{q})$  satisfying $\s_{\q}(\alpha)-\alpha= f$, then, there exist $\tilde{a} \in C_\q, \tilde{b} \in C_{\q}.C_{q}$ such that 
$\s_{\q}(\tilde{a} \ell_{q} +\tilde{b})- (\tilde{a} \ell_{q} +\tilde{b}) =f$.

\end{enumerate}  

\end{lemma}
\begin{proof}
\begin{trivlist} \item (1)
 Analogously to the proof of Lemma \ref{lemma:transcendanceofqlog1}, let us write $\alpha =P(x)y+Q(x)$ for $P(X),Q(X) \in C_{q}(X)$ and $C_{\q}=C(x,y)$.  Reasoning as in the proof of Lemma \ref{lemma:transcendanceofqlog1},  we find that $Q^{\s_{\q}}(x)-Q(x)=f$. Since $x$ is transcendental over $C_{q}$, we conclude as in Lemma \ref{lemma:transcendanceofqlog1} that  there is $\tilde{\beta} \in \overline{C_{q}}$, for some $\overline{C_{q}}$ algebraic closure of $C_q$ such that   $\s_{\q}(\tilde{\beta})-\tilde{\beta} =f$. Since by Lemma~\ref{lemma:tqconstantCq}, $\overline{C_{q}}^{\sigma_\q}= C_q^{\sigma_\q}=C$, Lemma \ref{lem:transccriteriaforlog} implies that there exists  $\beta \in C_q$ such that 
$\s_\q(\beta)-\beta=f$.
\item (2) First of all, let us note that since $\s_\q$ and $\s_{q}$ commute, there exists $d \in C_{q}$ such that 
\begin{equation}\label{eq:qeqell_q}
\s_\q(\ell_{q})=\ell_{q}+d.
\end{equation}
By Lemma \ref{lemma:transcendenceqlogfinal}, the function $\ell_{q}$ is transcendental over $C_{\q}.C_{q}$. This implies that $\ell_{q}\notin C_{\q}$ and then $d\neq 0$.  Since $C_{\q}.C_{q}(\ell_{q})^{\s_\q}=C_\q =\mathcal{M}er (C^{*})^{\s_\q}=C_{\q}.C_{q}^{\s_\q}=C_\q$,  Lemma \ref{lem:transclogpolynomialinlogcocycle}, applied to $\s_\q(\ell_q)=\ell_q +d$, implies that there exists $P \in C_{\q}.C_{q}[X]$ such that 
$$
f= \s_\q( P(\ell_{q})) -P(\ell_{q}).
$$
Now, let us write $P(X)=\sum_{k=0}^N a_k X^k$ with $a_k \in C_{\q}.C_{q}$,  and $N$ minimal. We find 
\begin{multline}\label{eq:polynomialcocycle}
f =(\s_\q(a_N)-a_N) \ell_{q}^N+ (\s_\q(a_{N-1})-a_{N-1} +Nd \s_\q(a_N))\ell_{q}^{N-1}+ \\
\mbox{terms of order less than } N-1.
\end{multline}

 We conclude in view of \eqref{eq:polynomialcocycle} that if $N =0$ we are done by setting $\tilde{a}=0$ and ${\tilde{b}=a_N}$. Let us now assume that $N>0$. Then, by minimality of $N$, ${\s_\q(a_N)=a_N}$. We claim that ${\s_\q(a_{N-1})-a_{N-1} +Nd \s_\q(a_N) =\s_\q(a_{N-1})-a_{N-1} +Nd a_N \neq 0}$. To the contrary, ${\s_\q(a_{N-1})=a_{N-1} -Nd a_N}$ implies $\s_\q( \frac{a_{N-1}}{a_N} +N \ell_{q})=\frac{a_{N-1}}{a_N} +N \ell_{q}$ and ${\frac{a_{N-1}}{a_N} +N \ell_{q} \in C_\q}$, contradicting the transcendence of $\ell_{q}$ over $C_{\q}.C_{q}$, see Lemma \ref{lemma:transcendenceqlogfinal}. This proves the claim. If $N>1$, then \eqref{eq:polynomialcocycle} with $\s_\q(a_N)=a_N$ and
 $\s_\q(a_{N-1})-a_{N-1} +Nd a_N \neq 0$, would give an equation of order $N-1$ which would contradicts the transcendence of $\ell_{q}$ over $C_{\q}.C_{q}$. This proves that $N=1$ and $f = \s_\q(a_1 \ell_{q} +a_0) - (a_1 \ell_{q} +a_0)$ for some $a_1 \in C_\q,a_0\in C_{\q}.C_{q}$.
\end{trivlist}
 \end{proof}

\begin{lemma}\label{lemma:algindependceqlogtqlog}
The function $\ell_{\q}$ is transcendental over $C_{\q}.C_{q}(\ell_{q})$. In particular, the assumption \textbf{(H3)} of Proposition \ref{prop2} holds for  $K=C_{\q}.C_{q}(\ell_{q})$.
\end{lemma}
\begin{proof}
By Lemma \ref{lem:transccriteriaforlog}, the function $\ell_\q$ is algebraic over $C_{\q}.C_{q}(\ell_{q})$ if and only if we have ${\ell_\q \in C_{\q}.C_{q}(\ell_{q})}$.  Suppose to the contrary that ${\ell_\q \in C_{\q}.C_{q}(\ell_{q})}$. Since $1 =\s_\q(\ell_\q)-\ell_\q$, we conclude by Lemma \ref{lemma:descentcocycles} that there exist $\tilde{a} \in C_\q, \tilde{b} \in C_{\q}.C_{q}$ such that $1 =\s_\q(\tilde{a} \ell_{q} +\tilde{b})- (\tilde{a} \ell_{q} +\tilde{b})$. Combining this equation with $\s_\q(\ell_\q)-\ell_\q =1$, we find that $\s_\q(\ell_\q)-\ell_\q=\s_\q(\tilde{a} \ell_{q} +\tilde{b})- (\tilde{a} \ell_{q} +\tilde{b})$, proving that $\s_\q(\tilde{a} \ell_{q} +\tilde{b}-\ell_\q)=\tilde{a} \ell_{q} +\tilde{b}-\ell_\q\in C_\q$. Then, there exists $\widetilde{b_1} \in C_{\q}.C_{q}$ such that 
\begin{equation}\label{eq:relationlqandltq}
\ell_\q =\widetilde{a} \ell_{q} +\widetilde{b_1}.
\end{equation}
Deriving \eqref{eq:relationlqandltq} with respect to $\partial_s$, we find
$$
\partial_s(\ell_\q) =\partial_s(\widetilde{a} ) \ell_{q}+ \widetilde{a}  \partial_s(\ell_{q}) +\partial_s(\widetilde{b_1} ).
$$
By Remark \ref{rmk:derivateqlog}, $\partial_s(\ell_\q), \partial_s(\ell_{q})\in C_{\q}.C_{q}$. In virtue of the commutation between $\partial_{s}$ and $\s_\q,\s_q$, the fields $C_{q},C_{\q}$ are stabilized by $\partial_{s}$, which implies $\partial_s(\widetilde{a} ),  \partial_s(\widetilde{b_1} ) \in C_{\q}.C_{q}$. By Lemma \ref{lemma:transcendenceqlogfinal}, the function  $\ell_{q}$ is transcendental over the latter field, we conclude that $\partial_s(\widetilde{a})=0$ and therefore $\widetilde{a} \in C$. In particular it belongs to $C_{\q}$ and $C_{q}$. Using $1 =\s_\q(\tilde{a} \ell_{q} +\tilde{b})- (\tilde{a} \ell_{q} +\tilde{b})$, we find 
$$
1-\widetilde{a} d =\s_\q(\widetilde{b})-\widetilde{b},$$
where $d =\s_\q(\ell_q)-\ell_q \in C_q$, see \eqref{eq:qeqell_q}.
Since $1-\widetilde{a} d\in C_{q}$, we conclude by Lemma \ref{lemma:descentcocycles}, that there exists $\widetilde{b_2} \in C_{q}$ such that $1-\widetilde{a} d =\s_\q(\widetilde{b_2})-\widetilde{b_2}$. Replacing the left hand side gives $$\s_\q(\ell_{\q})-\ell_{\q}-
\s_\q(\widetilde{a}\ell_{q})+\widetilde{a}\ell_{q}=\s_\q(\widetilde{b_2})-\widetilde{b_2}.$$

This shows that $\ell_{\q}-
\widetilde{a}\ell_{q}-\widetilde{b_2}\in C_{\q}$ and then, there exists $c \in C_\q$ such that $\ell_\q + c= \widetilde{a} \ell_{q} +\widetilde{b_2}$. Deriving this equation with respect to $\partial_s$, we find (we use  
$\partial_{s}(\widetilde{a})=0$)
$$\partial_s(\ell_\q) +\partial_s (c) = \widetilde{a}\partial_s(\ell_{q})+ \partial_s(\widetilde{b_2}).$$ 
By Remark \ref{rmk:derivateqlog}, the left hand side  of the equation belongs to $C_\q$ whereas the right hand side is in $C_{q}$. By Lemma \ref{lemma:tqconstantCq}, we conclude that $\partial_s(\ell_\q +c) \in C$. This means that  there exist $a_0, b_0 \in C$ such that $\ell_\q =a_0 s+ b_0 -c$ in contradiction with $\ell_\q$ transcendental over $C_{\q}(s)$, see Lemma~\ref{lemma:transcendenceqlogfinal}.
\end{proof}

We can now prove  that our fields satisfy the assumption \textbf{(H2)} of Proposition \ref{prop2}.

\begin{lemma}\label{lemma:relativealgebraicclosure}
The following holds:
\begin{enumerate}
\item$C_{\q}$ is relatively algebraically closed in $C_{\q}  (s, \ell_\q)$;
\item $C_{\q}$ is relatively algebraically closed in $C_{\q}.C_{q} ( \ell_\q,\ell_{q})$.
\end{enumerate}
In particular, \textbf{(H2)} of Proposition \ref{prop2} holds for $K=C_{\q}(s)$ and $K=C_{\q}.C_{q}(\ell_{q})$.\end{lemma}

\begin{proof}
\begin{trivlist}
\item (1)
The first point is a consequence of transcendence of $s$ over $C_{\q}$, and the transcendence of $\ell_{\q}$
 over $C_{\q}(s)$, see Lemma \ref{lemma:transcendenceqlogfinal}.

\item (2) Let us prove the second point.  Let us start by proving that $C_{\q}$ is relatively  algebraically closed in $C_{\q}.C_{q}$. As in the proof of Lemma~\ref{lemma:transcendanceofqlog1}, we have $C_{\q}= C(x,y)$ and $C_{\q}.C_{q} = C_q(x,y)$  where  $y$ is of degree $2$ over both $C(x)$ and $C_q(x)$. Let $f \in C_q(x,y)$. Then $f = P(x) y + Q(x)$ with $P(x), Q(x) \in C_q(x)$. 
If $f$ is algebraic over $C_{\q}$ then Lemma~\ref{lemma:algoverdifferenceconstants} implies that $\s_{\q}^r(f) = f$ for some $r \in \ZX^*$ and therefore $\s_{\q}^r(P(x)) = P(x)$ and $\s_{\q}^r(Q(x)) = Q(x)$. We claim that $P(x)$ and $Q(x)$ are in $C(x)$, and therefore that $f \in C_{\q}$. Let $P(x) = P_1(x)/P_2(x)$ where $P_1(x), P_2(x) \in C_{q}[x]$ are relatively prime and $P_1(x)$ is monic.  We then have that $\s_{\q}^r(P_1(x))P_2(x)  =\s_{\q}^r(P_2(x)) P_1(x)$ and consequently $P_1(x) $ divides $\s_{\q}^r(P_1(x))$ (resp. $\s_{\q}^r(P_1(x))$ divides $P_1(x)$). Since $P_1(x)$ is monic,  $P_1(x) = \s_{\q}^r(P_1(x))$ and $P_2(x) = \s_{\q}^r(P_2(x))$. This implies that the coefficients of $P_1(x)$ and $P_2(x)$ are left fixed by $\s_{\q}^r$ . Note that by assumption, $q$ and $\q^{r}$ are multiplicatively independent. Therefore, by Lemma~\ref{lemma:tqconstantCq}, applied with $\q$ replaced by $\q^{r}$, $P_1, P_2 \in C[X]$. The proof for $Q$ is similar. This proves our claim and show that $f\in C_{\q}$. Then $C_{\q}$ is relatively  algebraically closed in $C_{\q}.C_{q}$.\par 

Note that Lemma~\ref{lemma:transcendenceqlogfinal} implies that $\ell_{\q}$ is transcendental over $C_{\q}.C_{q}$  and Lemma~\ref{lemma:algindependceqlogtqlog} implies that $\ell_q$ is transcendental over $C_{\q}.C_{q}(\ell_{\q})$.  Therefore $C_{\q}$ is relatively algebraically closed in $C_{\q}.C_{q} (\ell_{\q}, \ell_q)$.
\end{trivlist}
\end{proof}

Finally, we prove a lemma that will allows us to descend some \emph{telescoping relations} on smaller base fields.

\begin{lemma}\label{Lemma:telescoperdescentgenus1}
Let $b \in C_q$ such that there exist $N\in \N$, $c_{i} \in C_{\q}$ with $c_N \neq 0$, and ${g \in C_{\q}.C_{q} (\ell_\q,\ell_q)}$ that satisfy
\begin{equation}\label{eq:telescoperbigfield}
\sum_{i=0}^N c_{i} \partial_s^{i}(b)=\s_\q(g)-g.
\end{equation}
Then, there exist $m \in \N$, $d_0,\dots,d_m \in C$ not all zero and $h \in C_{q}$ such that 
$$
d_0 b_2+d_1\partial_s(b_2)+\dots +d_m \partial_s^m(b_2)=\s_\q(h)-h.
$$
\end{lemma}
\begin{proof}
 First of all note that the left hand side of \eqref{eq:telescoperbigfield} belongs to $C_{\q}.C_{q}$.
By Lemma \ref{lemma:algindependceqlogtqlog}, the function $\ell_\q$ is transcendental over 
$C_{\q}.C_{q}(\ell_{q})$. 
By Lemma \ref{lem:transclogpolynomialinlogcocycle},  $g \in C_{\q}.C_{q}(\ell_{q})[\ell_\q]$. So let us write   $g=\sum_{k=0}^R \alpha_k \ell_\q^k$ with $\alpha_k \in C_{\q}.C_{q}(\ell_{q}) $, $\alpha_{R} \neq 0$.

\textbf{Claim.} There exist $m \in \N$, $c'_k \in C_{\q}$, $c'_m \neq 0$, and $\alpha \in C_{\q}.C_{q}(\ell_{q})$ such that 
 \begin{equation}\label{eqn:telescopedxgenus1}
\sum_{k=0}^m c'_{k}  \partial_s^{k}(b) = \s_\q(\alpha)-\alpha.
\end{equation}
If $R=0$ the claim is proved. Assume that $R>0$.
Then, we have 
\begin{equation}\label{eqn:coeffrighthandside2}
 \s_\q(g)-g =\ell_\q^{R}( \s_\q(\alpha_R)-\alpha_R)) + \ell_\q^{R-1}(\s_\q(\alpha_{R-1}) -\alpha_{R-1} +R \alpha_R)  + P(\ell_\q),
\end{equation} 
where $P(X) \in C_{\q}.C_{q}(\ell_{q})[X]$ is a polynomial of degree smaller than $R-1$.
Then, comparing \eqref{eqn:coeffrighthandside2} and \eqref{eq:telescoperbigfield},  we find, by transcendence of $\ell_\q$ over $C_{\q}.C_{q}(\ell_{q})$, see Lemma  \ref{lemma:algindependceqlogtqlog}, that $\s_\q(\alpha_R)=\alpha_R$. Let us prove that $\s_\q(\alpha_{R-1}) -\alpha_{R-1} +R \alpha_R\neq 0$. Indeed if $\s_\q(\alpha_{R-1}) -\alpha_{R-1} +R \alpha_R=0$ then $\s_\q(\frac{\alpha_{R-1}}{\alpha_R}) - \frac{\alpha_{R-1}}{\alpha_R} +R =0$  with $\frac{\alpha_{R-1}}{\alpha_R} \in C_{\q}.C_{q}$ in contradiction  with  Lemma  \ref{lemma:transcendenceqlogfinal} and Lemma~\ref{lem:transccriteriaforlog}. We then obtain that $R=1$ since otherwise we would deduce from \eqref{eqn:coeffrighthandside2} an algebraic relation for $\ell_\q$ over $C_{\q}.C_{q}(\ell_{q})$, contradicting Lemma \ref{lemma:algindependceqlogtqlog}. Thus, 
\begin{equation}\label{eqn:telescopedxgenus1caseb}
 \sum_{i=0}^{N}\frac{c_{i}}{\alpha_1} \partial_s^{i}(b)  = \s_\q\left(\frac{\alpha_{0}}{\alpha_1}\right)-\frac{\alpha_{0}}{\alpha_1} +1.
 \end{equation}
Remind that $\alpha_1 \in C_{\q}$ and the latter field is stable by $\partial_{s}$ due to the commutation between $\partial_{s}$ and $\s_\q$.  By Lemma \ref{lemma:fielddefinitiongenus1}, the field $C_{\q}.C_{q}(\ell_{q})$ is stabilized by $\partial_s$. We can  derive \eqref{eqn:telescopedxgenus1caseb} with respect to  $\partial_s$ and using the commutation between $\s_\q$ and $\partial_{s}$, we obtain our claim.

\textbf{Claim.} There exist 
$M \in \N$, $d_k \in C_{\q}$, $d_M \neq 0$ and $\beta \in C_{\q}.C_{q}$ such that 
$$
\sum_{k=0}^M d_{k}  \partial_s^{k}(b) = \s_\q(\beta)-\beta.
$$
Indeed, by Lemma \ref{lemma:descentcocycles}, we can find $a \in C_\q, b \in C_{\q}.C_{q}$ such that 
\begin{equation}\label{eq:telescoper2genus1}
\sum_{k=0}^m c'_{k}  \partial_s^{k}(b) = \s_\q(a\ell_{q} +b)-(a \ell_{q} +b).
\end{equation}
Either $a =0$ and $\sum_{k} c'_{k}  \partial_s^{k}(b) = \s_\q(b)-(b)$ for some $b \in C_{\q}.C_{q}$. Or $a \neq 0$ and dividing \eqref{eq:telescoper2genus1} by $a$ and deriving with respect to $\partial_s$, we find
$$
\sum_{k=0}^{m+1} d_{k}  \partial_s^{k}(b) = \s_\q(\partial_s(\ell_{q})  +\partial_s(b/a))-(\partial_s( \ell_{q}) +\partial_s(b/a)),
$$
where the $d_k$ are in $C_\q$, $d_{m+1}=\frac{c'_m}{a} \neq 0$. Furthermore, by Remark \ref{rmk:derivateqlog} and the fact that $C_{\q}$, $C_{q}$, are stable by $\partial_{s}$, we find
 $\partial_s(\ell_{q})+\partial_s(b/a) \in C_{\q}.C_{q}$. This proves the claim.
 
Now, let us consider an equation of the form
$$
\sum_{k=0}^{M} d_{k}  \partial_s^{k}(b) = \s_{\q}(\beta)-\beta,
$$
 with $\beta \in C_{\q}.C_{q}$, $d_k \in C_\q$ and $d_M \neq 0$, minimal with respect to the maximal  order of derivation $M$ of $b$.  We can write this minimal equation as follows
$$
d_M\partial_s^{M}(b)+ \sum_{k=0}^{M-1} d_{k}  \partial_s^{k}(b) = \s_\q(\beta)-\beta,
$$
with $d_M \in C_\q^*$. Then dividing by $d_M$, we find
$$
\partial_s^{M}(b)+ \sum_{k=0}^{M-1} \frac{d_{k}}{d_M}  \partial_s^{k}(b) = \s_\q \left(\frac{\beta}{d_M}\right)-\frac{\beta}{d_M}.
$$
Therefore, we can without loss of assumption assume that $d_M=1$. Now, if we compute the element $\s_{q}(\s_\q(\beta)-\beta))-(\s_\q(\beta)-\beta))$ and use the fact that $b \in C_{q}$, we find
$$
\sum_{k=0}^{M-1}(\s_{q}( d_{k})- d_k)  \partial_s^{k}(b) = \s_\q(  \s_{q}( \beta)-\beta)-( \s_{q}( \beta)-\beta).
$$
By minimality, we find that, for all $k$, the element $d_k \in C_{\q}$ is fixed by $\s_{q}$. This means that $d_k \in C$ by Lemma \ref{lemma:tqconstantCq}.

Since $\partial_s^{M}(b)+ \sum_{k=0}^{M-1} d_{k}  \partial_s^{k}(b) \in C_{q}$ and 
$\partial_s^{M}(b)+ \sum_{k=0}^{M-1} d_{k}  \partial_s^{k}(b)=\s_\q(\beta)-\beta$ with $\beta \in C_{\q}.C_{q}$, Lemma   \ref{lemma:descentcocycles}  shows that we have the existence of $h \in C_{q}$ such that 
 $$\partial_s^{M}(b)+ \sum_{k=0}^{M-1} d_{k}  \partial_s^{k}(b)=\s_\q(h)-h. $$
\end{proof}

%%%%%%%%%%%%%%%%%%%%%%%%%%%%%%%%%%%%%%
%%%%%%%%%%%%%%%%%%%%%%%%%%%%%%%%%%%

The   results of Appendix \ref{sec:trsnforellipticfunctions}  are summarized in  the following crucial corollary.

\begin{cor}\label{cor1}
The assumptions  of Proposition \ref{prop2} are satisfied for 
\begin{itemize}
\item Genus zero case: $K=C_{\q}(s)$ and $b \in C(s)$ with $\q\in C^*$ such that $|\q|\neq 1$;
\item Genus one case: $K=C_{\q}.C_{q}(\ell_{q})$ and $b \in C_{q}(\ell_q)$ with $\q,q\in C^*$ such that ${|\q|,|q|\neq 1}$ and $\q$ and $q$ are multiplicatively independent.
\end{itemize}
\end{cor}

\begin{proof}
The fact that the field $K$ and $b$ satisfy the assumptions \textbf{(Hi)}  is  Lemmas \ref{lemma:fielddefinitiongenus1},  \ref{lemma:transcendenceqlogfinal}, \ref{lemma:algindependceqlogtqlog}, and \ref{lemma:relativealgebraicclosure}.
\end{proof}

\end{appendix}

\bibliographystyle{alpha}
\bibliography{walkbib, qG}

\end{document}